\documentclass[a4paper, intlimits, 10pt,reqno]{amsart}
\usepackage[T1]{fontenc}
\usepackage{float}

\usepackage{times}
\usepackage[english]{babel}
\usepackage[T1]{fontenc}
\usepackage{enumitem}
\usepackage[colorlinks=true,urlcolor=blue, citecolor=blue,linkcolor=blue,
linktocpage,pdfpagelabels, bookmarksnumbered,bookmarksopen]{hyperref}
\usepackage{graphicx}

\usepackage{subcaption}
\usepackage{tabularx}
\usepackage{color}
\usepackage{eurosym}
\usepackage[numbers, sort&compress]{natbib}

\usepackage[normalem]{ulem}
\usepackage{multirow}
\usepackage{indentfirst}
\usepackage{array}
\usepackage{accents}

\setlist[enumerate]{topsep=2pt,itemsep=3pt,parsep=1pt} 

\usepackage{tocbasic}
\makeatletter
\def\@tocline#1#2#3#4#5#6#7{\relax
  \ifnum #1>\c@tocdepth 
  \else
    \par \addpenalty\@secpenalty\addvspace{#2}%
    \begingroup \hyphenpenalty\@M
    \@ifempty{#4}{%
      \@tempdima\csname r@tocindent\number#1\endcsname\relax
    }{%
      \@tempdima#4\relax
    }%
    \parindent\z@ \leftskip#3\relax \advance\leftskip\@tempdima\relax
    \rightskip\@pnumwidth plus4em \parfillskip-\@pnumwidth
    #5\leavevmode\hskip-\@tempdima
      \ifcase #1
       \or\or \hskip 1em \or \hskip 2em \else \hskip 3em \fi%
      #6\nobreak\relax
    \hfill\hbox to\@pnumwidth{\@tocpagenum{#7}}\par
    \nobreak
    \endgroup
  \fi}
\makeatother

\usepackage{amsthm, amssymb, amsfonts}
\usepackage{amsmath}

\theoremstyle{plain}
\newtheorem{theorem}{Theorem}[section]
\newtheorem{corollary}[theorem]{Corollary}

\newtheorem{lemma}[theorem]{Lemma}
\newtheorem{proposition}[theorem]{Proposition}
\newtheorem{example}[theorem]{Example}

\newtheorem{remark}[theorem]{Remark}

\newtheorem{question}[theorem]{Question}
\numberwithin{equation}{section}
\theoremstyle{definition}
\newtheorem{definition}[theorem]{Definition}

\begin{document}

\title{On the Structure of Busemann Spaces with Non-Negative Curvature I}

\date{\today}

\thanks{}

\author{Bang-Xian Han}

\author{Liming Yin}

\begin{abstract}
    We extend the structure theory of Burago--Gromov--Perelman for Alexandrov spaces with curvature bounded below, to the setting of Busemann spaces with non-negative curvature. 
    We prove that any finite-dimensional Busemann space with non-negative curvature, satisfying Ohta's $S$-concavity and local semi-convexity, admits a non-trivial Hausdorff measure of integer dimension, satisfies the measure contraction property and contains an open dense topological manifold part with full measure.
    We also show that such spaces are rectifiable and almost every point admits a unique tangent cone isometric to a finite-dimensional Banach space. 
    Finally, based on a refined analysis on strainers and asymmetric angles, we obtain a Hausdorff dimension estimate for singular strata, and show that such spaces admit a natural measure-theoretic stratification.
    Our results enrich the theory of synthetic sectional curvature lower bounds for metric spaces, and provide some useful tools and examples to study Finslerian metric measure spaces.  
    \medskip

    \noindent\textbf{Keywords:} Busemann concave spaces, non-negative curvature, non-Riemannian metric spaces, rectifiability, singular sets, unique Banach tangent cones 
\end{abstract}

\maketitle
\tableofcontents

\section{Introduction}
\subsection{Motivation and object}
The study of how curvature shapes the geometry of spaces is a central theme in both mathematics and physics. 
While classical differential geometry--particularly Riemannian geometry--has provided deep insights into the interplay between curvature tensor and large-scale geometry of smooth spaces, it is now widely recognized that synthetic notions of curvature bound are powerful tools for understanding the structure of non-smooth spaces.    
Such non-smooth metric (measure) spaces naturally arise as (measured) Gromov--Hausdorff limits, quotients, gluing, and suspensions of smooth manifolds.

Among various synthetic notions of curvature bound, those defined via comparison of geodesic triangles with their counterparts in model spaces are considered as `generalized' sectional curvature bounds.
More precisely, if geodesic triangles in a geodesic space are `fatter' or `thinner' than their comparison triangles in space forms, we say that the space has curvature bounded below (CBB spaces) or curvature bounded above (CBA spaces) respectively.
This approach, initiated by A.D. Alexandrov \cite{alexandrov1957uber}, has been extensively studied from various perspectives, resulting in a rich and well-developed theory; see for instance \cite{ballmann1995lectures,bridson1999metric,burago1992ad,alexander2024alexandrov} and the bibliography therein.

After the groundbreaking work of Burago, Gromov and Perelman \cite{burago1992ad}, a comprehensive structure theory for Alexandrov spaces with curvature bounded below has been extensively developed.
This theory, which illuminates how sectional curvature bounds interact with the topological, geometric, and measure-theoretic properties of underlying spaces, has been further investigated by Otsu and Shioya \cite{otsu1994riemannian}, Perelman \cite{perelman1994dc,perelman1991alexandrov}, Perelman and Petrunin \cite{perel1994extremal,perelman1994quasigeodesics}, Petrunin \cite{petrunin1998parallel}, Ambrosio and Bertrand \cite{ambrosio2018dc}, et al.
Recently, Lytchak and Nagano \cite{lytchak2019geod,lytchak2021topological} established a comprehensive structure theory of Alexandrov spaces with curvature bounded above and the local geodesic extension property (so-called GCBA spaces).
These structure theories reveal the \emph{Riemannian nature} of finite-dimensional Alexandrov spaces, which possess many properties analogous to those of Riemannian manifolds.

Besides Alexandrov geometry, there have been significant efforts to develop synthetic notions of (sectional) curvature for non-Riemannian metric spaces such as normed spaces and Finsler manifolds.
These efforts can be traced back to the pioneering works of Busemann \cite{busemann1948spaces,busemann1955geometry}, who introduced a weaker notion of non-positive curvature, characterized by the convexity of distance functions along geodesics, to study Finslerian spaces:
\begin{definition}[Busemann convex]\label{def2:intro}
    A complete geodesic space $(X,d)$ is said to be Busemann convex if for any pair of constant-speed geodesics $\gamma, \eta: [0,1]\rightarrow X$, the function
    \begin{equation}
        t\mapsto {d(\gamma(t), \eta(t))}
    \end{equation}
    is convex on $[0,1]$.
\end{definition}
\noindent
By the triangle inequality, one can equivalently define the Busemann convexity as follows (see for example \cite[Proposition 8.1.2]{papadopoulos2014metric}).  
\begin{definition}[Busemann convex]\label{def1:intro}
    A complete geodesic space $(X,d)$ is said to be Busemann convex if for any pair of constant-speed geodesics $\gamma, \eta: [0,1]\rightarrow X$ starting from a common point $\gamma(0)=\eta(0)$, the function
    \begin{equation}\label{eq:Busemann_monotone_intro1}
        t\mapsto \frac{d(\gamma(t), \eta(t))}{t}
    \end{equation}
    is non-decreasing on $[0,1]$.
\end{definition}
\noindent
This synthetic notion of curvature upper bound, together with another notion introduced by Busemann, called G-spaces, has aroused great interest in the study of non-Riemannian geometry, leading to a rich body of literature; see for example \cite{papadopoulos2014metric,thurston1996Gspace} and bibliography therein.

Similar to Definition \ref{def1:intro}, there is a counterpart of non-negative (sectional) curvature in the sense of Busemann, called \emph{Busemann concavity}.
However, despite the terminology `concavity', it is not meaningful to define a non-negative curvature bound, in the same way as Definition \ref{def2:intro}.
 \begin{definition}[Busemann concave]
    A complete geodesic space $(X,d)$ is said to be Busemann concave if for any pair of constant-speed geodesics $\gamma, \eta: [0,1]\rightarrow X$ starting from a common point $\gamma(0)=\eta(0)$, the function
    \begin{equation}\label{def3:intro}
        t\mapsto \frac{d(\gamma(t), \eta(t))}{t}
    \end{equation}
    is non-increasing on $[0,1]$.
\end{definition}
\noindent
The Busemann concavity can also be defined via comparison triangles by asking for a comparison inequality, which is strictly weaker than the Alexandrov triangle condition (see Remark \ref{rmk:Busemann_concavity_via_comparison_triangle}).
In addition,  using appropriate distortion coefficients, it is not hard to define Busemann spaces with (arbitrary) lower curvature bound.

Recently, in an effort to investigate the compatibility between synthetic analogues of sectional curvature lower bounds for Finsler manifolds and a synthetic notion of Ricci curvature bound called \emph{measure contraction property} (see \cite{ohta2007measure} and \cite{sturm2006geometry2}), Kell \cite{kell2019sectional} studied the metric measure geometry of Busemann concave spaces and uniformly smooth spaces. 
Under the existence of a non-trivial Hausdorff measure, Kell proved several geometric and analytic properties of Busemann concave spaces, including the measure contraction property, Poincar\'e inequality and uniqueness of tangent cones.
However, due to the instability of the Busemann concavity with respect to Gromov--Hausdorff convergence, it is not clear whether tangent cones are Banach spaces or not.

In light of Burago--Gromov--Perelman \cite{burago1992ad} and Kell \cite{kell2019sectional}, a natural question arises: 
\begin{question}
    Do Busemann concave spaces admit a refined structure theory analogous to that of Alexandrov spaces with curvature bounded below, such as rectifiability, uniqueness of Banach tangent cones, and fine measure-theoretic estimates for singular sets?
\end{question}
\noindent

While some geometric properties of Busemann spaces appear straightforward, establishing finer regularity results presents significant challenges.
Besides the instability of Busemann spaces in the Gromov--Hausdorff topology, one of the main difficulties is that Busemann spaces, in general, do not possess a robust notion of angles as in Alexandrov spaces.
This difficulty reflects the non-Riemannian nature of Busemann spaces, making their tangent cones hard to handle.
It becomes more subtle, if we do not have a non-trivial Hausdorff measure. 
For instance, there is a compact convex subset $K$ in the infinite-dimensional $\ell_p$-space with $1<p<\infty$, such that the Gromov--Hausdorff limit of blow-ups fails to coincide with the tangent cone at any interior point of $K$ \cite{kell2019sectional}.

In a recent work of Fujioka and Gu \cite{fujioka2025top}, building upon Lytchak--Nagano \cite{lytchak2021topological} and Lytchak--Nagano--Stadler \cite{stadler2024cat}, a breakthrough about the topological regularity of Busemann spaces with non-positive curvature and the local geodesic extension property (so-called GNPC spaces) has been made.
A key ingredient in their work is the introduction of two distinct notions of angle.
These notions of angle, which coincide in GCBA spaces, capture the subtle non-Riemannian structure of Busemann spaces.
Based upon these notions of angle, Fujioka and Gu introduced `almost orthogonal coordinates' in the context of Busemann spaces with non-positive curvature, namely strainer maps, enabling them to implement the strategy developed in \cite{stadler2024cat}.

Concerning Busemann spaces with non-negative curvature, the problem seems \emph{more difficult}.
In contrast to the Busemann convexity and Alexandrov triangle condition, the Busemann concavity yields less geometric information.
For instance, unlike Busemann convex spaces, the Busemann concavity provides no information on the behavior of distance functions from a fixed point along geodesics, resulting in limited understanding of differentiability of distance functions.
In addition, the `geodesic contraction map' fails to provide a deformation retraction.
Therefore, it is not clear whether Busemann concave spaces are locally contractible or not (see \cite[Section 8]{fujioka2025top}), thereby leading to little knowledge of their local topological structure.

In order to establish a fine structure theory beyond Kell \cite{kell2019sectional} (and Le Donne \cite{donne2010unique_tangent_cone}), in which the authors establish that almost all points in Busemann concave spaces having non-trivial Hausdorff measure admit unique tangent cones isometric to Carnot groups equipped with left-invariant sub-Finsler metrics, we need to impose some natural (intrinsic) assumptions, which are compatible with a large class of Finsler manifolds and normed spaces.
Notice that any Banach space with strictly convex norm is \emph{both Busemann convex and Busemann concave}.  
From the viewpoint of Banach and Minkowski space theory, to obtain finer geometric properties, it is natural to impose some quantitative convexity condition on the norm.
In their well-known paper \cite{ball1994sharp}, Ball, Carlen and Lieb introduced two fundamental notions in the geometric theory of Banach spaces, called \emph{$2$-uniform convexity} and \emph{$2$-uniform smoothness} (see \cite{ball1994sharp} for general notions of $p$-uniform convexity and $q$-uniform smoothness), which ask for the existence of two least constants $C, S\geq 1$ such that for any $u,v$ in a Banach space $(E,\|\cdot\|)$, it holds
\begin{equation}\label{eq1:intro}
    \left\| \frac{u+v}{2}\right\|^2
    \leq
    \frac{1}{2}\|u\|^2 + \frac{1}{2}\|v\|^2 - \frac{1}{4C}\| u-v\|^2,
\end{equation}
and
\begin{equation}\label{eq2:intro}
    \left\| \frac{u+v}{2}\right\|^2
    \geq
    \frac{1}{2}\|u\|^2 + \frac{1}{2}\|v\|^2 - \frac{S}{4}\|u-v\|^2.
\end{equation}
Notice that the uniform convexity constant $C$ and the uniform smoothness constant $S$ are in fact the moduli of convexity and concavity of the squared norm $\|\cdot\|^2/2$ (see \cite{ohta2012non}). 
In particular, $C=1$ or $S=1$ holds if and only if the norm $\|\cdot\|$ is induced by an inner product.  
Thus, the $2$-uniform convexity \eqref{eq1:intro} and the $2$-uniform smoothness \eqref{eq2:intro} can also be viewed as quantitative continuity conditions, where the constants $S, C$ measure the degree to which the Banach space $(E, \| \cdot\|)$ deviates from a Hilbert space (see \cite[Section 4]{ohta2009uniform}).
Inspired by Ball--Carlen--Lieb \cite{ball1994sharp} and Busemann's work \cite{busemann1955geometry}, Ohta \cite{ohta2007convexities,ohta2009uniform} (see also \cite{ohta2021comparison,kondo2012toponogov}) generalized these two concepts to metric spaces, called $C$-convexity and $S$-concavity\footnote{They are called 2-uniform convexity and 2-uniform smoothness by Ohta \cite{ohta2009uniform}, and renamed as $k$-convexity and $k$-concavity in \cite{ohta2021comparison}.}, and found their corresponding geometric quantities in Finsler geometry.   
As noticed by Shen \cite{shen2001lectures} and Ohta \cite{ohta2009uniform,ohta2021comparison}, the $S$-concavity is satisfied by Finsler manifolds with non-negative flag curvature whose tangent spaces are uniformly smooth Minkowski spaces.
In particular, $S$-concave spaces are just Alexandrov spaces with non-negative curvature when $S=1$ (see \cite{ohta2021comparison}). 
Within this framework, Ohta \cite{ohta2007convexities,ohta2009uniform} further generalized the Alexandrov--Toponogov comparison theorem to Finsler spaces.

According to the preceding discussion, we impose the following assumptions on the underlying space: the \emph{$S$-concavity} introduced by Ohta \cite{ohta2007convexities,ohta2009uniform}, and the \emph{local semi-convexity}, which is weaker than Ohta's \emph{$C$-convexity} (see Section \ref{sect:S_concave_local_squared_convex_Busemann_concave} for more details).
The present work is devoted to establishing a fine structure theory for Busemann concave spaces satisfying such properties, extending the well-known results \cite{burago1992ad} on Alexandrov spaces with curvature bounded below to the Busemann setting.

\subsection{Main results}
Throughout this part, $(X,d)$ denotes a complete, separable geodesic space.

Our first main theorem concerns the Hausdorff dimension and Hausdorff measure of $X$.  
A key ingredient is the strainer number, which is the maximal cardinality of `local almost orthogonal coordinates' in $X$.

\begin{theorem}[Proposition \ref{prop:Buseman_concave_MCP} and Corollary \ref{cor:strained_points_open_dense_top_manifold}]\label{thm:main_result_1}
    Let $(X,d)$ be an $S$-concave, locally semi-convex, Busemann concave space with $S \geq 1$.
    Then $X$ is of finite Hausdorff dimension if and only if it has finite strainer number.
    In either case, both values coincide with the topological dimension of $X$.
    Moreover, $X$ admits a non-trivial Hausdorff measure $\mathcal{H}^n$ of integer dimension $n$, and $(X,d,\mathcal{H}^n)$ satisfies the measure contraction property $\mathrm{MCP}(0, n)$.
    In particular, $(X,d)$ is doubling and proper, and contains a topological $n$-manifold part which is open and dense in $X$. 
\end{theorem}

Theorem \ref{thm:main_result_1} generalizes the well-known results of Burago--Gromov--Perelman \cite[Corollary 6.5, 6.7, 6.8]{burago1992ad} to the Busemann setting.
We remark that, finding sufficient conditions for the existence of a non-trivial Hausdorff measure on Busemann concave spaces, is an open problem left in \cite{kell2019sectional}.

\medskip

Our second main result shows that, if $X$ is of finite Hausdorff dimension $n$, then its almost regular part $\mathcal{A}(n,\delta)$ has full measure.  
We refer to Section \ref{sect:dimension} for the definition of almost regular parts $\mathcal{A}(n,\delta)$, namely the set of \emph{$(k,\delta)$-strained points}.
Roughly speaking, $\mathcal{A}(n,\delta)$ consists of points some of whose neighborhoods admit almost orthogonal coordinate charts of dimension $n$.
\begin{theorem}[Theorem \ref{thm:full_measure_n_strained_points}]\label{thm:main_result_2}
    Let $(X,d)$ be an $n$-dimensional $S$-concave, locally semi-convex, Busemann concave space with $S \geq 1$.
    Then for any $\delta>0$, $\mathcal{H}^n(X \setminus \mathcal{A}(n,\delta))=0$.
    In particular, $X$ contains an open dense topological $n$-manifold part which has full $n$-Hausdorff measure.
\end{theorem}

The next theorem generalizes a well-known result for finite-dimensional Alexandrov spaces with curvature bounded below, on which $\mathcal{H}^n$-almost every point admits a unique tangent cone isometric to the Euclidean space $\mathbb{R}^n$.
Under an additional geometric condition called local $p$-uniform convexity, we can prove more geometric properties of these Banach tangent cones.
For more on the connection with the work of Kell \cite{kell2019sectional}, see Remark \ref{rmk:charcter_Banach_tangent_cone_Kell}.

\begin{theorem}[Theorem \ref{thm:n_rect} and \ref{thm:Banach_tangent_cone}, Corollary \ref{cor:characterize_Banach_tangent_cone}]\label{thm:main_result_3}
    Let $(X,d)$ be an $n$-dimensional $S$-concave, locally semi-convex, Busemann concave space with $S \geq 1$.
    Then $X$ is $n$-rectifiable, and $\mathcal{H}^n$-almost every point $x\in X$ admits a unique tangent cone $(T_xX, d_x, o)$, which is isometric to a finite-dimensional Banach space.
    If $X$ further satisfies the local $p$-uniform convexity, then every Banach tangent cone possesses a strictly convex norm, and is both $2$-uniformly smooth and $p$-uniformly convex.
\end{theorem}

We also obtain a measure-theoretic estimate for singular strata of spaces.
This result, together with Theorem \ref{thm:main_result_2}, generalizes the well-known results about the Hausdorff measure and Hausdorff dimension of singular strata in Alexandrov spaces with curvature bounded below \cite{burago1992ad}.

\begin{theorem}[Theorem \ref{thm:main_thm_Hausdorff_dim_est}, Corollary \ref{cor:stratification}]\label{thm:main_result_5}
    Let $X$ be an $n$-dimensional $S$-concave, locally semi-convex, Busemann concave space with $S\geq 1$. 
    Then for any $\delta>0$ and $k\in\{1,\ldots,n\}$, the Hausdorff dimension of the singular set $X\setminus\mathcal{A}(k, \delta)$ is at most $k-1$ for all $k=1,\ldots,n$.
    In particular, $X$ admits a stratification $\{X_k\}_{k=0}^n$ such that $X$ is the disjoint union of $\{X_k\}_{k=0}^n$, with $\mathrm{dim}_H(X_k)\leq k$ for all $k=0,\ldots,n$, and the top-dimensional stratum $X_n$ is a topological $n$-manifold.
\end{theorem}
\begin{remark}
    Although Theorem \ref{thm:main_result_5} yields Theorem \ref{thm:main_result_2}, its proof is considerably more intricate and includes a careful analysis of tangent cones, as well as precise control over the discrepancies between different notions of angle. 
    In contrast, the proof of Theorem \ref{thm:main_result_2} is more direct and does not need any information on tangent cones.
\end{remark}

\begin{remark}
    For $S\geq 1$ small enough,  all the Theorems \ref{thm:main_result_1}--\ref{thm:main_result_5}  remain valid  without the local semi-convexity assumption.
    This will be discussed  in \cite{ylm2026Busemann_II}, where we establish a  structure theory for finite-dimensional $S$-concave Busemann concave spaces satisfying the so-called $(\varepsilon,\delta)$-weak quadruple condition.
    This condition generalizes the classical quadruple condition (see \cite{burago2001course}) for Alexandrov spaces with curvature bounded below, and it is satisfied by all $S$-concave Busemann concave spaces with $1\leq S<s(\varepsilon)$, as well as by all finite-dimensional $2$-uniformly smooth, $\infty$-uniformly convex Banach spaces, including finite-dimensional $\ell_p^n$ spaces with $p\in [2,\infty)$.
    In particular, our framework includes finite-dimensional Alexandrov spaces with non-negative curvature. 
\end{remark}

\subsection{Main ideas and tools}
We now summarize the main techniques and new ideas in our proof, and compare our methods with those employed in the structure theory of Alexandrov spaces with curvature bounded below \cite{burago1992ad}.

The main geometric tool is a variant of a classical notion, called \emph{strainer maps}, which can be interpreted as local almost orthogonal coordinate charts.
The notion of strainer map is introduced by Burago--Gromov--Perelman \cite{burago1992ad} in the context of Alexandrov spaces with curvature bounded below, and later generalized by Lytchak and Nagano \cite{lytchak2019geod} to GCBA spaces, and recently further developed by Fujioka and Gu \cite{fujioka2025top} in the setting of GNPC spaces.

Similar to \cite{fujioka2025top}, to establish strainer maps, we introduce two notions of angle to study the geometry of the underlying space: the first one, called \emph{angle viewed from a fixed point} (Definition \ref{def:angle_view_from_fixed_point}), measures the orthogonality of geodesics, while the second one, called \emph{angle of fixed scale} (Definition \ref{def:angle_of_fixed_scale}), is closely related to tangent cones.
These two notions of angle, which may not coincide in the Busemann setting (see Example \ref{example:angle_asymmetry_2}), are implicitly correlated with \emph{Euclidean comparison angle} by the uniform smoothness constant $S$ as well as the \emph{ratio of side-lengths} of geodesic triangles (see Lemma \ref{lemma:angle_well_defined} and \ref{lemma:relation_metric_and_angle}).
In Table \ref{tab:angle}, we list the notions of angle used in this paper.
For more details, we refer to Section \ref{sect:angles}.
\begin{table}[h] 
    \centering 
    \begin{tabular}{|c|c|c|} 
        \hline
        Notation & Meaning\\ 
        \hline
        $\angle px\xi$ & angle at $x$ viewed from the point $p$ along the geodesic $\xi$  \\
        \hline
        $\angle_x(\gamma(t),\eta(s))$ & angle of the geodesics $\gamma,\eta$ at $x$ of scale $t/s$\\
        \hline 
        $\tilde{\angle}_x(p,q)$ or $\tilde{\angle}pxq$ & comparison angle of the Euclidean triangle $\tilde{\Delta}pxq$ at $x$\\
        \hline
        $\angle_x((\gamma,t),(\eta,t))$ & angle quasi-metric on space of directions with common length \\
        \hline
    \end{tabular}
    \caption{Different notions of angle}\label{tab:angle}
\end{table}

Using the first notion of angle, we define strainers and strainer maps.
Our definition is adapted from the ones in \cite{burago1992ad} and \cite{fujioka2025top}, to address the non-uniqueness of geodesics between pairs of points and the lack of monotonicity of comparison angles in the Busemann setting.
To construct local almost orthogonal coordinate charts with strainer maps, we then prove the $\varepsilon$-openness and bi-Lipschitz continuity of strainer maps.  
The main challenge in our setting arises from the asymmetric nature of strainer maps (see Remark \ref{rmk:asymmetry_k_strainer}).
This challenge, which also arises in the study of Busemann spaces with non-positive curvature \cite{fujioka2025top}, stems from the inherent asymmetry of the first notion of angle, due to the non-Riemannian nature of Busemann spaces, as illustrated in Example \ref{example:angle_asymmetry_1}.
To overcome this difficulty, we adopt a strategy similar to that of \cite{fujioka2025top} by introducing an anisotropic $\ell_1$-norm for the target spaces of strainer maps.
Then we are able to establish the properties we need.

In the proof of Theorem \ref{thm:main_result_1}, a key ingredient is the self-improvement property of strainers (Lemma \ref{lemma:self_improvement_strainer}), which ensures that once a `less orthogonal' strainer map is found, one can automatically obtain a `more orthogonal one'.
Unfortunately, the `straightening strategy' to obtain improved strainers used in \cite{burago1992ad} (see also \cite[Proposition 10.8.17]{burago2001course}) is not applicable in the Busemann setting.
In fact, the new $k$-tuple $(p'_1,p_2,\ldots,p_k)$ obtained from an old strainer $(p_1,p_2,\ldots,p_k)$ by the straightening procedure is no longer a strainer, due to the asymmetry of strainers and angles in the Busemann setting.
To overcome this difficulty, we borrow a strategy from computer science, known as `dequeuing and enqueuing'\footnote{In computer science terminology, to dequeue means to remove the first element from a queue, and to enqueue means to add a new element to the end of the queue.}.
This method allows us to improve the orthogonality of strainers step by step, ultimately leading to the desired self-improvement property.

To estimate the Hausdorff measure of almost regular sets (Theorem \ref{thm:main_result_2}), we employ an approach substantially different from the one used for Alexandrov spaces \cite[Theorem 10.6]{burago1992ad}, in which the authors make use of the compactness of the space of directions at a point, to get an upper bound for the maximal cardinality of $r$-separated subsets within singular sets (see \cite[Lemma 10.5]{burago1992ad}).
This method relies implicitly on the metric cone structure of tangent cones in Alexandrov spaces, which is not applicable directly in our context.
Instead, we adopt a more straightforward method, motivated by the study of GCBA spaces in \cite{lytchak2019geod}, leveraging the {\em almost extendable property} of geodesics established in Lemma \ref{lemma:null_measure_cut_points}.
Such a property is fulfilled by metric measure spaces satisfying the measure contraction property \cite{von2008local}.
This approach enables us to analyze the infinitesimal behavior of strainer maps when restricted to the singular strata without any knowledge about the structure of tangent cones, thereby allowing us to determine the top-dimensional Hausdorff measure of singular sets.
It is worth mentioning that we do not need the so-called \emph{local geodesic extension property} (see \cite[Definition II.5.7]{bridson1999metric}).
This geometric property, which is intimately connected to the local geometry and topology of the underlying spaces (see, for example, \cite[Theorem 1.5]{lytschak2004affine}), is much stronger than the almost extendable property.

Based on Theorem \ref{thm:main_result_2}, the $n$-rectifiability of $X$ follows from the local bi-Lipschitz property of strainer maps and a standard covering argument.
The existence and uniqueness of Banach tangent cones at almost every point in $X$, is a consequence of a well-known structure theorem of Kirchheim, concerning rectifiable sets in metric spaces \cite[Theorem 9]{kirchheim1994rectifiable}.

Finally, to estimate the Hausdorff dimension of singular strata (Theorem \ref{thm:main_result_5}), the main difficulty is the absence of a metric cone structure for tangent cones. 
We cannot apply the technique used in Theorem \ref{thm:main_result_2} either, as the almost extendable property of geodesics only holds up to null top-dimensional Hausdorff measure sets, which prevents us from obtaining a finer estimate for the Hausdorff dimension of singular strata.
To overcome these difficulties, we introduce a new notion called \emph{space of directions with common length} (Definition \ref{def:space_directions_common_length}).
While in general no angle metric exists as in spaces of directions in Alexandrov spaces, the space of directions with common length carries a natural {\em angle quasi-metric} (see Definition \ref{def:space_directions_common_length}) induced by the Euclidean law of cosines on tangent cones, which coincides with the second notion of angle in a dense subset of it.
We then establish the uniform compactness of spaces of directions with common length (Lemma \ref{lemma:space_direction_uniform_compact}) in the sense of maximal cardinality of $\varepsilon$-separated subsets, and provide a technical lemma (Lemma \ref{lemma:variant_S_concave}) that quantifies how the ratio of side-lengths in angles of fixed scale controls its discrepancy with angles viewed from a fixed point.
This quantification is in fact closely connected to a strengthened doubling condition called the \emph{$\mathrm{ATB}$-condition} introduced by Lebedeva--Ohta--Zolotov \cite{lebedeva2021self}.
This lemma, together with the uniform compactness of spaces of directions with common length, allows us to adapt a strategy similar to that of \cite[Lemma 10.5]{burago1992ad}, to derive an upper bound for the maximal cardinality of $r$-separated subsets of singular sets within small cylindrical regions (Lemma \ref{lemma:hausdorff_dim_est}), thereby enabling us to determine the Hausdorff dimension of singular strata.

\subsection{Organization of the paper}
The paper is organized as follows: in Section \ref{sect:prelimin_notations}, we introduce basic terminologies and notations used throughout the paper.
In Section \ref{sect:S_concave_local_squared_convex_Busemann_concave}, we introduce several geometric conditions, including the $S$-concavity, local semi-convexity and Busemann concavity, and give some examples.
In Section \ref{sect:angles}, we introduce two notions of angle and establish several key properties, including the almost comparison inequalities.
We further define tangent cones and spaces of directions with common length, and prove several properties of tangent cones and spaces of directions with common length.
We then study strainers and strainer maps in Section \ref{sect:strainer_maps}, and prove several important properties including $\varepsilon$-openness, bi-Lipschitz continuity, and self-improvement property.
In Section \ref{sect:dimension}, we introduce strainer numbers and prove Theorem \ref{thm:main_result_1}.
In Section \ref{sect:Hausdorff_measure_strata_main}, we investigate the structure and measure-theoretic properties of Busemann concave spaces, and prove Theorems \ref{thm:main_result_2}--\ref{thm:main_result_5}.

\section{Preliminaries and notations}\label{sect:prelimin_notations}
\noindent
In this section, we briefly recall the terminologies and notations used in this manuscript.
Our notations are standard and mainly follow from \cite{burago2001course,burago1992ad,shanmugalingam2015sobolev,ambrosio2004topics}.

\subsection{Spaces and maps}
Throughout this paper, we denote by $(X,d)$ a complete, separable metric space with the distance function $d$.
Given two points $x, y \in X$, we often denote their distance $d(x,y)$ by $|xy|$ for simplicity.
For $r > 0$, we denote by $B(x,r)$ and $\bar{B}(x,r)$ the open and closed balls centered at $x$ with radius $r$, respectively.

Given a subset $E \subset X$ and $r > 0$, we say that $E$ is \emph{$r$-separated} if every pair of distinct points $x, y \in E$ satisfies $|xy| \geq r$.
An $r$-separated subset $E$ is called \emph{maximal} if there does not exist any $r$-separated subset $E' \subset X$ that properly contains $E$.
For a set $E$, we denote by $\beta_E(r) \in \mathbb{N} \cup \{\infty\}$ the largest possible cardinality of a maximal $r$-separated subset of $E$.
A metric space $(X, d)$ is said to be \emph{doubling}, or \emph{$N$-doubling}, if for any $r > 0$, the cardinality of any maximal $r/2$-separated subset of any ball $B(x, r)$ is at most $N$, i.e., $\beta_{B(x, r)}(r/2) \leq N$ for all $x \in X$ and $r > 0$.
Equivalently, any ball $B(x, r) \subset X$ can be covered by at most $N$ balls of radius $r/2$.

For a curve $\gamma \subset X$, we denote its length by $l(\gamma)$.
A curve $\gamma:[0,1] \to X$ is called a \emph{constant-speed geodesic} from $x$ to $y$ if it is a length-minimizing curve connecting $x$ and $y$ satisfying that $|\gamma(t)\gamma(s)| = |t-s|\,|xy|$ for all $t, s \in [0,1]$.
A length-minimizing curve is called a \emph{unit-speed geodesic} or \emph{geodesic} for short, if it is parametrized by arc-length.
A metric space $(X,d)$ is said to be \emph{geodesic} if any pair of distinct points can be connected by a geodesic.
It is said to be \emph{uniquely geodesic} if any pair of distinct points can be connected by a unique geodesic.
A geodesic space $(X,d)$ is said to be non-branching if for any pair of constant-speed geodesics $\gamma,\eta:[0,1]\to X$, the condition $\gamma|_{[0,t]}=\eta|_{[0,t]}$ for some $t\in (0,1)$ implies that $\gamma=\eta$.
We say a function $f:X\to \mathbb{R}$ is convex if $f\circ \gamma$ is convex on $[0,1]$ for any constant-speed geodesic $\gamma$.

For a geodesic space $(X,d)$, we denote by $\Delta xyz$ a \emph{geodesic triangle} in $X$ with vertices $x,y,z\in X$ which is the union of three geodesics which connect $x,y,z$ pairwise.
A triangle $\tilde{\Delta}xyz \subset \mathbb{R}^2$ is said to be a Euclidean comparison triangle of $\Delta xyz\subset X$ if $|xy|=|\tilde{x}\tilde{y}|,|yz|=|\tilde{y}\tilde{z}|$ and $|xz|=|\tilde{x}\tilde{z}|$.
We denote the angle at vertex $\tilde{x}$ of the Euclidean comparison triangle $\tilde{\Delta}xyz$ by $\tilde{\angle}_x(y,z)$, or $\tilde{\angle}yxz$ for simplicity.

A map $F:X\to Y$ between two metric spaces $(X,d_X)$ and $(Y,d_Y)$ is called \emph{Lipschitz} or \emph{$L$-Lipschitz} if there is $L>0$ such that $|F(x)F(y)|\leq L|xy|$ for any $x,y\in X$ \footnote{For different metric spaces, we use the same notation $|xy|$ and $|F(x)F(y)|$ to denote the distance $d_X(x,y)$ and $d_Y(F(x),F(y))$ respectively, if there is no ambiguity.}.
It is called \emph{$L$-biLipschitz} for some  $L\geq 1$ if $|xy|/L\leq |F(x)F(y)|\leq L|xy|$ for any $x,y\in X$.

\subsection{Hausdorff measure, dimensions and rectifiability}\label{subsect:H_measure_dimensions_rect}

Let $(X,d)$ be a complete, separable metric space.
For $\delta>0, \alpha\in [0,\infty)$ and $E\subset X$, $\mathcal{H}_{\delta}^{\alpha}(E)$ is defined as
\begin{equation}
    \mathcal{H}^{\alpha}_{\delta}(E):=\inf\left\{\omega_\alpha\sum_{i=1}^{\infty}\left(\frac{\mathrm{diam}(E_i)}{2}\right)^{\alpha}: E\subset \bigcup_{i=1}^{\infty}E_i, \mathrm{diam}(E_i)\leq \delta  \right\},
\end{equation}
where $\omega_\alpha:=\pi^{\alpha/2}/\Gamma(\alpha/2+1)$ and $\Gamma$ is the Gamma function.
The \emph{$\alpha$-Hausdorff measure} $\mathcal{H}^{\alpha}(E)$ of $E$ is defined as $\mathcal{H}^{\alpha}(E):=\sup_{\delta>0}\mathcal{H}^{\alpha}_{\delta}(E)$.
We also use $\mathcal{H}^{\alpha}_d$ to emphasize the metric $d$.

It is known that $\alpha$-Hausdorff measures on metric spaces are outer measures, and moreover Borel regular measures; see for example \citep[Theorem 2.1.4]{ambrosio2004topics}.
Furthermore, $\alpha$-Hausdorff measures are also Radon measures on complete metric spaces, see \citep[Proposition 3.3.44]{shanmugalingam2015sobolev}.
For $n\in \mathbb{N}$, the $n$-Hausdorff measure on $\mathbb{R}^n$ coincides with the standard $n$-dimensional Lebesgue measure on $\mathbb{R}^n$.

The \emph{Hausdorff dimension} $\mathrm{dim}_H(E)$ of a subset $E$ is the infimum of numbers $\alpha>0$ such that $\mathcal{H}^{\alpha}(E)=0$ if such numbers exist.
Otherwise, we say the Hausdorff dimension of $E$ is infinite.
For an $L$-Lipschitz map $F:X\to Y$, it can be seen from the definition that $\mathcal{H}^{\alpha}(F(E))\leq L^{\alpha}\mathcal{H}^{\alpha}(E)$.
In particular, the Hausdorff dimension is invariant under bi-Lipschitz homeomorphisms.
The Hausdorff dimension is monotone in the sense that $\mathrm{dim}_H(E)\leq \mathrm{dim}_H(E')$ for each measurable subset $E\subset E'\subset X$.
Furthermore, if $\{E_i\}_i$ is an at most countable covering of $X$, then $\mathrm{dim}_H(X)=\sup_{i}\mathrm{dim}_H(E_i)$.
For more discussions about Hausdorff dimension, we refer to \cite[Section 1.7.4]{burago2001course}.

The \emph{rough dimension} $\mathrm{dim}_r(E)$ of a set $E$ is defined as the infimum of numbers $\alpha>0$ such that $\limsup_{r\searrow 0}r^{\alpha}\beta_{E}(r)=0$ if such numbers exist, otherwise we take the rough dimension of $E$ equal to infinity. 
It is well known that the topological dimension (Lebesgue covering dimension), Hausdorff dimension, and rough dimension satisfy the following relationship:
\begin{equation}
    \mathrm{dim}_T(E)\leq \mathrm{dim}_H(E)\leq \mathrm{dim}_r(E),\quad \text{for any }E\subset X.
\end{equation}
See \citep[Theorem 8.14]{heinonen2001lectures} and \citep[Section 10.6.4]{burago2001course}.

A metric space $(X,d)$ is said to be \emph{$n$-rectifiable} for $n\in \mathbb{N}$ if it can be covered, up to an $\mathcal{H}^n$-null set, by countably many Lipschitz images of $\mathbb{R}^n$, i.e., there exists a countable family of Lipschitz maps $f_i:E_i\to X$ defined on measurable subsets $E_i\subset \mathbb{R}^n$ such that $\mathcal{H}^n(X\setminus \cup_{i=1}^{\infty}f_i(E_i))=0$.
For more details on rectifiability of metric spaces, we refer to \cite{kirchheim1994rectifiable,ambrosio2000rectifiable,bate2017characterizations}.

\subsection{Gromov--Hausdorff convergence and blow-ups}
In this subsection, we recall the definitions of pointed Gromov--Hausdorff convergence and blow-ups of pointed metric spaces.
More details and equivalent definitions of Gromov--Hausdorff convergence can be found in \cite{burago2001course,shanmugalingam2015sobolev,guy_c_david2015GAFA,bate2022characterising}.

Given two pointed metric spaces $(X,d_X, x)$ and $(Y,d_Y,y)$ and $\varepsilon>0$, a map $F:(X,d_X,x)\to (Y,d_Y,y)$ is called an \emph{$\varepsilon$-isometry} if it satisfies the following three conditions:
\begin{enumerate}
    \item $F(x)=y$;
    \item for any $z,z'\in B(x,1/\varepsilon)\subset X$, it holds that $\big||zz'|- |F(z)F(z')| \big|<\varepsilon$;
    \item for any $r<1/\varepsilon$, $B(y,r-\varepsilon)$ is included in the $\varepsilon$-neighborhood of $F(B(x,r))$.
\end{enumerate}

Given a sequence $\{(X_n,d_n,x_n)\}_{n}$ of pointed metric spaces, we say that $(X_n,d_n,x_n)$ pointed Gromov--Hausdorff converges to a pointed metric space $(X,d,x)$, if for any $\varepsilon>0$, there exists an $N_0:=N_0(\varepsilon)\in \mathbb{N}$ such that for any $n\geq N_0$, we can find $\varepsilon$-isometries $F_n$ from $(X_n,d_n,x_n)$ to $(X,d,x)$.
It is known that if the doubling constants of a sequence of pointed metric spaces are uniformly bounded above, then it has a subsequence converging to a pointed metric space with the same doubling constant bound, in the pointed Gromov--Hausdorff topology, see for example \cite[Proposition 2.2]{gigli2015euclidean}.

For a pointed metric space $(X,d,x)$ and a positive number $\lambda\in (0,1]$, we call the rescaled space $(X, d/\lambda, x)$ a \emph{blow-up} of $X$ at $x$.
We denote by $\mathrm{Tan}(X,d,x)$ the collection of all pointed Gromov--Hausdorff limits of the blow-ups $\{(X,d/\lambda_n,x)\}$ at $x$ for some $\{\lambda_n\}\subset (0,1]$ converging to $0$.
If $(X,d)$ is doubling, we can see that $\mathrm{Tan}(X,d,x)$ is non-empty.

\subsection{Uniform convexity and smoothness}
In this subsection, we recall the uniform convexity and smoothness of normed spaces.
See \cite{ohta2021comparison} for more details.

Let $(E, \|\cdot\|)$ be a normed space and $p\in (1,\infty)$.
We say $E$ is \emph{$p$-uniformly convex} if there exists a constant $C\geq 1$ such that for any $u,v\in E$, it holds
\begin{equation}\label{eq:def_p_uniform_convex}
    \left\| \frac{u+v}{2}\right\|^p
    \leq
    \frac{1}{2}\|u\|^p + \frac{1}{2}\|v\|^p - \frac{1}{4C}\| u-v\|^p.
\end{equation}
The least constant $C$ satisfying \eqref{eq:def_p_uniform_convex} is called the \emph{uniform convexity constant}.   
This characterizes the convexity of the norm in a quantitative way.
Similarly, we say that $E$ is \emph{$p$-uniformly smooth} if there exists a constant $S\geq 1$ such that for any $u,v\in E$, it holds
\begin{equation}\label{eq:def_p_uniform_smooth}
    \left\| \frac{u+v}{2}\right\|^p
    \geq
    \frac{1}{2}\|u\|^p + \frac{1}{2}\|v\|^p - \frac{S}{4}\|u-v\|^p.
\end{equation}
The least constant $S$ satisfying \eqref{eq:def_p_uniform_smooth} is called the \emph{uniform smoothness constant}, which characterizes the concavity of the norm.
Note that any normed space $(E,\|\cdot\|)$ satisfies the inequality $\|(u+v)/2\|^p\leq \frac{1}{2}\|u\|^p + \frac{1}{2}\|v\|^p$ for all $u,v\in E$, due to the convexity of norm and of the function $t\mapsto t^p$.

For a normed space $(\mathbb{R}^n, \|\cdot\|)$, it is well-known that the norm $\|\cdot\|$ is induced by an inner product if and only if $(\mathbb{R}^n, \|\cdot\|)$ is $2$-uniformly convex with $C = 1$, which is also equivalent to $(\mathbb{R}^n, \|\cdot\|)$ being $2$-uniformly smooth with $S = 1$.
Furthermore, it can be deduced from Clarkson's inequality (see, for example, \cite{ball1994sharp}) that any finite-dimensional $\ell^n_p:=(\mathbb{R}^n, \|\cdot\|_p)$-space with $p\in [2,\infty)$ is $p$-uniformly convex with $C=2^p/4$ as well as being $2$-uniformly smooth with $S=p-1$.
By duality, any finite-dimensional $\ell^n_q$-space with $q\in (1,2]$ is $2$-uniformly convex with $C=(q-1)^{-1}$ as well as being $q$-uniformly smooth with $S=4/2^q$.
Note that among $\ell^n_p$-spaces, only $\ell^n_2$-space is both $2$-uniformly convex and $2$-uniformly smooth; see \cite[Section 1.2]{ohta2021comparison} and \cite{ohta2009uniform,kuwae2014jensen} for more details.    
For infinite-dimensional $L_p$-spaces, similar properties also hold due to Beckner's inequality.

\section{S-concavity, local semi-convexity and Busemann concavity}\label{sect:S_concave_local_squared_convex_Busemann_concave}
\noindent
In this section, we introduce $S$-concavity, local semi-convexity, and Busemann concavity for geodesic spaces.
These synthetic sectional curvature bounds for metric spaces have been well-studied by Ohta \cite{ohta2009uniform,ohta2021comparison}, Kell \cite{kell2019sectional}, et al. (see for example \cite{kann1960bonnet}), as metric generalizations of curvature bounds for Finsler manifolds. 
For other synthetic notions of sectional curvature bounds, such as Busemann's notion of non-positive curvature, we refer to \cite{busemann1948spaces,papadopoulos2014metric,ivanov2019rigidity} and the recent work \cite{fujioka2025top}.

\begin{definition}[$S$-concave]\label{def:S_concave}
 We say that a complete geodesic space $(X,d)$ is \emph{$S$-concave} for $S \geq 1$, or semi-concave for short, if for any point $p \in X$ and any constant-speed geodesic $\xi:[0,1] \to X$, it holds that
    \begin{equation}\label{eq:S-concave}
        \left|p\xi(t)\right|^2
        \geq
        (1-t)\left|p\xi(0)\right|^2 + t\left|p\xi(1)\right|^2 - S\,t(1-t)\left|\xi(0)\xi(1)\right|^2,
    \end{equation}
    for any $t\in [0,1]$. 
    If we only require the above inequality to hold for curves satisfying $\sup_{t\in [0,1]}|p\xi(t)|<D$ for some $D>0$, then we call $(X, d)$ locally $S$-concave or locally semi-concave.
\end{definition}
\begin{remark}
    The $S$-concavity condition \eqref{eq:S-concave} is equivalent to saying that the function $t\mapsto |p\xi(t)|^2 - St^2|\xi(0)\xi(1)|^2$ is concave on $[0,1]$ for any constant-speed geodesic $\xi$.
    Furthermore, no geodesic space can be $S$-concave for $S<1$, except for the trivial case when $X$ is a singleton, see \cite[Section 8.3]{ohta2021comparison}.
\end{remark}

The $S$-concavity condition is a metric generalization of $2$-uniform smoothness of Banach spaces. 
It characterizes the infinitesimal concavity of underlying spaces. 
Such a geometric condition plays an important role in the study of the geometry of Banach spaces, metric spaces and Wasserstein spaces (see, for example, \cite{ohta2009uniform,ohta2007convexities,ohta2009gradient}).
For example, it is known that $1$-concave spaces are exactly Alexandrov spaces with non-negative curvature.  
In the setting of Finsler geometry, it has been proved by Ohta \cite[Theorem 4.2]{ohta2009uniform} (see also \cite[Corollary 8.20]{ohta2021comparison}) that any complete Berwald space with non-negative flag curvature and the uniform smoothness constant bounded above by $S$ is $S$-concave in the sense of Definition \ref{def:S_concave}.
It is clear that any $2$-uniformly smooth Banach space with the uniform smoothness constant $S_F\leq S$ is $S$-concave.
In particular, for any $p\geq 2$, the $\ell^n_p$-space is $S$-concave with $S=p-1$.

\begin{definition}[Locally semi-convex]\label{def:squared_convex}
We say that a complete geodesic space $(X,d)$ is \emph{$(C, D)$-semi-convex}, or locally semi-convex for simplicity, if there are $C\geq 0, D>0$ such that for any point $p \in X$ and any constant-speed geodesic $\xi:[0,1] \to X$ satisfying $\sup_{t\in [0,1]}|p\xi(t)|<D$, it holds 
    \begin{equation}\label{eq:squared-convex}
        \left|p\xi(t)\right|^2
        \leq
        (1-t)\left|p\xi(0)\right|^2 + t\left|p\xi(1)\right|^2 + C t (1-t)\left|\xi(0)\xi(1)\right|^2,\quad \text{for any }t\in [0,1].
    \end{equation}
 If $C=0$, we say that $(X,d)$ is \emph{locally convex}.
\end{definition}

It is clear that uniquely geodesic Banach spaces are surely locally convex.
So, locally semi-convex spaces are generalizations of Banach spaces, in a quantitative and non-linear way.
\begin{remark}
    There is a stronger convexity condition introduced by Ohta \cite{ohta2007convexities,ohta2009uniform}, called \emph{$k$-convexity}, which asks for the following stronger inequality:
    \begin{equation}\label{eq:k_convexity_ohta}
        \left|p\xi(t)\right|^2
        \leq
        (1-t)\left|p\xi(0)\right|^2 + t \left|p\xi(1)\right|^2 - k\,t(1-t) \left|\xi(0)\xi(1)\right|^2,\quad k\in (0,1],
    \end{equation}
    for any point $p$ and any constant-speed geodesic $\xi:[0,1]\to X$.
    This stronger condition has been studied in the setting of Finsler geometry by Ohta himself \cite[Section 8.3]{ohta2021comparison}, in which it is shown that any forward complete Finsler manifold with flag curvature bounded above by $\kappa\geq 0$ with vanishing $T$-curvature and the uniform convexity constant bounded from above by $C\geq 1$ satisfies the following inequality:
    \begin{equation}
        \lim_{t\to 0}\frac{|p\xi(-t)|^2+|p\xi(t)|^2-2|p\xi(0)|^2}{2t^2}
        \geq
        C^{-1}\frac{\sqrt{\kappa}r\cos(\sqrt{\kappa}r)}{\sin(\sqrt{\kappa}r)},
    \end{equation}
    for any point $p$ and unit-speed geodesic $\xi:[-\varepsilon,\varepsilon]\to X$ with $r:=|p\xi(0)|<\pi/\sqrt{\kappa}$.
    In particular, any complete, simply connected Berwald space with non-positive flag curvature and the uniform convexity constant bounded from above by $C\geq 1$ is $C^{-1}$-convex in the sense of \eqref{eq:k_convexity_ohta}.
\end{remark}

\begin{definition}[Busemann concave]\label{def:Busemann_concave}
    A complete geodesic space $(X,d)$ is called Busemann concave if for any pair of constant-speed geodesics $\gamma, \eta: [0,1]\rightarrow X$ starting from a common point $\gamma(0)=\eta(0)$, the function
    \begin{equation}\label{eq:Busemann_monotone_1}
        t\mapsto \frac{d(\gamma(t), \eta(t))}{t}
    \end{equation}
    is non-increasing on $[0,1]$.
\end{definition}
\begin{remark}\label{rmk:Busemann_concavity_via_comparison_triangle}
    Busemann concave spaces can be defined equivalently by comparison triangles, in a way similar to Alexandrov spaces with non-negative curvature.
    More precisely, for two constant-speed geodesics $\gamma,\eta:[0,1]\rightarrow X$ from $x$ to $y,z$ respectively, the Busemann concavity \eqref{eq:Busemann_monotone_1} asks for the following comparison inequality:
    \begin{equation}\label{eq:Busemann_concavity}
        \left|\gamma(t)\eta(t)\right|\geq \left|\tilde{\gamma}(t)\tilde{\eta}(t)\right|,\quad \text{for all }t\in [0,1],
    \end{equation}
    where $\tilde{\gamma},\tilde{\eta}:[0,1]\to \mathbb{R}^2$ are the edges of Euclidean comparison triangle $\tilde{\Delta}xyz$ from $\tilde{x}$ to $\tilde{y}$ and $\tilde{z}$ respectively.
    We note, however, that compared to the Busemann concavity, the comparison condition for Alexandrov spaces is stronger since it requires $|\gamma(t)\eta(s)|\geq |\tilde{\gamma}(t)\tilde{\eta}(s)|$ for any $t,s\in [0,1]$, rather than only for $t=s$.
\end{remark}

It follows directly from the definition that Busemann concave spaces are non-branching.
One can also see that Banach spaces with strictly convex norms and Alexandrov spaces of non-negative curvature are Busemann concave.
For further constructions and examples of Busemann concave spaces, we refer the reader to \cite[Section 2]{kell2019sectional}.

In the following, we provide some basic examples of spaces that satisfy all three conditions introduced above.
\begin{example}
    The following are $S$-concave, locally convex Busemann concave spaces:
    \begin{enumerate}
        \item Any $2$-uniformly smooth Banach space with a strictly convex norm and uniform smoothness constant $S_F\leq S$ is an $S$-concave, locally convex Busemann concave space.
        In particular, for any $p\in [2,\infty)$,  $\ell_p, \ell^n_p$ and $L_p$-space are spaces of this kind with uniform smoothness constant $S_F=p-1$.

        \item A geodesic space satisfying both $\mathrm{CBB}(0)$ and $\mathrm{CAT}(\kappa)$ conditions for some $\kappa \geq 0$ is a $1$-concave, locally convex Busemann concave space.
            
        \item The product space $Y := X \times E$, where $(X, d_X)$ is a geodesic space satisfying both the $\mathrm{CBB}(0)$ and $\mathrm{CAT}(\kappa)$ conditions for some $\kappa > 0$, and $(E, \|\cdot\|_E)$ is a $2$-uniformly smooth Banach space with strictly convex norm and the uniform smoothness constant $S > 1$, equipped with the product metric $d_Y = \sqrt{d_X^2 + \|\cdot\|_E^2}$, is also an $S$-concave, locally convex Busemann concave space.
        Note that in this case, $Y$ is not an Alexandrov space with curvature bounded below.
    \end{enumerate}
\end{example}
\begin{remark}[Berwald spaces]
    While one can show that Berwald spaces with flag curvature bounded below satisfy the Busemann concavity locally, it is unknown whether all such Berwald spaces satisfy the Busemann concavity globally, as pointed out by Kell \cite[Section 2.1]{kell2019sectional}.
    In contrast, it has been shown in \cite{kozma2004dispersing} that among connected complete Berwald spaces with absolutely homogeneous Finsler metric, non-positive flag curvature is equivalent to the global Busemann convexity (see \cite{kristaly2006metric} for equivalence between general Berwald spaces with non-positive flag curvature and the local Busemann convexity).
    We refer to \cite{kell2015Berwald} for further discussion on non-negatively curved Berwald spaces and to \cite{fujioka2025top} for discussions on the Busemann convexity.
\end{remark}

Finally, we emphasize that $S$-concave, locally semi-convex Busemann concave spaces are very different from the GNPC spaces studied in \cite{fujioka2025top}.
While these spaces are non-branching, they do not necessarily possess the local geodesic extension property, which serves as a key assumption in the study of GCBA spaces \cite{lytchak2019geod} and GNPC spaces \cite{fujioka2025top}.
This geometric condition, which requires that any geodesic segment can be extended as a local geodesic beyond its endpoints, is intimately related to the local homology at a point (see \cite[Theorem 1.5]{lytschak2004affine}), and provides superior control over both the local geometry and topological regularity of underlying spaces.
Furthermore, any pair of points in our spaces may be connected by multiple geodesics, which is in contrast to the uniquely geodesic property of GNPC spaces.

\section{Angles}\label{sect:angles}
\noindent
In this section, we introduce two notions of angle: one is well-defined on $S$-concave and locally semi-convex spaces, and one is well-defined on Busemann concave spaces. 
We then establish several fundamental properties of these angles.

The first notion of angle, referred to as \emph{angle viewed from a fixed point}, measures the orthogonality of geodesics, which serves as a key ingredient for establishing strainers and strainer maps.
The second notion of angle, referred to as \emph{angle of fixed scale}, is related to the infinitesimal structure of the underlying space and is fundamental to developing the notion of tangent cone.

Generally, there is no direct link between these two notions of angle, except on Alexandrov spaces.
This difference, as well as the asymmetry of angles, is rooted in the non-Riemannian nature of the spaces and brings a considerable challenge.

\subsection{Angles viewed from a fixed point}
In this subsection, we define the notion of angle viewed from a point, in $S$-concave and locally semi-convex spaces. 
We show that this notion of angle is well-defined in both contexts, and possesses fine properties for constructing `almost orthogonal coordinates', namely, strainer maps.

\begin{definition}[Angles viewed from a fixed point]\label{def:angle_view_from_fixed_point} 
    Let $(X,d)$ be an $S$-concave geodesic space for some $S\geq 1$ and let $p\in X$ be a point.
    Let $x\in X$ be a point different from $p$ and let $\xi$ be a unit-speed geodesic from $x$ to $y\in X$.
    The angle $\angle px\xi$ is defined as
    \begin{equation}\label{eq:angle_view_from_fixed_point}
        \angle px\xi:=\lim_{t\searrow 0}\tilde{\angle}px\xi (t),
    \end{equation}
    where $\tilde{\angle}px\xi(t)$ is the angle at ${x}$ of the Euclidean comparison triangle $\tilde{\Delta}px\xi(t)$.
    We call $\angle px \xi$ the \emph{angle viewed from $p$ at $x$ towards $y$, along the geodesic $\xi$}.
\end{definition}
\begin{remark}
    Our notation $\angle px\xi$ is slightly different from the one used in \cite{fujioka2025top} for GNPC spaces, due to the possible non-uniqueness of geodesics joining $x$ and $y$ in general.
\end{remark}

The lack of monotonicity for comparison angles in $S$-concave spaces makes it non-trivial to determine whether the limit in \eqref{eq:angle_view_from_fixed_point} exists or not.
Next, we show that, under the $S$-concavity, $\angle px\xi$ is indeed well-defined and dominates $\tilde{\angle}px\xi(t)$ up to an explicit error term. 

\begin{lemma}\label{lemma:angle_well_defined}
    Let $(X,d)$ be an $S$-concave geodesic space with $S\geq 1$.
    Let $p,x\in X$ be two different points and $\xi$ be a unit-speed geodesic starting from $x$.
    \begin{enumerate}[fullwidth]
        \item The angle $\angle px\xi$ is well-defined and satisfies the following almost comparison inequality:
        \begin{equation}\label{eq:angle_almost_comparison}
            \tilde{\angle}px\xi(t) \leq \angle px\xi + \delta_S(t;|px|),\quad t\in [0,t_0],
        \end{equation}
        where $\delta_S(t;|px|):=\arccos(1-\frac{(S-1)t}{2|px|})$ is a non-negative continuous function defined on $t\in [0,t_0]$ satisfying $\delta_S(t;|px|)\rightarrow 0$ as $t\rightarrow 0$.
        Here $t_0>0$ is a constant depending only on the distance $|px|$, the constant $S$ and the length $l(\xi)$.
        
        \item The function $t\mapsto |p\xi(t)|$ is differentiable at $t=0$ and satisfies
        \begin{equation}
            \frac{d}{dt} \bigg|_{t=0}\left|p\xi(t)\right|
            =
            -\cos\angle px\xi.
        \end{equation}
    \end{enumerate}    
\end{lemma}

\begin{proof}
    \begin{enumerate}[label*=(\arabic*), fullwidth]
        \item By applying the Euclidean law of cosines to the Euclidean comparison triangle $\tilde{\Delta}px\xi(t)$, it follows that
        \begin{equation}\label{eq:angle_well_defined_1}
            \cos \tilde{\angle}px\xi(t)
            =
            \frac{|px|^2 + t^2 -|p\xi(t)|^2}{2t|px|}\\
            =
            \frac{|px|^2-|p\xi(t)|^2 + St^2}{2t|px|} - \frac{(S-1)t}{2|px|}.
        \end{equation}
        From the $S$-concavity, it follows that the quotient $(|p\xi(t)|^2-|p\xi(0)|^2-St^2)/t$ is non-increasing in $t\in [0,l(\xi)]$.
        Therefore, $\lim_{t\searrow 0}(|px|^2-|p\xi(t)|^2+St^2)/t$ exists.
        By the monotonicity and continuity of the cosine function, it follows from the equality \eqref{eq:angle_well_defined_1} that $\lim_{t\rightarrow 0}\tilde{\angle}px\xi(t)$ exists and satisfies
        \begin{multline}
            \cos \angle px\xi=\lim_{t\searrow 0}\cos \tilde{\angle}px\xi(t)
            =
            \lim_{t\searrow 0}\frac{|px|^2-|p\xi(t)|^2 + St^2}{2t|px|}\\
            =
            \inf_{t>0}\frac{|px|^2-|p\xi(t)|^2 + St^2}{2t|px|}.
        \end{multline}
        Moreover, the monotonicity of $t\mapsto(|px|^2-|p\xi(t)|^2 + St^2)/t$ implies that
        \begin{equation}\label{eq:angle_well_defined_2}
            \cos \tilde{\angle}px\xi(t)
            =
            \frac{|px|^2-|p\xi(t)|^2 + St^2}{2t|px|} - \frac{(S-1)t}{2|px|}\\
            \geq
            \cos \angle px\xi - \tilde{\delta}_S(t/|px|),
        \end{equation}
        where $\tilde{\delta}_S(r):=(S-1)r/2$.
        Take $\delta_S(t;|px|):=\arccos(1-\tilde{\delta}_S(t/|px|))$.
        By the concavity of the cosine function, it can be deduced from the inequality \eqref{eq:angle_well_defined_2} that
        \begin{equation}
            \tilde{\angle}px\xi(t) \leq \angle px\xi + \delta_S(t;|px|),
        \end{equation}
        for any $t>0$ satisfying $t\leq l(\xi)$ and $(S-1)t/(2|px|)\leq 1$.

        \item By \eqref{eq:angle_well_defined_1}, it follows that
        \begin{multline}\label{eq:angle_well_defined_3}
            \cos\tilde{\angle}px\xi(t)
            =
            \frac{|px|^2+t^2-|p\xi(t)|^2}{2t|px|}\\
            =
            \frac{2|px|\left(|px|-|p\xi(t)|\right)}{2t|px|}+ \frac{t^2}{2t|px|}
            -\frac{(|p\xi(t)|-|px|)^2}{2t|px|}\\
            =
            \frac{|px|-|p\xi(t)|}{t} + \delta(t;\xi),
        \end{multline}
        where $\delta(t;\xi):=\frac{t}{2|px|} - \frac{(|p\xi(t)|-|px|)^2}{2t|px|}$.
        Note that from $0\leq \delta(t;\xi)\leq \frac{t}{2|px|}\to 0$ as $t\searrow 0$ and $\lim_{t\rightarrow 0}\cos\tilde{\angle}px\xi(t)$ exists, we know that $\lim_{t\rightarrow 0}\frac{|px|-|p\xi(t)|}{t}$ exists, and therefore it holds 
        \begin{equation}
            \lim_{t\searrow 0}\frac{|px|-|p\xi(t)|}{t}
            =
            \lim_{t\searrow 0}\cos \tilde{\angle}px\xi(t)
            =
            \cos\angle px\xi.
        \end{equation}
    \end{enumerate}
\end{proof}
\begin{remark}
    As shown in Lemma \ref{lemma:angle_well_defined}, $\cos\angle px\xi$ represents the directional derivative of the distance function from $p$. 
    In the case when $X$ is an Alexandrov space, the first variation formula gives that the directional derivative of the distance function from $p$ satisfies
    \begin{equation}
        \frac{d}{dt}\bigg|_{t=0}\left|p\xi(t)\right|
        =
        D_x d_p\left(\xi'(0)\right)
        =
        -\cos \min_{w\in \Uparrow_x^p}\angle(w,\xi'(0)),
    \end{equation}
    where $\xi'(0)\in \Sigma_xX$ denotes the initial direction of the geodesic $\xi$, and $\Uparrow_x^p\subset \Sigma_xX$ is the set of all initial directions of geodesics from $x$ to $p$. 
    In this case, the angle $\angle px\xi$ coincides with the minimal angle between the geodesics from $x$ to $p$ and the direction $\xi'(0)$, i.e., $\angle px\xi = \min_{w \in \Uparrow_x^p}\angle (w,\xi'(0))$. 
    Here, $\angle(\cdot,\cdot)$ denotes the angle metric on the space of directions $\Sigma_xX$.
\end{remark}

\begin{remark}
    We emphasize that the error function $\delta_S$ in the almost comparison inequality \eqref{eq:angle_almost_comparison} depends only on the ratio $t/|px|$ and the uniform smoothness constant $S$, and does not depend on $x$ and the geodesic $\xi$.
    In particular, for any $\varepsilon>0$, we can find a neighborhood $U$ of $x$ and a constant $l>0$, such that for any unit-speed geodesic $\xi$ starting from $z\in U$, it holds
    \begin{equation}
        \tilde{\angle}pz\xi(t)\leq \angle pz\xi + \delta_S(t;|pz|)< \angle px\xi + \varepsilon,\quad \text{for any }t\in [0,l].
    \end{equation}
\end{remark}

Note that angles viewed from a fixed point generally exhibit an asymmetric nature due to the non-Riemannian nature of spaces, as demonstrated by the following example:
\begin{example}[Asymmetry of angles viewed from a fixed point]\label{example:angle_asymmetry_1}
    Let $a:=(1,2)$ and $b:=(4,-1)$ be two points in $\ell^2_3:=(\mathbb{R}^2,\|\cdot\|_3)$, and let $\gamma,\eta$ be the unit-speed geodesic segments from the origin $o$ to $a$ and $b$, respectively.
    A direct computation shows that $\angle ao\eta=\pi/2$, whereas $\angle bo\gamma=\arccos(14/195^{2/3})\in (0,\pi/2)$, showing that $\angle ao\eta\neq \angle bo\gamma$.
    In particular, the vector $\vec{oa}$ is orthogonal to $\vec{ob}$ at $o$, while $\vec{ob}$ is not orthogonal to $\vec{oa}$ at $o$, in the sense of Birkhoff--James orthogonality.
    We refer to Remark \ref{rmk:asymmetry_k_strainer} for a related asymmetry phenomenon of strainers, and to \cite{alonso2022orthogonality,birkhoff1935orthogonality,james1945orthogonality,james1947orthogonality,kell2016symmetric,chmielinski2018approximate} for further discussion of the asymmetry of Birkhoff--James orthogonality in normed spaces.
\end{example}

Next we prove several fundamental properties of angles viewed from a fixed point. 
The first one is the lower semi-continuity of the map $(p,\xi)\mapsto\angle px\xi$.
\begin{lemma}[Lower semi-continuity of $\angle px\xi$]\label{lemma:angle_lsc}
    Let $(X,d)$ be an $S$-concave space with $S\geq 1$.
    Let $\{p_i\}_i\subset X$ be a sequence of points and $\{\xi_i\}_i$ be a sequence of constant-speed geodesics such that $\xi_i$ converges pointwise to a non-trivial geodesic $\xi$ and $p_i$ converges to a point $p\neq \xi(0)$.
    Reparametrize $\xi_i$ and $\xi$ by arc-length, and denote them instead by $\eta_i$ and $\eta$ respectively.
    Then it holds
    \begin{equation}
        \angle px\eta \leq \liminf_{i\rightarrow \infty} \angle p_ix_i\eta_i,
    \end{equation}
    where $x_i=\eta_i(0)$ and $x=\eta(0)$.
\end{lemma}
\begin{proof}
    Let $0<t<l(\eta)$.
    From the assumption that $\xi_i$ converges pointwisely to the geodesic $\xi$, it can be seen that $\eta_i(t)\to \eta(t)$ as $i$ goes to infinity.
    In particular, it holds that
    \begin{equation}
        |p_i\eta_i(t)|\to |p\eta(t)|,\quad |p_i\eta_i(0)|\to |p\eta(0)|,\quad |\eta_i(0)\eta_i(t)|\to |\eta(0)\eta(t)|,
    \end{equation}
    as $i$ goes to infinity.
    From the Euclidean law of cosines, it follows that $\tilde{\angle} p_ix_i\eta_i(t)$ converges to $\tilde{\angle} px\eta(t)$.
    On the other hand, by the choice of the error function $\delta_S$ in the almost comparison \eqref{eq:angle_almost_comparison}, we have $\lim_{i\rightarrow \infty}\delta_S(t;|p_ix_i|)=\delta_S(t;|px|)$.
    Therefore, by the almost comparison inequality \eqref{eq:angle_almost_comparison}, it follows that
    \begin{multline}
        \tilde{\angle}px \eta(t)
        =
        \lim_{i\to \infty}\tilde{\angle}p_ix_i\eta_i(t)
        \leq
        \liminf_{i\to \infty}\angle p_ix_i\eta_i + \lim_{i\to \infty}\delta_S(t;|p_ix_i|)\\
        =
        \liminf_{i\to \infty}\angle p_ix_i\eta_i + \delta_S(t;|px|).
    \end{multline}
    By taking $t\rightarrow 0$ on both sides of the inequality above, our claim follows.
\end{proof}

Next we show that the sum of angles viewed from a common point $p$ at $x$ towards two reversed directions along the geodesic passing through $x$ is bounded above by $\pi$.
\begin{lemma}\label{lemma:sum_of_angle_1}
    Let $(X,d)$ be an $S$-concave space with $S\geq1$.
    Let $p,x\in X$ and $\gamma:[-\varepsilon,\varepsilon]\to X$ be a unit-speed geodesic passing through $x:=\gamma(0)$, and let $\xi, \eta$ be the reparametrizations by arc-length of the geodesic segments of $\gamma$ restricted to $[0,\varepsilon]$ and $[-\varepsilon,0]$ respectively with the common starting point $\xi(0)=\eta(0)=x$.
    Then it holds
    \begin{equation}
        \angle px \xi + \angle px\eta \leq \pi.
    \end{equation}
\end{lemma}
\begin{proof}
    Let $t\in (0,\varepsilon)$.
    Applying the $S$-concave inequality \eqref{eq:S-concave} to the point $p$ and the geodesic $\gamma:[-\varepsilon,\varepsilon]\to X$, it follows that
    \begin{equation}
        \left|px\right|^2
        \geq
        \frac{1}{2}\left|p\gamma(-t)\right|^2
        + \frac{1}{2}\left|p\gamma(t)\right|^2 - \frac{S}{4}\left(2t\right)^2.
    \end{equation}
    By plugging the inequality above into the Euclidean law of cosines, we obtain that
    \begin{multline}
        \cos\tilde{\angle}px\xi(t) + \cos\tilde{\angle}px\eta(t)
        =
        \frac{t^2 + |px|^2- |p\gamma(t)|^2}{2|px|t}\\
        + \frac{t^2 + |px|^2 - |p\gamma(-t)|^2}{2|px|t}
        \geq
        \frac{(1-S)t}{|px|}.
    \end{multline}
    Letting $t\to 0$ on both sides of the inequality above, we obtain that
    \begin{equation}
         \cos\angle px\xi + \cos\angle px\eta\geq 0.
    \end{equation}
    This implies that $\angle px\xi + \angle px\eta\leq \pi$.
\end{proof}

We now discuss angles in locally semi-convex spaces (cf. Definition \ref{def:squared_convex}).
It turns out by the following lemma, similar to Lemma \ref{lemma:angle_well_defined}, that angles defined in \eqref{eq:angle_view_from_fixed_point} are also well-defined in locally semi-convex spaces if the distance $|px|$ is small enough.
\begin{lemma}\label{lemma:angle_well_defined_2}
    Let $(X, d)$ be a $(C, D)$-semi-convex space. 
    Then the angle $\angle px\xi$ in \eqref{eq:angle_view_from_fixed_point} is well-defined for any unit-speed geodesic $\xi$ starting from $x$ if $|px|<D$.
    Furthermore, it satisfies the following almost comparison inequality:
    \begin{equation}\label{eq:angle_almost_comparison_lower}
        \tilde{\angle}px\xi(t)\geq \angle px\xi - \delta_{C}(t;|px|),\quad \text{for any }t\in [0,t_0],
    \end{equation}
    where $\delta_{C}(t;|px|):=\arccos(1-(C+1)t/(2|px|))$ is a non-negative function defined on $[0,t_0]$, and $t_0>0$ is a constant depending only on $|px|$, the constants $C,D$ and the length $l(\xi)$.
    Finally, the function $t\mapsto |p\xi(t)|$ is differentiable at $t=0$ and satisfies
    \begin{equation}
        \frac{d}{dt} \bigg|_{t=0}\left|p\xi(t)\right|
        =
        -\cos\angle px\xi.
    \end{equation}
\end{lemma}
\begin{proof}
    The proof is similar to that of Lemma \ref{lemma:angle_well_defined}.
    Let $p,x\in X$ be two points such that $|px|<D$, and let $\xi:[0,l]\to X$ be an arbitrary unit-speed geodesic starting from $x$.
    By restricting $\xi$ to a smaller domain, we may assume that $\sup_{t\in [0,l]}|p\xi(t)|<D$.
    Let $t\in (0,l)$.
    By applying the $(C,D)$-semi-convexity \eqref{eq:squared-convex} to the point $p$ and the geodesic $\xi$, it follows that
    \begin{equation}
        t\mapsto \frac{|p\xi(t)|^2 + Ct^2 - |p\xi(0)|^2}{t}
    \end{equation}
    is non-decreasing on $[0,l]$.
    Therefore, the limit of $(|px|^2-|p\xi(t)|^2-Ct^2)/(2t|px|)$ as $t\searrow 0$ exists and satisfies
    \begin{equation}\label{eq:angle_squared_convex_1}
        \lim_{t\searrow 0}\frac{|px|^2-|p\xi(t)|^2 - Ct^2}{2t|px|}
        =
        \sup_{t\in (0,l)}\frac{|px|^2-|p\xi(t)|^2-Ct^2}{2t|px|}.
    \end{equation}
    By the Euclidean law of cosines, it follows that the limit of $\cos\tilde{\angle}px\xi(t)$ as $t\searrow 0$ exists and satisfies
    \begin{multline}
        \lim_{t\searrow 0}\cos\tilde{\angle}px\xi(t)
        =
        \lim_{t\searrow 0}\left(\frac{|px|^2 - |p\xi(t)|^2-Ct^2}{2t|px|} + \frac{(C+1)t}{2|px|}\right)\\
        =
        \sup_{t\in (0,l)}\frac{|px|^2-|p\xi(t)|^2-Ct^2}{2t|px|}.
    \end{multline}
    From the monotonicity of the cosine function, we can see that $\angle px\xi=\lim \tilde{\angle}px\xi(t)$ exists.
    Furthermore, by the monotonicity of $t\mapsto (|px|^2 -|p\xi(t)|^2-Ct^2)/(2t|px|)$, it follows that
    \begin{multline}
        \cos\tilde{\angle}px\xi(t)
        =
        \frac{|px|^2-|p\xi(t)|^2-Ct^2}{2t|px|} + \frac{(C+1)t}{2|px|}\\
        \leq
        \sup_{s\in (0,l)}\left(\frac{|px|^2-|p\xi(s)|^2-Cs^2}{2s|px|}\right) + \frac{(C+1)t}{2|px|}
        =
        \cos\angle px\xi + \frac{(C+1)t}{2|px|}.
    \end{multline}
    Take $\delta_{C}(t;|px|)=\arccos(1- (C+1)t/(2|px|))$.
    By the concavity of the cosine function, we obtain that
    \begin{equation}
        \tilde{\angle}px\xi(t)
        \geq
        \angle px\xi - \delta_{C}(t;|px|),
    \end{equation}
    for any $t>0$ with $t\leq l$ and $t+|px|<D$ and $(C+1)t/(2|px|)\leq 1$.
    Finally, the differentiability of $t\mapsto |p\xi(t)|$ at $t=0$ and the formula for the derivative can be shown by the same argument as in Lemma \ref{lemma:angle_well_defined}.
\end{proof}

We conclude this subsection with the following two lemmas without proof.
Their proofs are almost the same as those of Lemma \ref{lemma:angle_lsc} and Lemma \ref{lemma:sum_of_angle_1}.

\begin{lemma}[Upper semi-continuity of $\angle px\xi$]\label{lemma:angle_usc}
    Let $(X,d)$ be a $(C,D)$-semi-convex space.
    Let $\{p_i\}_i\subset X$ be a sequence of points, and $\{\xi_i\}_i$ be a sequence of constant-speed geodesics on $X$ such that $\xi_i$ converges pointwise to a non-trivial geodesic $\xi$ and $p_i$ converges to a point $p$ with $p\neq \xi(0)$ and $\sup_{i}|p_i\xi_i(0)|, |p\xi(0)|<D$.
    Let $\eta_i$ and $\eta$ be the reparametrizations of $\xi_i$ and $\xi$ by arc-length respectively.
    Then it holds
    \begin{equation}
        \angle px\eta \geq \limsup_{i\rightarrow \infty} \angle p_ix_i\eta_i,
    \end{equation}
    where $x_i=\eta_i(0)$ and $x=\eta(0)$.
\end{lemma}

\begin{lemma}\label{lemma:sum_of_angle_2}
    Let $(X,d)$ be a $(C,D)$-semi-convex space.
    Let $p\in X$ and $\gamma:[-\varepsilon,\varepsilon]\to X$ be a unit-speed geodesic passing through $x:=\gamma(0)$ such that $|px|<D$. 
    Let $\xi$ and $\eta$ be the reparametrizations by arc-length, of the geodesic segments of $\gamma$ restricted to $[0,\varepsilon]$ and $[-\varepsilon,0]$, respectively, both starting at $x$.
    Then it holds
    \begin{equation}
        \angle px \xi + \angle px\eta \geq \pi.
    \end{equation}
\end{lemma}

\subsection{Angles of fixed scale and tangent cones}
In this subsection, we introduce angles of fixed scale and tangent cones, in the setting of Busemann spaces.
After recalling several important results obtained by Kell \cite{kell2019sectional} about tangent cones, we define \emph{space of directions at a point with common length} and establish that, when the underlying Busemann concave space is doubling, spaces of directions with common length equipped with an angle $\angle_x$ are quasi-metric spaces, and are uniformly compact.

\begin{definition}[Angles of fixed scale]\label{def:angle_of_fixed_scale} 
    Let $(X,d)$ be a Busemann concave space and let $x\in X$.
    Let $\gamma, \eta$ be two non-trivial unit-speed geodesics starting from $x$.
    For any $t,s>0$, the angle $\angle_x(\gamma(t),\eta(s))$ is defined as
    \begin{equation}\label{eq:angle_of_fixed_scale}
        \angle_x(\gamma(t),\eta(s))
        :=
        \sup_{\substack{\theta\in (0,1],\\ \max\{\theta t,\theta s\}\leq a}}\tilde{\angle}_x(\gamma(\theta t),\eta(\theta s)),
    \end{equation}
    where $a>0$ is an arbitrary positive number such that $\gamma,\eta$ are both defined on $I_a:=[0,a]$.
    We call $\angle_x(\gamma(t),\eta(s))$ the \emph{angle of fixed scale}.
    In the case $t=s$, we call $\angle_x(\gamma(t),\eta(t))$ the \emph{angle of common scale}.
\end{definition}
\begin{remark}
    The definition of angles of fixed scale in \cite{fujioka2025top} is slightly different from ours in that we do not assume that $\gamma(t)$ and $\eta(s)$ are well-defined.  
    We adopt the current definition to maintain consistency with the definition of tangent cones in Busemann concave spaces.
\end{remark}

It follows from the Busemann concavity \eqref{eq:Busemann_monotone_1} that the function $\theta \mapsto \tilde{\angle}_x(\gamma(\theta t),\eta(\theta s))$ is non-increasing on the interval $(0, \min\{1,a/t,a/s\}]$.
In particular, this monotonicity implies that the value on the right-hand side of \eqref{eq:angle_of_fixed_scale} does not depend on the choice of $a$, and the supremum in \eqref{eq:angle_of_fixed_scale} is actually a limit.
Therefore, the angle $\angle_x(\gamma(t),\eta(s))$ is well-defined.

We emphasize that in contrast to the angles in Alexandrov spaces, the angle $\angle_x(\gamma(t),\eta(s))$ in Busemann concave spaces generally depends on the choice of the length parameters $t, s > 0$.
We will demonstrate that $\angle_x(\gamma(t),\eta(s))$ depends solely on the ratio $t/s$, rather than on the individual values of $t$ and $s$.
We also note that angles of fixed scale are essentially different from angles viewed from a fixed point in Definition \ref{def:angle_view_from_fixed_point}, as illustrated by the following example.

\begin{example}[Asymmetry of two notions of angle]\label{example:angle_asymmetry_2}
    Let $\gamma, \eta\subset\ell^2_p:=(\mathbb{R}^2, \|\cdot\|_p)$ be two line segments with $p \in [2,\infty)$, both starting from the origin $o$ and ending at the points $u := (0,1)$ and $v := (1,0)$, respectively.
    A direct computation shows that $\angle_o(\gamma(1), \eta(1)) = \arccos(1 - 2^{2/p-1})$, while $\angle_o(\gamma(1), \eta(1/2)) = \arccos\left(5/4 - (1 + 1/2^p)^{2/p}\right)$, and $\angle u o \eta = \lim_{t \searrow 0} \arccos\left(\frac{1 + t^2 - (1 + t^p)^{2/p}}{2t}\right)$.
    It is straightforward to verify that these three angles are generally distinct unless $p = 2$.
\end{example}

We now introduce the notion of tangent cone in the Busemann setting, following \cite{kell2019sectional}.
\begin{definition}[Pre-tangent cone and tangent cone]
    Let $(X,d)$ be a Busemann concave space and let $\Gamma_x$ denote the set of all non-trivial maximal unit-speed geodesics starting from $x\in X$.
    The pre-tangent cone at $x$, denoted by $\hat{T}_xX$, is defined as the set $\Gamma_x \times [0,\infty)/\sim$, where all points of the form $(\gamma,0)$, for $\gamma\in \Gamma_x$, are identified as a single point $o$.
    The metric $d_x$ on $\hat{T}_xX$ is defined as follows: given any $(\gamma,t),(\eta,s)\in \hat{T}_xX$, let $I_a:=[0,a]$ be an interval such that both $\gamma$ and $\eta$ are defined on $I_a$.
    The distance $d_x$ is defined as
    \begin{equation}\label{def:tangent_cone_metric_dx}
        d_x((\gamma,t),(\eta,s)):=\sup_{\substack{\theta\in (0,1],\\ \max\{\theta t,\theta s\}\leq a}}\frac{\left|\gamma(\theta t)\eta(\theta s)\right|}{\theta}.
    \end{equation}
    The tangent cone $(T_xX,d_x)$ at $x$ is defined as the completion of the pre-tangent cone $\hat{T}_xX$ with respect to the metric $d_x$.
    For $v\in T_xX$, we denote by $|v|_x:=d_x(o,v)$ the distance from $v$ to $o$.
\end{definition}

It has been shown in \citep[Lemma 2.17]{kell2019sectional} that $d_x$ is a well-defined metric on $\hat{T}_xX$, and the supremum in \eqref{def:tangent_cone_metric_dx} is in fact a limit, due to the Busemann concavity.
Furthermore, the metric $d_x$ satisfies the following positive homogeneity property:
\begin{equation}
    d_x((\gamma,\lambda t),(\eta,\lambda s))
    =\lambda d_x((\gamma,t),(\eta,s)),\quad \text{for any }\lambda>0, (\gamma,t),(\eta,s)\in \hat{T}_xX.
\end{equation}

The metric $d_x$ on the tangent cone $T_xX$ is related to the angle of fixed scale in Definition \ref{def:angle_of_fixed_scale} through the following law of cosines.
\begin{lemma}\label{lemma:relation_metric_and_angle}
    Let $(X,d)$ be a Busemann concave space and let $x\in X$.
    Let $\gamma,\eta$ be two non-trivial unit-speed geodesics, and $\bar{\gamma},\bar{\eta}$ be their maximal extensions.
    Then for any $t,s>0$, we have
    \begin{equation}\label{eq:relation_metric_and_angle}
        d_x((\bar{\gamma},t),(\bar{\eta},s))^2 = t^2 + s^2 - 2ts\cos\angle_x(\gamma(t),\eta(s)).
    \end{equation}
    Moreover, the angle of fixed scale is positive scaling-invariant; that is, for any $\lambda>0$,
    \begin{equation}
        \angle_x(\gamma(\lambda t), \eta(\lambda s)) = \angle_x(\gamma(t),\eta(s)).
    \end{equation}
    In particular, the angle $\angle_x(\gamma(t),\eta(s))$ depends only on the ratio $t/s$.
\end{lemma}
\begin{proof}
    Let $a>0$ be a positive number such that $\gamma$ and $\eta$ are both defined on the interval $I_a=[0,a]$.
    Note that from the definition, the angle $\angle_x(\bar{\gamma}(t),\bar{\eta}(s))$ actually coincides with $\angle_x(\gamma(t),\eta(s))$, since $\tilde{\angle}_x(\bar{\gamma}(\theta t), \bar{\eta}(\theta s))=\tilde{\angle}_x(\gamma(\theta t),\eta(\theta s))$ when $\theta>0$ is small enough.
    For sufficiently small $\theta>0$, it follows by the Euclidean law of cosines that
    \begin{equation}\label{eq:relation_metric_and_angle_1}
        \left|\bar{\gamma}(\theta t)\bar{\eta}(\theta s)\right|^2
        =
        \left(\theta t\right)^2 + \left(\theta s\right)^2 - 2 \theta^2 ts \cos\tilde{\angle}_x(\gamma(\theta t), \eta(\theta s)).
    \end{equation}
    Dividing both sides of the equality above by $\theta^2$ and letting $\theta\searrow 0$, we obtain the equality \eqref{eq:relation_metric_and_angle}.
    The second claim follows from the positive homogeneity of $d_x$ and \eqref{eq:relation_metric_and_angle}.
\end{proof}

We now discuss the relationship between the tangent cone $(T_xX, d_x)$ and the pointed Gromov--Hausdorff limit of the blow-ups $\{(X,d/\lambda,x)\}_{\lambda>0}$.
It is important to note that, in general, the tangent cone $(T_xX, d_x)$ of a Busemann concave space $X$ does not necessarily coincide with the pointed Gromov--Hausdorff limit of the blow-ups $\{(X, d/\lambda, x)\}_{\lambda > 0}$ at $x$, as discussed in \citep[Section 2.3]{kell2019sectional}. 
However, the following proposition, due to \citep[Lemma 2.20, Corollary 2.21]{kell2019sectional}, shows that if the Busemann concave space $X$ is doubling, then the tangent cone and the pointed Gromov--Hausdorff limit of the blow-ups at $x$ do indeed coincide.

\begin{proposition}[Kell, \citep{kell2019sectional}]\label{prop:tangent_cone_busemann}
    Let $(X,d)$ be a Busemann concave space.
    If $X$ is (locally) doubling, then the tangent cone $(T_xX, d_x)$ at $x$ is locally compact and coincides with the unique pointed Gromov--Hausdorff limit of the blow-ups $\{(X,d/\lambda, x)\}_{\lambda}$ as $\lambda\to 0$.
\end{proposition}

The fact that $\angle_x(\gamma(t),\eta(s))$ depends on the ratio $t/s$ prevents a direct identification of the space of directions at $x$ with the set of equivalence classes of geodesics emanating from $x$ equipped with the angle metric, as in Alexandrov spaces. 
Consequently, the structure of tangent cones in Busemann concave spaces is more intricate than the metric cone over the space of directions, reflecting the non-Riemannian nature of Busemann concave spaces.
In the following, we introduce a subset of the tangent cone, referred to as the \emph{space of directions with common length}.
This subset carries a natural \emph{angle quasi-metric}, defined via the Euclidean law of cosines on tangent cones, and this quasi-metric agrees with the angle of fixed scale on a dense subset of it.

\begin{definition}[Space of directions with common length]\label{def:space_directions_common_length}
    Let $(X,d)$ be a Busemann concave space and $x\in X$ be a point.
    Given $l>0$, we denote by $\hat{\Sigma}^l_xX$ the subset of the pre-tangent cone $\hat{T}_xX$ consisting of elements of the form $(\gamma,l)$, and by $\Sigma^l_xX$ the {\em space of directions with common length at $x$}, defined as the closed subset of the tangent cone $T_xX$ consisting of elements $v$ whose distance from $o$ equals $l$.
    We define $\angle_x(\cdot,\cdot):\Sigma^l_xX\times \Sigma^l_xX\to [0,\pi]$ as:
    \begin{equation}
        \angle_x(v,w):=\arccos\left(\frac{|v|_x^2 + |w|_x^2 - d_x(v,w)^2}{2|v|_x|w|_x}\right),\quad \text{for any } v,w\in \Sigma^l_xX.
    \end{equation}
\end{definition}
From Lemma \ref{lemma:relation_metric_and_angle}, it is clear that $\angle_x((\gamma,l),(\eta,l))=\angle_x(\gamma(l),\eta(l))$ for any $(\gamma,l),(\eta,l)\in \hat{\Sigma}^l_xX$.

\begin{lemma}\label{lemma:angle_tangent_cone_scaling_invariant}
    The space $(\Sigma^l_xX,\angle_x)$ is a quasi-metric space; that is, $\angle_x$ satisfies the following properties: $\angle_x(v,v)=0$ for $v\in \Sigma^l_xX$, $\angle_x(v,w)=0$ for $v,w\in \Sigma^l_xX$ implies $v=w$, $\angle_x$ is symmetric and satisfies the following quasi-triangle inequality:
    \begin{equation}
        \angle_x(u,v)\leq 2\left(\angle_x(u,w) + \angle_x(w,v)\right),\quad \text{for any } u,v,w\in \Sigma^l_xX.
    \end{equation}
    Furthermore, for any $v\in \Sigma^l_xX$, there exists a sequence $\{v_n\}_n\subset \hat{\Sigma}^l_xX$ such that $v_n$ converges to $v$ with respect to $\angle_x$.
    Finally, $\angle_x$ is positive scaling-invariant; that is,
    \begin{equation}\label{eq:angle_tangent_cone_scaling_invariant}
        \angle_x((\gamma,\lambda l),(\eta, \lambda l))
        =
        \angle_x((\gamma,l),(\eta,l)),
    \end{equation}
    where we identify $\angle_x$ on the left-hand side of \eqref{eq:angle_tangent_cone_scaling_invariant} as a quasi-metric on $\hat{\Sigma}^{\lambda l}_xX$.
\end{lemma}
\begin{proof}
    Let $l>0$ be arbitrary.
    We first show that $\angle_x$ is a quasi-metric on $\Sigma^l_xX$.
    Firstly, the symmetry just follows from the definition of $\angle_x$.
    Secondly, it is clear from the definition that $\angle_x(v,v)=0$ for any $v\in \Sigma^l_xX$, and $\angle_x(v,w)=0$ for $v,w\in \Sigma^l_xX$ implies that $d_x(v,w)=0$, which implies that $v=w$.
    Finally, for the quasi-triangle inequality, let $u_i\in \Sigma^l_xX$ for $i=1,2,3$.
    It follows from the definition of $\angle_x$ that $d_x(u_i,u_j)=2l\sin(\angle_x(u_i,u_j)/2)$ for $i,j=1,2,3$.
    By the triangle inequality of $d_x$, we have that
    \begin{equation}
        \sin\frac{\angle_x (u_1,u_2)}{2}
        \leq
        \sin\frac{\angle_x(u_1,u_3)}{2} + \sin\frac{\angle_x(u_3,u_2)}{2}.
    \end{equation}
    By applying the inequalities $1/2\leq \sin t/t\leq 1$ for $t\in [0,\pi/2]$, we obtain that
    \begin{equation}
        \frac{1}{2}\angle_x(u_1,u_2)\leq \angle_x(u_1,u_3)+ \angle_x(u_3,u_2),
    \end{equation}
    which implies the quasi-triangle inequality.
    For the second claim, let $v\in \Sigma^l_xX$.
    Since $\Sigma^l_xX$ is a closed subset of $T_xX$, it follows from the definition of $T_xX$ that there exists a sequence $\{(\gamma_n,l_n)\}_n$ converging to $v$ in $d_x$.
    Note that this implies that $l_n\to l$, and therefore by the triangle inequality of $d_x$, we have that $d_x((\gamma_n,l),v)\to 0$.
    Let $v_n:=(\gamma_n,l)\in \hat{\Sigma}^l_xX$.
    From the definition of $\angle_x$, it follows that
    \begin{equation}
        d_x(v_n,v)^2= 2l^2 -2l^2\cos\angle_x(v_n,v).
    \end{equation}
    Since $v_n$ converges to $v$ in $d_x$, we have that $\angle_x(v_n,v)\to 0$.
    Finally, the positive scaling-invariance follows from the fact that $\angle_x$ restricted to $\hat{\Sigma}^l_xX$ coincides with the angle of common scale and the positive scaling-invariance of angles of fixed scale (Lemma \ref{lemma:relation_metric_and_angle}).
\end{proof}

We conclude this subsection by showing that the family of spaces of directions at $x$ with common length, $\{(\Sigma^l_xX, \angle_x)\}_{l>0}$, is uniformly compact whenever the Busemann concave space $(X, d)$ is doubling. 
Our proof follows a similar strategy to that of \citep[Proposition 10.9.1]{burago2001course} for the spaces of directions in Alexandrov spaces.

\begin{lemma}\label{lemma:space_direction_uniform_compact}
    Let $(X,d)$ be a doubling Busemann concave space and let $x\in X$.
    Then the family $\{(\Sigma^l_xX, \angle_x)\}_{l>0}$ is uniformly compact in the sense that for any $\varepsilon > 0$, there exists $N_0(\varepsilon) > 0$, depending only on the doubling constant of $X$ and $\varepsilon$, such that every $\varepsilon$-separated subset\footnote{A subset $A$ in a quasi-metric space $(Y,d)$ is called $\varepsilon$-separated if $d(x,y)\geq \varepsilon$ for any $x,y\in A$.} of $(\Sigma^l_xX, \angle_x)$ contains at most $N_0(\varepsilon)$ elements.
\end{lemma}
\begin{proof}
    For any $r > 0$ and $\varepsilon > 0$, let $N(\varepsilon) > 0$ denote the maximal cardinality of any $\varepsilon r$-separated subset of a ball of radius $r$ in $X$.
    We claim that the cardinality of any $\varepsilon$-separated set in $\Sigma^l_xX$ is at most $N(\varepsilon/32)$ for any $l>0$.
    Indeed, let $\{v_i\}_{i=1}^m\subset \Sigma^l_xX$ be an $\varepsilon$-separated set in $(\Sigma^l_xX, \angle_x)$.
    By Lemma \ref{lemma:angle_tangent_cone_scaling_invariant}, we can find $w_i:=(\gamma_i,l)\in \hat{\Sigma}^l_xX$ such that $\angle_x(v_i,w_i)<\varepsilon/16$ for each $i=1,\ldots,m$.
    By the quasi-triangle inequality of $\angle_x$, it follows that
    \begin{equation}
        \angle_x(w_i,w_j)
        \geq
        \frac{1}{4}\angle_x(v_i,v_j) - \angle_x(v_i,w_i)- \angle_x(v_j,w_j)
        \geq
        \frac{1}{8}\varepsilon,\quad \text{for any }1\leq i<j\leq m.
    \end{equation} 
    Note that $\angle_x$ restricted to $\hat{\Sigma}^l_xX$ coincides with the angle of common scale.
    Thus, it follows that we can find a sufficiently small number $\theta\in (0,1)$ such that
    \begin{equation}
        \tilde{\angle}_x(\gamma_i(\theta l),\gamma_j(\theta l))
        >
        \frac{1}{2}\angle_x\left((\gamma_i,l),(\gamma_j,l)\right)
        \geq
        \frac{\varepsilon}{16},\quad \text{for all } 1\leq i<j\leq m.
    \end{equation}
    Let $t_l:=\theta l$.
    Then from the elementary geometry of triangles, it follows that
    \begin{equation}
        \left|\gamma_i(t_l)\gamma_j(t_l)\right|
        \geq
        2t_l \sin\frac{\tilde{\angle}_x(\gamma_i(t_l),\gamma_j(t_l))}{2}
        \geq
        \frac{\varepsilon}{32}t_l,\quad \text{for any }1\leq i<j\leq m.
    \end{equation}
    This implies that the family $\{\gamma_i(t_l)\}_{i=1}^m$ is an $\varepsilon t_l/32$-separated subset of the $t_l$-ball centered at $x$ in $X$.
    By the doubling property of $X$, it follows that $m\leq N(\varepsilon/32)$.
    The claim follows by choosing the constant $N_0(\varepsilon):=N(\varepsilon/32)$.
\end{proof}

\section{Strainers and strainer maps}\label{sect:strainer_maps}
\noindent
In this section, we develop strainers and strainer maps for $S$-concave, locally semi-convex spaces, which can be regarded as `almost orthogonal coordinates' and play a key role in studying the structure theory.


\subsection{Definitions}

\begin{definition}[$(1,\delta)$-strainer]\label{def:1-strainer}
    Let $X$ be an $S$-concave space and let $x \in X$.
    Given $0<\delta <1/2$, a point $p \in X\setminus \{x\}$ is called a \emph{$(1,\delta)$-strainer at $x$} if there exists a point $q \in X \setminus \{x\}$ such that the Euclidean comparison angle $\tilde{\angle}pxq>\pi-\delta$ and that
    \begin{equation}\label{eq:def_strainer_1}
        \bar{\delta}_S(|qx|;|px|):=\arccos\left(1-S\frac{|qx|}{2|px|}\right)<\delta.
    \end{equation}
    In this case, we refer to the point $q$ as an \emph{opposite strainer} of $p$ at $x$, and to the pair $(p,q)$ as a \emph{$(1,\delta)$-strainer pair}.
    We say that $x$ is a \emph{$(1,\delta)$-strained point} if it admits a $(1,\delta)$-strainer at itself.
\end{definition}
\begin{remark}
    Our definition of strainer differs from the original one in \cite{burago1992ad} for Alexandrov spaces with curvature bounded below, by imposing an additional control over the error function with respect to the ratio of distances $|qx|/|px|$.
    We adopt such a modified version to overcome the absence of monotonicity of comparison angles.
\end{remark}


The next lemma connects $(1,\delta)$-strainers and angles viewed from a fixed point.
\begin{lemma}\label{lemma:1_strainer_property}
    Let $X$ be an $S$-concave space with $S\geq 1$.
    If the angle viewed from $p\in X$ at $x\in X$ along the unit-speed geodesic $\xi$ satisfies that $\angle px\xi>\pi-\delta$ for some small $\delta>0$, then $p$ is a $(1,\delta)$-strainer at $x$ with an opposite strainer $q$ which can be selected arbitrarily close to $x$.
    Conversely, if $p$ is a $(1,\delta)$-strainer at $x$ with $q$ as an opposite strainer, then for any unit-speed geodesic $\xi$ from $x$ to $q$, it holds
    \begin{equation}
        \angle px\xi > \pi - 2\delta.
    \end{equation}
\end{lemma}
\begin{proof}
    By the definition of $\angle px\xi$, we can find $t_0>0$ such that
    \begin{equation}
        \tilde{\angle}px\xi(t)>\pi -\delta,\quad \text{for any }t\in (0,t_0).
    \end{equation}
    By taking a smaller $t_0$ if necessary, we can assume that $\bar{\delta}_S(t;|px|)<\delta$ for any $t\in (0,t_0)$.
    This implies that $p$ is a $(1,\delta)$-strainer at $x$ with $\xi(t)$ as an opposite strainer for any $t\in (0,t_0)$.
    For the second statement, let $\xi$ be an arbitrary unit-speed geodesic from $x$ to $q$, and let $l:=l(\xi)=|qx|$.
    Then it follows from the almost comparison inequality \eqref{eq:angle_almost_comparison} that
    \begin{equation}
        \angle px\xi 
        \geq 
        \tilde{\angle}px\xi(l) - \delta_S(l;|px|)
        >
        \tilde{\angle}pxq - \bar{\delta}_S(|qx|;|px|)
        >
        \pi - 2\delta.
    \end{equation}
\end{proof}
Next we show that the set of $(1,\delta)$-strained points is open in an $S$-concave space.

\begin{lemma}\label{lemma:1_strainer_open}
    Let $X$ be an $S$-concave space.
    Let $x\in X$ be a $(1,\delta)$-strained point at which $p$ is a $(1,\delta)$-strainer with $q$ as an opposite strainer.
    Then there exists an open neighborhood $U$ of $x$ at each point of which $p$ is a $(1,\delta)$-strainer with $q$ as an opposite strainer.
\end{lemma}
\begin{proof}
    The assertion follows from the continuity of Euclidean comparison angles and the function $\bar{\delta}_S$.
    Indeed, since $\lim_{y\to x}\tilde{\angle}pyq=\tilde{\angle}pxq$, we can find a neighborhood $U$ of $x$ such that $\tilde{\angle}pzq>\pi-\delta$ for any $z\in U$.
    On the other hand, since the function $z\mapsto \bar{\delta}_S(|qz|;|pz|)$ is continuous provided $z\neq p$, we can shrink $U$ to a smaller neighborhood of $x$ if necessary so that $\bar{\delta}_S(|qz|;|pz|)<\delta$ for all $z\in U$.
    Thus, the open subset $U$ is the desired open neighborhood.
\end{proof}

Next, using an inductive procedure similar to \citep[Definition 5.2]{fujioka2025top}, we define $(k,\delta)$-strainers on $S$-concave spaces that further satisfy the local semi-convexity.

\begin{definition}[$(k,\delta)$-strainer]\label{def:k_strainer}
    Let $X$ be an $S$-concave, $(C,D)$-semi-convex space with $S\geq 1, C\geq 0$ and $D>0$, and let $0<\delta<1/2$.
    Given $k\geq 1$, we call a $k$-tuple $(p_1,\ldots,p_k)$ of points in $X$ a \emph{$(k,\delta)$-strainer} at $x$ if $|p_1 x|<D$ and the following inductive conditions hold:
    \begin{enumerate}
        \item\label{cond:k_strainer_1} $(p_1,\cdots,p_{k-1})$ is a $(k-1,\delta)$-strainer at $x$.
        
        \item\label{cond:k_strainer_2} $p_k$ is a $(1,\delta)$-strainer at $x$, with $\bar{\delta}_{S,C}(|p_kx|;|p_ix|)<\delta$ for $i=1,\ldots,k-1$, where $\bar{\delta}_{S,C}(r;t):=\arccos(1-(S+C)r/(2t))$.
        
        \item\label{cond:k_strainer_3}There exists an opposite $(1,\delta)$-strainer $q_k$ of $p_k$ in the sense of Definition \ref{def:1-strainer} such that
        \begin{equation}\label{eq:k_strainer_1}
            \left|\tilde{\angle}p_ixp_k - \pi/2\right|<\delta,\quad
            \left|\tilde{\angle} p_ixq_k - \pi/2\right|<\delta,
        \end{equation}
        for $i=1,\cdots, k-1$.
    \end{enumerate}
    In this case, we call the $k$-tuple $(q_1,\cdots,q_k)$ an \emph{opposite $(k,\delta)$-strainer} of $(p_1,\cdots,p_k)$ at $x$.
    A point $x$ is said to be a \emph{$(k,\delta)$-strained point} if it admits a $(k,\delta)$-strainer.
    If $(p_1,\cdots,p_k)$ is a $(k,\delta)$-strainer at each point of a subset $U$, we call the map
    \begin{equation}
        f(\cdot):=(d_{p_1}(\cdot),\cdots,d_{p_k}(\cdot)):U\subset X\to \mathbb{R}^k
    \end{equation}
    a $(k,\delta)$-strainer map on $U$ associated with $(p_1,\cdots,p_k)$, where $d_p(\cdot):=d(p,\cdot)$.
\end{definition}
\noindent
We emphasize that for the case $k=1$, we always assume that the $(1,\delta)$-strainer $p_1$ at $x$ on an $S$-concave, $(C,D)$-semi-convex space satisfies $|p_1x|<D$.
\begin{remark}
    Our definition of $(k,\delta)$-strainers differs from the definition in \cite[Definition 5.2]{fujioka2025top} for GNPC spaces, in which the almost orthogonality condition \eqref{eq:k_strainer_1} is imposed on angles viewed from a fixed point, rather than on comparison angles.

    Our definition also differs from \cite[Definition 5.2]{burago1992ad} for Alexandrov spaces with curvature bounded below. 
    Besides an inductive procedure, we impose an additional control over distances from strainers to address the absence of monotonicity of comparison angles. 
    We also impose upper bound controls on comparison angles in the almost orthogonality condition \eqref{eq:k_strainer_1} to address the absence of the quadruple condition (see \cite[Proposition 10.1.1]{burago2001course}).
    This quadruple condition, which is an equivalent characterization for Alexandrov spaces with curvature bounded below, automatically yields the upper bounds on comparison angles from the lower bounds.
\end{remark}

The next lemma demonstrates the `almost orthogonality' of $(k,\delta)$-strainers in terms of angles viewed from the strainer points.
\begin{lemma}\label{lemma:orthogonality_k_strainer}
    Let $X$ be an $S$-concave, $(C,D)$-semi-convex space.
    Given $0<\delta<1/2$, let $(p_1,\cdots,p_k)$ be a $(k,\delta)$-strainer at $x\in X$ in the sense of Definition \ref{def:k_strainer} with an opposite strainer $(q_1,\cdots,q_k)$.
    Then for any $1\leq i<j\leq k$ and any unit-speed geodesic $\xi_j,\eta_j$ from $x$ to $p_j$ and $q_j$ respectively, the following almost orthogonality holds:
    \begin{equation}\label{eq:orthogonality_k_strainer}
        \left|\angle p_ix\xi_j - \pi/2\right|<2\delta,\quad
        \left|\angle p_i x \eta_j -\pi/2\right|<2\delta.
    \end{equation}
\end{lemma}
\begin{proof}
    Note that $\sup_{i=1,\cdots,k}|p_ix|<D$.
    Indeed, the definition of $(k,\delta)$-strainer implies that $|p_1x|<D$.
    This fact, together with the inequality
    \begin{equation}\label{eq:orthogonality_k_strainer_1}
        \frac{|p_jx|}{2|p_1x|}
        \leq
        \arccos\left(1-(S+C)\frac{|p_jx|}{2|p_1x|}\right)
        <\delta<\frac{1}{2},\quad \text{for any } j=2,\ldots,k,
    \end{equation}
    implies that $|p_jx|<|p_1x|<D$.
    Thus, the claim follows directly from the almost comparison inequality in Lemma \ref{lemma:angle_well_defined}, Lemma \ref{lemma:angle_well_defined_2} and the assumption that $\bar{\delta}_{S,C}(|p_jx|;|p_ix|)<\delta$ for all $1\leq i<j\leq k$.
\end{proof}
\begin{remark}[Asymmetry of strainers and strainer maps]\label{rmk:asymmetry_k_strainer}
    We remark that $(k,\delta)$-strainers exhibit asymmetry in two aspects.
    On the one hand, a $k$-tuple $(p_1,\ldots,p_j,\ldots,p_i,\ldots,p_k)$ reordered from a $(k,\delta)$-strainer $(p_1,\ldots,p_k)$ is no longer a $(k,\delta)$-strainer due to the inductive nature of Definition \ref{def:k_strainer}.
    On the other hand, the associated strainer map $f=(d_{p_1},\ldots,d_{p_k})$, which serves as an `almost orthogonal coordinate chart' in a small neighborhood of $x$, only controls the orthogonality of each coordinate with respect to those of lower indices: while the $i$-th coordinate $|p_ix|$ remains nearly unchanged as $x$ moves along the $k$-th coordinate towards either $p_k$ or $q_k$ (due to the almost orthogonality \eqref{eq:orthogonality_k_strainer}), there is no corresponding control over the $k$-th coordinate $|p_kx|$ when $x$ approaches the point $p_i$. 
    This asymmetry stems from the asymmetric nature of angles viewed from a fixed point, see Example \ref{example:angle_asymmetry_1}.
    For similar phenomena of asymmetry, we refer to \cite{fujioka2025top} in GNPC spaces, and to \cite{alonso2022orthogonality,birkhoff1935orthogonality,james1945orthogonality,james1947orthogonality,kell2016symmetric,chmielinski2018approximate} about the Birkhoff--James orthogonality in normed spaces.
\end{remark}

We conclude this subsection with a lemma being analogous to Lemma \ref{lemma:1_strainer_open} without proof.
\begin{lemma}\label{lemma:strained_point_openness}
    Let $X$ be an $S$-concave, locally semi-convex space.
    Let a $k$-tuple $(p_1,\ldots,p_k)$ be a $(k,\delta)$-strainer at $x$ with $(q_1,\ldots,q_k)$ as an opposite strainer.
    Then there exists a neighborhood $U$ of $x$ at each point of which $(p_1,\ldots,p_k)$ is a $(k,\delta)$-strainer with $(q_1,\ldots,q_k)$ as an opposite strainer. 
\end{lemma}

\subsection{Openness and bi-Lipschitz of strainer maps}
We begin by recalling the following definition of $\varepsilon$-open maps.
For further discussion, see \cite{lytchak2006open, lytchak2019geod,guy_c_david2015GAFA}.
\begin{definition}[$\varepsilon$-open maps]\label{def:epsilon_open_map}
    A Lipschitz map $F: U\to Y$ from an open subset $U\subset X$ to a metric space $Y$ is said to be \emph{$\varepsilon$-open} if for any point $x\in U$, there exists $r>0$ such that the closed ball $\bar{B}(x,\varepsilon^{-1}r)\subset U$ is complete, and for any $v\in B(F(x),r)\subset Y$, there exists a point $y\in U$ such that $F(y)=v$ and $\varepsilon|yx|\leq |F(x)v|$.
    In particular, for any $s\in (0,r]$, we have the inclusion $B(F(x),s)\subset F(B(x,\varepsilon^{-1}s))$.
\end{definition}

We need the following useful criterion for $\varepsilon$-open maps.
This criterion is classical and is nearly identical to \citep[Lemma 5.15]{fujioka2025top}, with an exception that here $X$ is assumed to be locally complete rather than locally compact.
For the sake of completeness, we provide a proof in Appendix \ref{appendix}.
\begin{lemma}[criterion for $\varepsilon$-open maps]\label{lemma:criterion_open_map}
    Let $f:X\rightarrow Y$ be a locally Lipschitz map from a locally complete\footnote{A metric space $(X,d)$ is said to be locally complete if every point $x\in X$ admits a complete neighborhood, see \cite{lytchak2006open}.} metric space $X$ to a geodesic space $Y$.
    Suppose there exists $\varepsilon>0$ such that the following holds: for every $x\in X$ and every $v\in Y\setminus \{f(x)\}$ sufficiently close to $f(x)$, there exists $y\in X$ such that
    \begin{equation}\label{eq:criterion_open_map}
        |f(y)v|-|f(x)v| \leq -\varepsilon|xy|.
    \end{equation}
    Then $f$ is an $\varepsilon'$-open map for any $0<\varepsilon'<\varepsilon$.
\end{lemma}
\begin{remark}
    In \cite[Lemma 5.15]{fujioka2025top}, the endpoint case $\varepsilon' = \varepsilon$ can be achieved. 
    We cannot reach the endpoint case due to the lack of local compactness.
\end{remark}

Next, we show that the $(k,\delta)$-strainer map is $\varepsilon$-open for $\delta>0$ small enough.
The main difficulty lies in the asymmetry of $(k,\delta)$-strainer maps (see Remark \ref{rmk:asymmetry_k_strainer}).
This drawback prevents us from applying the standard procedure to show the $\varepsilon$-openness of strainer maps.
To overcome this difficulty, we follow an idea similar to \citep[Proposition 5.17]{fujioka2025top} and introduce an anisotropic $\ell_1$-norm for the target space so that the asymmetry can be made up for.
\begin{proposition}\label{prop:strainer_map_open}
    Let $X$ be an $S$-concave, locally semi-convex space.
    There exists a constant $\delta_k:=k^{-1}2^{-2k-1}\in (0,1/2)$, such that for any $\delta<\delta_k$, any $(k,\delta)$-strainer map $f:U\to \mathbb{R}^k$ from an open subset $U\subset X$ to $\mathbb{R}^k$ equipped with the $\ell_1$-norm is $\varepsilon_k(\delta):=\frac{(1-2\delta)}{4^{k-1}}$-open.
\end{proposition}
\begin{proof}
    Let $(p_1,\ldots,p_k)$ be the associated $(k,\delta)$-strainer of $f$. 
    Denote $f_i:=d_{p_i}$ and $f_{[i]}:=(f_1,\ldots,f_i)$ for $i=1,\ldots,k$.
    We show the statement by induction.
    \begin{enumerate}[label*=\textbf{\textsc{Step \arabic*:}}, fullwidth]
        \item We first show that $f_1$ is $\varepsilon_1(\delta)$-open.
        Let $x\in U$ and $q_1$ be an opposite $(1,\delta)$-strainer of $p_1$ at $x$.
        Let $v\in \mathbb{R}\setminus f(x)$ be such that $|v-f_1(x)|<r_x$, where $r_x\in (0,R_x)$ is a small constant to be determined later, and $R_x$ is the supremum of all radii $r>0$ such that $B(x,2r)\subset U$.
        Let $\eta,\xi$ be two unit-speed geodesics from $x$ to $p_1$ and $q_1$, respectively. 
        By Lemma \ref{lemma:1_strainer_property}, we have $\angle p_1x\xi>\pi-2\delta$.
        Combining this with Lemma \ref{lemma:angle_well_defined}, it holds that
        \begin{equation}
            \lim_{t\searrow 0}\frac{|p_1\xi(t)|-|p_1x|}{t}
            =
            -\cos\angle p_1x\xi
            >
            \cos2\delta> 1-2\delta.
        \end{equation}
        So we can find $t_0>0$ such that $f_1(\xi(t))=|p_1\xi(t)|\geq f_1(x) + (1-2\delta)t$ for all $t\in [0,t_0]$.
        Therefore, if $f_1(x)<v<f_1(x)+(1-2\delta)t_0$, by continuity we can find $t\in (0,t_0)$ such that $v=f_1(\xi(t))$.
        Taking $y=\xi(t)$, we have $v=f_1(y)$ and
        \begin{equation}
            (1-2\delta)|xy|\leq |p_1y|-|p_1x|=\left|f_1(y)-f_1(x)\right|.
        \end{equation}
        In the case that $f_1(x)-|p_1x|/2<v<f_1(x)$, there exists $y=\eta(t)$ with $t\in (0,|p_1x|/2)$, such that $v=f_1(y)$ and $|yx|= f_1(x)-f_1(y)$.
        In conclusion, we can find a point $y$ in either $\xi$ or $\eta$ such that $v=f_1(y)$ and
        \begin{equation}
            (1-2\delta)|xy|\leq |f_1(x)-v|,
        \end{equation}
        if $|v-f_1(x)|<r_x:=\min\{t_0,|px|/2,R_x\}$.
        This shows that $f_1$ is a $(1-2\delta)$-open map from $U$ to $\mathbb{R}$.
        We can take $\varepsilon_1(\delta):=1-2\delta$ for $\delta<\delta_1:=1/8$.

        \item Suppose that the $(k-1,\delta)$-strainer map $f_{[k-1]}$ is an $\varepsilon_{k-1}(\delta)$-open map for $\delta<\delta_{k-1}$ from $U$ to $\mathbb{R}^{k-1}$ equipped with the $\ell_1$-norm, where $\varepsilon_{k-1}(\delta):=(1-2\delta)/4^{k-2}$.
        For simplicity, we denote $\varepsilon_{k-1}(\delta)$ by $\varepsilon_{k-1}$.
        Note that from Step 1, it holds that $f_k$ is a $(1-2\delta)$-open map.
        
        We aim to apply Lemma \ref{lemma:criterion_open_map}. 
        To do that, we first introduce the following anisotropic norm $\|\cdot\|$ on $\mathbb{R}^k$:
        \begin{equation}
            \| v \| := \left\| v_{[k-1]}\right\|_1 + \frac{\varepsilon_{k-1}}{2}|v_k|,\quad v=(v_1,\ldots,v_k)\in \mathbb{R}^k,
        \end{equation}
        where $v_{[k-1]}:=(v_1,\ldots,v_{k-1})\in \mathbb{R}^{k-1}$ is the first $(k-1)$-coordinates of $v$, and $\|\cdot\|_1$ denotes the $\ell_1$-norm on the Euclidean space.
        Let $x\in U$ be fixed, and let $(q_1,\ldots,q_k)$ be an opposite $(k,\delta)$-strainer of $(p_1,\ldots,p_k)$ at $x$.
        It suffices to find $y\in U$ satisfying the inequality \eqref{eq:criterion_open_map} whenever $v\in \mathbb{R}^k\setminus \{f(x)\}$ is sufficiently close to $f(x)$.
        We consider the following two cases:
    
        \begin{enumerate}[label=\textbf{Case \arabic*:}, fullwidth]
            \item $f_{[k-1]}(x)\neq v_{[k-1]}$.
            In this case, by the assumption that $f_{[k-1]}$ is an $\varepsilon_{k-1}$-open map, it follows that we can find $y\in U$ such that
            \begin{equation}
                f_{[k-1]}(y)=v_{[k-1]},\quad \varepsilon_{k-1}|xy|\leq \|f_{[k-1]}(x)-v_{[k-1]}\|_1.
            \end{equation}
            Combining with the triangle inequality, this implies that
            \begin{multline}
                \left\|f(y)-v\right\| - \left\|f(x)-v\right\|
                =
                \left\|f_{[k-1]}(y)-v_{[k-1]}\right\|_1 +\frac{\varepsilon_{k-1}}{2}|f_k(y)-v_k|\\
                -
                \left(\left\|f_{[k-1]}(x)-v_{[k-1]}\right\|_1 + \frac{\varepsilon_{k-1}}{2}|f_k(x)-v_k|\right)\\
                \leq
                \frac{\varepsilon_{k-1}}{2}|f_k(x)-f_k(y)| - \varepsilon_{k-1}|xy|
                \leq
                -\frac{\varepsilon_{k-1}}{2}|xy|.
            \end{multline}

            \item $f_k(x)\neq v_k$. 
            In this case, note that from Step 1, we can find $y\in U$ lying on either a geodesic $\xi_k$ from $x$ to $p_k$ or a geodesic $\eta_k$ from $x$ to $q_k$ such that
            \begin{equation}\label{eq:strainer_map_open_0}
                f_k(y)=v_k,\quad (1-2\delta)|xy|\leq \left|f_k(x)-v_k \right|.
            \end{equation}
            This implies that
            \begin{multline}\label{eq:strainer_map_open_2}
                \|f(y)-v\| - \|f(x)-v\|
                =
                \left\|f_{[k-1]}(y) - v_{[k-1]}\right\|_1 +\frac{\varepsilon_{k-1}}{2}|f_k(y) - v_k|\\
                -
                \left(\left\|f_{[k-1]}(x) - v_{[k-1]}\right\|_1 + \frac{\varepsilon_{k-1}}{2}|f_k(x) - v_k|\right)\\
                \leq
                \left\|f_{[k-1]}(y)-f_{[k-1]}(x)\right\|_1 - \frac{\varepsilon_{k-1}}{2}(1-2\delta)|xy|,
            \end{multline}
            where we use the triangle inequality and the fact $f_k(y)=v_k$ in the first inequality.
            To estimate the first term on the right-hand side of the inequality \eqref{eq:strainer_map_open_2}, we note from Lemma \ref{lemma:orthogonality_k_strainer} that
            \begin{equation}\label{eq:strainer_map_open_3}
                \left|\cos\angle p_ix\xi_k\right|<2\delta,\quad
                \left|\cos\angle p_ix\eta_k\right|<2\delta,\quad \text{for } i=1,\ldots,k-1.
            \end{equation}
            This, together with Lemma \ref{lemma:angle_well_defined} and the inequality \eqref{eq:strainer_map_open_0}, implies that $|f_i(y)-f_i(x)|\leq 2\delta |xy|$ for $i=1,\ldots,k-1$ whenever $v_k$ is sufficiently close to $f_k(x)$.
            Summing over indices $i$ from 1 to $k - 1$, it follows that
            \begin{equation}
                \left\|f_{[k-1]}(y)-f_{[k-1]}(x)\right\|_1
                =
                \sum_{i=1}^{k-1}\left|f_i(x)-f_i(y)\right|
                \leq
                2\delta(k-1)|xy|.
            \end{equation}
            Plugging it into the inequality \eqref{eq:strainer_map_open_2}, we obtain that
            \begin{equation}
                \|f(y)-v\| - \|f(x)-v\|
                \leq
                -\left(\frac{\varepsilon_{k-1}}{2}(1-2\delta)-2\delta(k-1)\right)|xy|.
            \end{equation}
            Let $\delta_k:=k^{-1}2^{-2k-1}$. 
            One can check that for any $0<\delta<\delta_k$, it holds that
            \begin{equation}
                \|f(y)-v\|-\|f(x)-v\| \leq -\frac{5}{16 }\varepsilon_{k-1}|xy|<-\frac{1}{4}\varepsilon_{k-1}|xy|.
            \end{equation}
        \end{enumerate}

        Combining the discussions above and Lemma \ref{lemma:criterion_open_map}, we obtain that $f$ is an $\varepsilon_{k-1}/4$-open map from $U$ to $\mathbb{R}^k$ equipped with the anisotropic norm $\|\cdot\|$.
        One can readily check that $f$ is also an $\varepsilon_{k-1}/4$-open map from $U$ to $\mathbb{R}^k$ equipped with the $\ell_1$-norm since the $\ell_1$-norm dominates the anisotropic norm $\|\cdot\|$.
        By induction, it follows that $f$ is an $\varepsilon_k(\delta)$-open map for $\delta<\delta_k$, where $\varepsilon_k(\delta):=\frac{(1-2\delta)}{4^{k-1}}$ and $\delta_k=k^{-1}2^{-2k-1}$.
    \end{enumerate}
\end{proof}

The following result follows directly from Lemma \ref{lemma:strained_point_openness} and Proposition \ref{prop:strainer_map_open}.
\begin{corollary}\label{cor:k_strainer_opennes}
    Let $X$ be an $S$-concave, locally semi-convex space.
    Let $(p_1,\ldots,p_k)$ be a $(k,\delta)$-strainer at $x$ for some $\delta<\delta_k$, with $(q_1,\ldots,q_k)$ as an opposite strainer.
    Then there exists a neighborhood $U$ of $x$ such that the associated $(k,\delta)$-strainer map from $U$ to $\mathbb{R}^k$ equipped with the $\ell_2$-norm is $\sqrt{k}$-Lipschitz and $\bar{\varepsilon}_k(\delta)$-open, with $(q_1,\ldots,q_k)$ as an opposite strainer at each point of $U$.
    In particular, the Hausdorff dimension of $U$ is at least $k$.
    Here $\bar{\varepsilon}_k(\delta):=\varepsilon_k(\delta)/\sqrt{k}$, and $\varepsilon_k, \delta_k$ are the constants in Proposition \ref{prop:strainer_map_open}.
\end{corollary}

We now establish the following bi-Lipschitz property of strainer maps.
\begin{proposition}\label{prop:bi_Lip_k_strainer}
    Let $X$ be an $S$-concave, $(C,D)$-semi-convex space.
    Let $(p_1,\ldots,p_k)$ be a $(k,\delta)$-strainer at $x'\in X$ with $\delta<\delta_k$.
    Assume there is a neighborhood $V$ of $x'$ such that no point in $V$ admits a $(k+1,2\delta)$-strainer. 
    Then the $(k,\delta)$-strainer map $f:=(d_{p_1},\ldots,d_{p_k})$ is a bi-Lipschitz homeomorphism from some neighborhood of $x'$ to a domain in $\mathbb{R}^k$.
\end{proposition}
\begin{proof}
    It suffices to show the injectivity of the map $f$ on some neighborhood of $x'$.
    The proof that a Lipschitz map being both injective and $\varepsilon$-open is locally bi-Lipschitz is standard, and we refer to, for example, \cite[Theorem 5.30]{fujioka2025top}.

    Let $r_0>0$ be small enough such that $f$ is a $(k,\delta)$-strainer map on $U:=B(x',r_0)\subset V$, $r_0+|p_1x'|<D$ and $\bar{\delta}_{S,C}(4r_0;|p_kx'|)< 0.01\delta$. 
    By assumption, no point in $U$ admits a $(k+1,2\delta)$-strainer.

    Assume by contraction that $f$ is not injective on $U$.
    Then there exist two different points $x,y\in U$ such that $f(x)=f(y)$. 
    So $|p_ix|=|p_iy|$ for $i=1,\ldots,k$.
    Let $\xi:[0,l]\to X$ be a unit-speed geodesic from $x$ to $y$, where $l=l(\xi)$.

    We will show that, for a point $z=\xi(t)\in U$ sufficiently close to $x$, $(p_1,\ldots,p_k,y)$ is a $(k+1,2\delta)$-strainer at $z$. 

    \begin{enumerate}[label*=\textbf{\textsc{Step \arabic*:}}, fullwidth]
        \item We first check the first two conditions in Definition \ref{def:k_strainer}.
        Firstly, it follows from our choice of $r_0$ that the distance of any point in $U$ from $p_1$ is less than $D$, and that the $k$-tuple $(p_1,\ldots,p_k)$ is a $(k,\delta)$-strainer at each point of $U$ with an opposite strainer $(q_1,\ldots,q_k)$.
        Secondly, choose $t>0$ such that $\bar{\delta}_{S,C}(t;l-t)<\delta/8$ and $z:=\xi(t)\in U$.
        Note that the Euclidean comparison triangle $\tilde{\Delta}xzy$ is degenerated, so $\tilde{\angle}yzx=\pi>\pi-\delta$, and $\bar{\delta}_S(|zx|;|zy|)\leq\bar{\delta}_{S,C}(|zx|;|zy|)<\delta/8$.
        This shows that $y$ is a $(1,\delta)$-strainer at $z$ with $x$ as an opposite strainer.
        Moreover, from the choice of $r_0$, it follows that 
        \begin{equation}\label{eq:bi_Lip_k_strainer_1}
            \bar{\delta}_{S,C}(|yz|;|p_iz|)
            \leq
            \bar{\delta}_{S,C}(2r_0;|p_ix'|-r_0)
            \leq
            \bar{\delta}_{S,C}(4r_0;|p_ix'|)
            \leq
            \bar{\delta}_{S,C}(4r_0;|p_kx'|)<0.01\delta,
        \end{equation}
        for all $i=1,\ldots,k$.
        Therefore, the conditions \eqref{cond:k_strainer_1}-\eqref{cond:k_strainer_2} in Definition \ref{def:k_strainer} are satisfied for the $(k+1)$-tuple $(p_1,\ldots,p_k,y)$ at $z$.

        \item We are left to check the almost orthogonality condition \eqref{cond:k_strainer_3} in Definition \ref{def:k_strainer}.
        
        We first show that $|\tilde{\angle} p_izy-\pi/2|<\delta$.
        Indeed, the conditions $\bar{\delta}_{S,C}(|p_jz|;|p_iz|)<\delta$ for $1\leq i<j\leq k$ in the definition of $(k,\delta)$-strainer maps imply that $|p_jz|\leq |p_iz|$ for $1\leq i<j\leq k$.
        Combining with $\bar{\delta}_{S,C}(4r_0;|p_kx'|)<0.01\delta$, the inequality \eqref{eq:bi_Lip_k_strainer_1} and $r\leq \arccos(1-r)$, we obtain that $r_0\leq 0.01 \delta\min_{i=1\ldots,k}|p_ix|$.
        Applying the Euclidean law of cosines to the comparison triangle $\tilde{\Delta}p_izy$, we obtain that
        \begin{multline}\label{eq:bi_Lip_k_strainer}
            \left|\cos\tilde{\angle}p_izy\right|
            =
            \left|\frac{|p_iz|^2+|zy|^2-|p_iy|^2}{2|p_iz||zy|}\right|
            \leq
            \frac{|zy|}{2|p_iz|} + \frac{(|p_iz|+|p_ix|)\left||p_iz|-|p_ix|\right|}{2|p_iz||zy|}\\
            \leq
            \frac{r_0}{2(1-0.01\delta)|p_ix|} + \frac{(2+0.01\delta)|p_ix||xz|}{2(1-0.01\delta)|p_ix||zy|}
            \leq
            0.1\delta + 1.1\,\frac{t}{l-t},
        \end{multline}
        where we use the assumption that $|p_ix|=|p_iy|$ in the first inequality.
        By our choice of $t$ that $\bar{\delta}_{S,C}(t;l-t)<\delta/8$, we can derive that $|\cos\tilde{\angle}p_izy|<\delta/2$.
        This implies that $|\tilde{\angle}p_izy-\pi/2|<\delta$.

        Finally, we show that $|\tilde{\angle} p_izx-\pi/2|<2\delta$.
        Let $\gamma,\eta$ be the reparametrizations by arc-length of geodesic segments of the geodesic $\xi$ from $z$ to $y,x$ respectively.
        Combining the inequality $|\tilde{\angle}p_izy-\pi/2|<\delta$ with the almost comparison inequalities (Lemma \ref{lemma:angle_well_defined} and Lemma \ref{lemma:angle_well_defined_2}) and the fact $\bar{\delta}_{S,C}(|yz|;|p_iz|)<0.01\delta$, it follows that $|\angle p_iz\gamma-\pi/2|<1.5\delta$.
        Note that Lemma \ref{lemma:sum_of_angle_1} and Lemma \ref{lemma:sum_of_angle_2} imply $\angle p_iz\gamma + \angle p_iz\eta=\pi$.
        This together with the fact $|\angle p_iz\gamma -\pi/2|<1.5\delta$ implies that $|\angle p_iz\eta -\pi/2|<1.5\delta$.
        By again applying the almost comparison inequality \eqref{eq:angle_almost_comparison} and \eqref{eq:angle_almost_comparison_lower} and noting that $\bar{\delta}_{S,C}(|xz|;|p_iz|)\leq \bar{\delta}_{S,C}(|yz|;|p_iz|)\leq \bar{\delta}_{S,C}(4r_0;|p_ix'|)<0.01\delta$, we obtain that $|\tilde{\angle}p_izx -\pi/2|<2\delta$.
    \end{enumerate}

    In conclusion, the $(k+1)$-tuple $(p_1,\cdots,p_k,y)$ is a $(k+1,2\delta)$-strainer at $z\in U$.
    This contradicts our assumption that no point in $U$ admits a $(k+1,2\delta)$-strainer.
    Hence, $f$ is injective on $U$, which ends the proof.  
\end{proof}

\subsection{Self-improvement of strainers}\label{sect:self_improvement_strainer}

In this subsection, we establish a self-improvement property for $(k,\delta)$-strainers: any $(k,\delta)$-strainer can be improved to a $(k,\delta')$-strainer for any $0<\delta'<\delta$. 
This property is analogous to the one \citep[Lemma 5.9]{burago1992ad} for Alexandrov spaces (see also \citep[Lemma 10.8.17]{burago2001course}), whose proof is based on a straightening strategy.
However, we cannot apply this strategy directly in our setting due to the asymmetry of strainer maps: the new $k$-tuple $(p_1,\ldots, p'_i,\ldots, p_k)$ obtained by straightening the strainer pair $(p_i, q_i)$ from a $(k,\delta)$-strainer $(p_1,\ldots,p_k)$ is no longer a strainer.

Here, we adopt a new `dequeuing and enqueuing' strategy: we first straighten the first strainer pair $(p_1, q_1)$ to obtain a new strainer pair $(p'_1,q'_1)$.
Then we remove $p_1$ and add $p'_1$ to the end of the strainer list, forming a new $k$-tuple $(p_2,\ldots,p_k,p'_1)$. 
This new $k$-tuple turns out to be a `mixed' strainer in the following sense: while it is still a $(k,\delta)$-strainer, the point $p'_1$ is a $(1,\delta')$-strainer with a better parameter $\delta' < \delta$.

We continue this dequeuing and enqueuing procedure for $k$ steps. 
At the end of each step, we obtain a new $(k, \delta)$-strainer $(p_j, \ldots, p_k, p'_1, \ldots, p'_{j-1})$ at some point $z$ near $x$, whose last $j-1$ elements form a $(j-1, \delta')$-strainer at $z$. 
Finally, we obtain the desired $(k, \delta')$-strainer.

\begin{lemma}[Self-improvement of a strainer]\label{lemma:self_improvement_strainer}
    Let $X$ be an $S$-concave, $(C,D)$-semi-convex space.
    Let $x$ be a $(k,\delta)$-strained point for $\delta<\delta_k$.
    Then for any $\delta'<\delta$, any open neighborhood of $x$ contains a $(k,\delta')$-strained point.
\end{lemma}
\begin{proof}
    Let $(p_1,\ldots,p_k)$ be a $(k,\delta)$-strainer at $x$ with $(q_1,\ldots,q_k)$ as the opposite strainer.
    Let $f:=(d_{p_1},\cdots,d_{p_k})$ be the associated strainer map, and $0<\delta'<\delta$ be arbitrary.
    The proof for the case $k=1$ is trivial, since we can pick a point $y$ in some geodesic joining $x$ and $p_1$ sufficiently close to $x$ such that $y$ is a $(1,\delta')$-strained point with $(p_1,x)$ as its $(1,\delta')$-strainer pair.
    For $k\geq 2$, we will dequeue and enqueue the strainer points $p_i$ from $i=1$ to $k$ inductively to construct a $(k,\delta')$-strainer.

    \begin{enumerate}[label*=\textbf{\textsc{Step \arabic*:}}, fullwidth]
        \item We first deal with the point $p_1$.
        By Corollary \ref{cor:k_strainer_opennes}, we can take $r_1>0$ small enough such that $f$ is an $\bar{\varepsilon}_k$-open $(k,\delta)$-strainer map on $U_1:=B(x,r_1)$ with a common opposite strainer $(q_1,\ldots,q_k)$ on $U_1$, and 
        \begin{equation}\label{eq:self_improvement_strainer_0}
            \bar{\delta}_{S,C}(4r_1;|p_ix|)<0.01\delta'/3, \quad \text{for } i=2,\ldots,k.
        \end{equation}
        Since $f$ is open on $U_1$, we can find $v\in f(U_1)$ such that $v_i=|p_ix|$ for $i=2,\cdots,k$, and $y\in U_1\setminus \{x\}$ such that $f(y)=v$.
        Let $\xi$ be a geodesic from $x$ to $y$.
        Take $z_1:=\xi(t)$, where $t\in (0,l(\xi))$ is to be determined later.
        Applying the Euclidean law of cosines to $\tilde{\Delta}p_iz_1y$ and following the same argument as the inequality \eqref{eq:bi_Lip_k_strainer}, we obtain that
        \begin{equation}
            |\cos\tilde{\angle}p_iz_1y|\leq 0.1\frac{\delta'}{3} + 1.1\,\frac{t}{l-t},
        \end{equation}
        for $i=2,\cdots,k$.
        Using the same argument as in Step 1 and 2 of Proposition \ref{prop:bi_Lip_k_strainer}, we can find sufficiently small $t>0$, so that $y$ is a $(1,\delta')$-strainer at $z_1\in U_1$ with $x$ as an opposite strainer point, and
        \begin{equation}\label{eq:self_improvement_strainer_1}
            \left|\tilde{\angle}p_iz_1y - \pi/2 \right|<\frac{2}{3}\delta',\quad
            \left|\tilde{\angle}p_iz_1x - \pi/2\right|<\delta',\quad i=2,\ldots,k.
        \end{equation}
        Note that $(p_2,\ldots,p_k)$ is a $(k-1,\delta)$-strainer on $U_1$.
        Furthermore, it follows from \eqref{eq:self_improvement_strainer_0} that
        \begin{equation}
            \bar{\delta}_{S,C}(|z_1y|;|p_kz_1|)
            \leq
            \bar{\delta}_{S,C}(2r_1;|p_kx|-r_1)
            \leq
            \bar{\delta}_{S,C}(4r_1;|p_kx|)
            < 0.01\delta'/3<\delta'.
        \end{equation}
        Denote $p'_1:=y$ and $q'_1:=x$.
        Together with \eqref{eq:self_improvement_strainer_1}, it follows from Definition \ref{def:k_strainer} that the $k$-tuple $(p_2,\ldots,p_k, p'_1)$ is a $(k,\delta)$-strainer at $z_1$ with an opposite strainer $(q_2,\ldots,q_k,q'_1)$.
        It can also be seen that $|z_1x|<r_1$.

        \item Assume that for $2\leq j<k$, we have a $(k,\delta)$-strainer $(p_j,\ldots,p_k,p'_1,\ldots,p'_{j-1})$ at $z_{j-1}$ with $(q_j,\ldots,q_k,q'_1,\ldots,q'_{j-1})$ as an opposite strainer, such that $(p'_1,\ldots, p'_{j-1})$ is a $(j-1,\delta')$-strainer at $z_{j-1}$ with $(q'_1,\ldots,q'_{j-1})$ as an opposite $(j-1,\delta')$-strainer and that $|xz_{j-1}|<\sum_{i=1}^{j-1}r_1/2^{i-1}$.
        Let $f$ be the associated $(k,\delta)$-strainer map.
        Let $0<r_j< r_1/2^{j-1}$ be a small constant such that $f$ is an $\bar{\varepsilon}_k$-open $(k,\delta)$-strainer map on $U_j:=B(z_{j-1},r_j)$ with a common opposite strainer $(q_j,\ldots,q_k,q'_1,\ldots,q'_{j-1})$ on $U_j$, and
        \begin{equation}
            \bar{\delta}_{S,C}(4r_j;|p'_{j-1}z_{j-1}|)<\frac{0.01}{3}\delta'.
        \end{equation}
        By the openness of $f$, we can find $p'_{j}\in U_j$ such that $|p_iz_{j-1}|=|p_ip'_{j}|$ and $|p'_lz_{j-1}|=|p'_lp'_j|$ for all $i=j+1,\ldots,k$ and $l=1,\ldots,j-1$.
        Let $\xi$ be a geodesic from $z_{j-1}$ to $p'_j$, and denote $z_j:=\xi(t)$ for $t\in (0,l(\xi))$ to be determined below.
        Now by taking $t>0$ sufficiently small and following the same argument as in Step 1, one can readily check that the $k$-tuple $(p_{j+1},\ldots,p_k,p'_1,\cdots,p'_{j-1},p'_j)$ is a $(k,\delta)$-strainer at $z_j\in U_j$ such that $(p'_1,\ldots,p'_{j-1},p'_j)$ is a $(j,\delta')$-strainer at $z_j$.
        Moreover, by our choice of $r_j$, it follows that $|z_{j-1}z_j|<r_j\leq r_1/2^{j-1}$ and therefore $|xz_j|< \sum_{i=1}^{j}r_1/2^{i-1}$.
        Denote $q'_j:=z_{j-1}$.
        Then we obtain a new $(k,\delta)$-strainer $(p_{j+1},\ldots,p_k,p'_1,\ldots,p'_j)$ at $z_j$ with $(q_{j+1},\ldots,q_k,q'_1,\ldots,q'_j)$ as an opposite strainer such that $(p'_1,\ldots,p'_j)$ is a $(j,\delta')$-strainer at $z_j$.

        \item By induction, we obtain a new $(k,\delta')$-strainer $(p'_1,\ldots,p'_k)$ at some point $z$ with $|zx|<2r_1$.
        Since $r_1>0$ can be arbitrarily small, the statement follows.
    \end{enumerate}
\end{proof}

\section{Dimension}\label{sect:dimension}
\noindent
In this section, we investigate the dimension of $S$-concave, locally semi-convex Busemann concave spaces.
We prove that the Hausdorff dimension of such a space $X$ coincides with a constant, called the \emph{strainer number} of $X$. 
Furthermore, if $X$ has either finite Hausdorff dimension or finite strainer number, then $X$ is proper, its topological dimension agrees with its Hausdorff dimension, and $X$ carries a non-trivial $n$-dimensional Hausdorff measure for some $n \in \mathbb{N}$. 
In particular, the metric measure space $(X, d, \mathcal{H}^n)$ satisfies a synthetic curvature-dimension condition, called the measure contraction property $\mathrm{MCP}(0, n)$.

We first recall the following definition of strainer number (see \citep[Section 6]{burago1992ad} and \cite[Definition 10.8.11]{burago2001course}).
\begin{definition}[Strainer number]\label{def:strainer_number}
    Let $(X,d)$ be an $S$-concave, locally semi-convex Busemann concave space.
    A natural number $m \in \mathbb{N}$ is called the \emph{strainer number} at a point $x \in X$ if, for any $\delta >0$ and any neighborhood $U_x$ of $x$, there exists a $(m,\delta)$-strained point $y \in U_x$, but the analogous property fails for $m+1$.
    If no such number exists, the strainer number at $x$ is defined to be $\infty$.
    The strainer number of $X$ is the supremum of all $m$ such that there exists a point $x \in X$ with strainer number $m$.
\end{definition}

It can be seen that the strainer number at a point is well-defined, i.e., it is either a unique natural number or infinity.
Moreover, the strainer number at any point is at least 1.
\begin{lemma}\label{lemma:existence_1_strainer}
    Let $X$ be an $S$-concave, $(C,D)$-semi-convex Busemann concave space.
    Then for any $\delta>0$, the set of $(1,\delta)$-strained points is dense in $X$.
    In particular, the strainer number at any point is at least 1.
\end{lemma}
\begin{proof}
    Let $x\in X$ be any point, and $p\in X$ be such that $|px|<D$.
    Let $\xi$ be a unit-speed geodesic from $x$ to $p$.
    Take $t>0$ small enough such that $\bar{\delta}_S(t;|px|-t)<\delta$.
    Let $z:=\xi(t)$.
    Then we have $\tilde{\angle}pzx=\pi>\pi-\delta$ and $\bar{\delta}_S(|zx|;|pz|)<\delta$.
    This implies that $p$ is a $(1,\delta)$-strainer at $z=\xi(t)$ with $x$ as an opposite strainer, and therefore $z$ is a $(1,\delta)$-strained point.
    Since $t$ can be taken arbitrarily small, the assertion follows.
\end{proof}

We now show that the strainer number of an $S$-concave, locally semi-convex Busemann concave space is exactly the Hausdorff dimension of the space. 
Our proof is based on the constancy of the Hausdorff dimensions of Busemann concave spaces established by Kell \citep[Lemma 2.22]{kell2019sectional}, which states that all bounded open subsets of a Busemann concave space have the same Hausdorff dimension.
\begin{lemma}\label{lemma:strainer_num_equal_H_dim}
    Let $X$ be an $S$-concave, locally semi-convex Busemann concave space.
    Then the strainer number at any point of $X$ is equal to the Hausdorff dimension of the space.
    If this strainer number is finite, then it coincides with the topological dimension of the space.
\end{lemma}
\begin{proof}
    Let $U$ be any bounded open neighborhood of a point $x\in X$.
    We show that the Hausdorff dimension of $U$ is equal to the strainer number $m$ at $x$.

    In the case that $m=\infty$: for any $k \in \mathbb{N}$ and any $0<\delta < \delta_k$, by Definition \ref{def:strainer_number} and Corollary \ref{cor:k_strainer_opennes}, there exists a $(k, \delta)$-strained point $y\in U$, such that the associated $(k,\delta)$-strainer map $f:V\to \mathbb{R}^k$ defined on a neighborhood $V\subset U$ of $y$, is $\sqrt{k}$-Lipschitz and $\bar{\varepsilon}_k(\delta)$-open.
    Thus,
    \begin{equation}
        k=\mathrm{dim}_H(f(V))\leq \mathrm{dim}_H(V)\leq \mathrm{dim}_H(U).
    \end{equation}
    Since $k\in \mathbb{N}$ is arbitrary, it follows that $\mathrm{dim}_H(X)=\mathrm{dim}_H(U)=\infty$.

    In the case that the strainer number $m$ at $x$ is finite: by definition, we can find a small neighborhood $U_x\subset U$ of $x$ which does not contain any $(m+1,\delta)$-strained point for some $\delta<\delta_{m+1}/2$.
    By the self-improvement property of strainers (Lemma \ref{lemma:self_improvement_strainer}), it follows that $U_x$ does not contain any $(m+1,2\delta)$-strained point either.
    Now let $y\in U_x$ be an $(m,\delta)$-strained point and $f$ be the associated strainer map.
    Since there is no $(m+1,2\delta)$-strained point near $y$, by Proposition \ref{prop:bi_Lip_k_strainer} we can find an open neighborhood $V\subset U_x$ of $y$ such that $f: V\to f(V)\subset \mathbb{R}^m$ is a bi-Lipschitz homeomorphism.
    Since all bounded open subsets have the same Hausdorff dimension in Busemann concave spaces (see \citep[Lemma 2.22]{kell2019sectional}), and $X$ can be covered by countably many bounded open neighborhoods of $x$, it follows that
    \begin{equation}
        m=\mathrm{dim}_H(f(V))=\mathrm{dim}_H(V)=\mathrm{dim}_H(U_x)=\mathrm{dim}_H(U)=\mathrm{dim}_H(X),
    \end{equation}
    which is the thesis.

    For the second assertion, assume the strainer number $m$ at $x$ is finite.
    On the one hand, the topological dimension of $U$ is not greater than the Hausdorff dimension of $U$.
    On the other hand, from the proof of the first assertion, we can find a neighborhood $V\subset U$ which is bi-Lipschitz homeomorphic to some open subset of $\mathbb{R}^m$.
    These imply that
    \begin{equation}
        m=\mathrm{dim}_T(V)\leq \mathrm{dim}_T(U)\leq \mathrm{dim}_H(U)=m.
    \end{equation}
    This shows that the topological dimension of $U$ is equal to $m$.
    By the constancy of the Hausdorff dimension again, our assertion follows.
\end{proof}

We now prove our main result of this section.

\begin{proposition}[Measure contraction property]\label{prop:Buseman_concave_MCP}
    Let $X$ be an $S$-concave, locally semi-convex, Busemann concave space.
    Then $X$ has finite Hausdorff dimension if and only if it has finite strainer number.
    In either case, both values coincide with the topological dimension of $X$, and $X$ admits a non-trivial $n$-Hausdorff measure.
    Moreover, the metric measure space $(X, d, \mathcal{H}^n)$ satisfies the measure contraction property $\mathrm{MCP}(0, n)$ and the Bishop--Gromov volume comparison inequality $\mathrm{BG}(0, n)$.
    In particular, $(X,d)$ is doubling and proper\footnote{A metric space is called \emph{proper} if all closed and bounded subsets are compact.}.
\end{proposition}
\begin{proof}
    The first assertion follows from Lemma \ref{lemma:strainer_num_equal_H_dim}.
    From the proof of Lemma \ref{lemma:strainer_num_equal_H_dim}, we can find a bounded open subset $U\subset X$ that is bi-Lipschitz homeomorphic to a bounded open subset of $\mathbb{R}^n$.
    This implies that $0<\mathcal{H}^n(U)<\infty$, and therefore $X$ admits a non-trivial $n$-Hausdorff measure $\mathcal{H}^n$.

    For the second assertion, both the measure contraction property and the Bishop--Gromov volume comparison inequality follow directly from the first assertion and \cite[Proposition 2.23]{kell2019sectional}.
    In particular, the Bishop--Gromov inequality implies that $\mathcal{H}^n$ is a doubling measure on $X$, which in turn implies that $(X,d)$ is proper (see \cite[Proposition 3.1]{bjorn2011nonlinear}) and doubling (see for example \cite[page 102]{shanmugalingam2015sobolev}).
\end{proof}

In light of the equality between the strainer number and the Hausdorff dimension, we say that an $S$-concave, locally semi-convex Busemann concave space is \emph{finite dimensional} if it has either finite Hausdorff dimension or finite strainer number.

We conclude this section with a topological result.
\begin{corollary}[Topological manifold part]\label{cor:strained_points_open_dense_top_manifold}
    Let $X$ be an $n$-dimensional $S$-concave, locally semi-convex Busemann concave space.
    Then for any $\delta>0$, the set $\mathcal{A}(n,\delta)$ of $(n,\delta)$-strained points in $X$ is open and dense in $X$.
    Moreover, $\mathcal{A}(n,\delta)$ is a topological $n$-manifold if $0<\delta<\delta_{n+1}/2$.
\end{corollary}
\begin{proof}
    By Lemma \ref{lemma:strained_point_openness}, $\mathcal{A}(n,\delta)$ is open. 
    By Lemma \ref{lemma:strainer_num_equal_H_dim}, $\mathcal{A}(n,\delta)$ is dense in $X$.
    
    For any $0<\delta<\delta_{n+1}/2$ and $x\in \mathcal{A}(n,\delta)$, it follows from the proof of Lemma \ref{lemma:strainer_num_equal_H_dim} that there exists a neighborhood $U$ of $x$ which does not contain any $(n+1,2\delta)$-strained point.
    By applying Proposition \ref{prop:bi_Lip_k_strainer}, we can find a neighborhood $V \subset \mathcal{A}(n,\delta)$ of $x$ which is bi-Lipschitz homeomorphic to some open subset of $\mathbb{R}^n$.
    This shows that $\mathcal{A}(n,\delta)$ is a topological $n$-manifold.
\end{proof}

\section{Hausdorff measure and Hausdorff dimension of singular sets}\label{sect:Hausdorff_measure_strata_main}
\noindent
In this section, we further investigate the structure of $n$-dimensional $S$-concave, locally semi-convex Busemann concave spaces.

\subsection{Hausdorff measure of singular sets}\label{sect:Hausdorff_measure_strata}

We recall the following definition of cut points in geodesic spaces, following \cite{von2008local}.
See also \cite{sormani2009conjugate} for alternative notions of cut loci.
\begin{definition}
    Let $X$ be a geodesic space, and $x\in X$.
    Let $I_x$ be the set of all points $y$ that are connected to $x$ by at least one extendable geodesic, i.e.,
    \begin{equation}
        I_x:=\left\{\gamma(t): t\in [0,1),\quad \text{$\gamma$ any constant-speed geodesic starting at $x$}\right\}.
    \end{equation} 
    Let $T_x:=X\setminus I_x$ denote the set of cut points of $x$, i.e., those points that do not lie in the interior of any geodesic starting at $x$.
\end{definition}

The next lemma shows that the cut locus has null top-dimensional Hausdorff measure.
A similar result for Alexandrov spaces with curvature bounded below was proved by Otsu and Shioya \cite[Proposition 3.1]{otsu1994riemannian}.
\begin{lemma}\label{lemma:null_measure_cut_points}
    Let $X$ be an $n$-dimensional $S$-concave, locally semi-convex Busemann concave space.
    Then $\mathcal{H}^n(T_x)=0$ for all $x\in X$.
    In particular, $\mathcal{H}^n(X\setminus \mathcal{A}(1,\delta))=0$ for any $\delta>0$.
\end{lemma}
\begin{proof}
    The first assertion follows from Proposition \ref{prop:Buseman_concave_MCP} and a well-known fact (see \cite[Section 3.3]{von2008local}) that in a metric measure space satisfying the measure contraction property, the set of cut points is negligible.
    Using an argument similar to that in Lemma \ref{lemma:existence_1_strainer} and the almost extendable property of geodesics, we can prove the second assertion.
\end{proof}

We now present a key technical lemma, which states that for any $1 \leq k \leq n-1$, the set of points admitting a $(k,\delta)$-strainer but not a $(k+1,4\delta)$-strainer has zero $n$-Hausdorff measure.
A similar result for Alexandrov spaces was proved in \cite[Lemma 10.5 and Theorem 10.6]{burago1992ad}, based on the metric cone structure of tangent cones. 
Here we adopt a more straightforward approach, where we utilize the `almost extendable' property of geodesics, established in Lemma \ref{lemma:null_measure_cut_points}, to analyze the infinitesimal behavior $\liminf_{y \to x} |f(y)f(x)| / |xy|$ of strainer maps restricted to singular sets. 
Then a lemma of Lytchak \cite[Lemma 3.1]{lytchak2006open} enables us to determine the top-dimensional Hausdorff measure of singular sets.

\begin{lemma}\label{lemma:hausdorff_measure_strained_set}
    Let $X$ be an $n$-dimensional $S$-concave, $(C,D)$-semi-convex Busemann concave space.
    Given $1\leq k \leq n-1$, let $(p_1, \ldots, p_k)$ be a $(k, \delta)$-strainer on an open set $U$ for some $\delta < \delta_k$.
    Let $E \subset U$ be the set of points that do not admit any $(k+1, 4\delta)$-strainer.
    Then $\mathcal{H}^n(E) = 0$.
\end{lemma}
\begin{proof}
    Let $f:=(d_{p_1},\ldots,d_{p_k}):U\to \mathbb{R}^k$ be the associated $(k,\delta)$-strainer map.
    Assume by contradiction that $\mathcal{H}^n(E)>0$.

    \begin{enumerate}[label*=\textbf{\textsc{Step \arabic*:}}, fullwidth]
        \item Let $F\subset E$ be the set of density points\footnote{Given a metric space $(X,d)$, we say a point $x\in X$ is a density point of $E\subset X$ with respect to $\mathcal{H}^n$ if $\lim_{r\to 0} \frac{\mathcal{H}^n(B(x,r)\cap E)}{\mathcal{H}^n(B(x,r))}=1$.} of $E$ with respect to $\mathcal{H}^n$.
        From Proposition \ref{prop:Buseman_concave_MCP}, we know that $\mathcal{H}^n$ is doubling.
        Then by the Lebesgue differentiation theorem (see for example \cite[p. 77]{shanmugalingam2015sobolev}), it follows that $\mathcal{H}^n(E\setminus F)=0$. 
        By the definition of density points, it follows that $\mathcal{H}^n(B(x,r)\cap E )>0$ for any $x\in F$ and any $r>0$.
        By the inner regularity of the $n$-Hausdorff measure\footnote{Recall that Hausdorff measures on complete metric spaces are Radon measures.}, we can find a sequence of increasing compact subsets $K_n\subset F$ such that $\mathcal{H}^n(F\setminus \cup_n K_n)=0$.
        Let $K:=\cup_{n}K_n$.
        It is clear that $\mathcal{H}^n(E\setminus K)=0$. 
        Furthermore, from $\mathcal{H}^n(B(x,r)\cap E )>0$ it follows that $\mathcal{H}^n(B(x,r)\cap K)>0$ for any $x\in K$ and $r>0$.
        
        By Lemma \ref{lemma:null_measure_cut_points}, it holds that $I_x$ has full $\mathcal{H}^n$-measure of $X$ for any $x\in K$. 
        Then $I_x\cap B(y,r)\cap K\neq \emptyset$ for any $x, y\in K$ and $r>0$.
        In particular, $I_x\cap K$ is dense in $K$ for any $x\in K$.

        \item In this step, we show that $\liminf_{K \ni y\to x}\frac{|f(y)f(x)|}{|yx|}\geq \delta$ for every $x\in K$.
        Since $I_x\cap K$ is dense in $K$, it suffices to show that
        \begin{equation}\label{eq:hausdorff_measure_strained_set_2}
        \lim_{r\to 0}\inf_{\substack{y\in B(x,r)\cap K,\\ y\in I_x}}\frac{|f(y)f(x)|}{|xy|}\geq \delta,\quad \text{for any $x\in K$}.
        \end{equation}
        Note that the infimum on the left-hand side of \eqref{eq:hausdorff_measure_strained_set_2} is not infinity\footnote{We take the convention that the infimum of the empty set is $\infty$.}, as $B(x,r)\cap K \cap I_x$ is not empty for all $x\in K$ and $r>0$.

        Suppose by contradiction that the claim fails.
        Then there exists a point $x\in K$ and a sequence $\{x_j\}_{j}\subset K\cap I_x$ converging to $x$ such that
        \begin{equation}\label{eq:lemma_rect_2}
            \lim_{j\to \infty}\frac{\big||p_ix_j|-|p_ix|\big|}{|xx_j|}
            \leq
            \lim_{j\to \infty}\frac{|f(x)f(x_j)|}{|xx_j|}<\delta,
        \end{equation}
        for all $i=1,\ldots,k$.
        Denote $t_j:=|xx_j|$. 
        Choose $j$ sufficiently large such that $t_j+|p_1x_j|<D$, $\bar{\delta}_{S,C}(2t_j;|p_ix|)<0.01\delta$, and $\big||p_ix_j|-|p_ix|\big|/t_j<\delta$ for all $i=1,\ldots,k$.
        Since $x_j\in I_x$, there exists $\varepsilon_j>0$ and a unit-speed geodesic $\eta_j:[0,t_j + \varepsilon_j]\to X$ such that $\eta_j(0)=x$ and $\eta_j(t_j)=x_j$.
        Choose $s_j\in (0,\varepsilon_j)$ sufficiently small such that $\bar{\delta}_{S,C}(s_j;t_j)<\delta$ and set $y:=\eta_j(t_j+s_j)$.

        We claim that the $(k+1)$-tuple $(p_1,\ldots,p_k, x)$ is a $(k+1,4\delta)$-strainer at $x_j$.
        By assumption, $(p_1,\ldots,p_k)$ is a $(k,\delta)$-strainer at $x_j\in U$.
        Moreover, our choice of $y$ gives that $\tilde{\angle}xx_jy=\pi$ and
        \begin{equation}
            \bar{\delta}_S(|yx_j|;|xx_j|)\leq \bar{\delta}_{S,C}(s_j;t_j)<\delta.
        \end{equation} 
        Therefore, $x$ is a $(1,\delta)$-strainer at $x_j$, with $y$ as an opposite strainer.
        Furthermore, our choice of $t_j$ ensures that, for every $i=1,\ldots,k$,
        \begin{equation}
            \bar{\delta}_{S,C}(|xx_j|;|p_ix_j|)
            \leq
            \bar{\delta}_{S,C}(t_j;|p_ix|-t_j)
            \leq
            \bar{\delta}_{S,C}(2t_j;|p_ix|)
            <\delta.
        \end{equation}
        Consequently, conditions \eqref{cond:k_strainer_1}--\eqref{cond:k_strainer_2} in Definition \ref{def:k_strainer} of $(k+1,\delta)$-strainers are satisfied for $(p_1,\ldots,p_k,x)$ at $x_j$.

        It remains to verify the almost orthogonality condition in Definition \ref{def:k_strainer} \eqref{cond:k_strainer_3}.
        For each $i=1,\ldots,k$, by \eqref{eq:angle_well_defined_3} in Lemma \ref{lemma:angle_well_defined}, we have
        \begin{equation}
            \cos\tilde{\angle}p_ix_jx
            =
            \frac{|p_ix_j|-|p_ix|}{t_j} + \delta(t_j;\xi_j),
        \end{equation}
        where $\xi_j$ is a unit-speed geodesic from $x_j$ to $x$, and $\delta$ is the error term in \eqref{eq:angle_well_defined_3} satisfying $0\leq \delta(t_j;\xi_j)\leq t_j/(2|p_ix_j|)$.
        Therefore, our choice of $t_j$ yields
        \begin{equation}
            \left|\cos\tilde{\angle}p_ix_jx\right|
            \leq
            \left|\frac{|p_ix_j|-|p_ix|}{t_j}\right|
            +
            \frac{t_j}{2|p_ix_j|}
            <
            \delta + \frac{\delta}{2}
            =
            \frac{3}{2}\delta.
        \end{equation}
        This implies that $|\tilde{\angle}p_ix_jx - \pi/2|<3\delta$.
        Finally, since $\eta_j$ is a unit-speed geodesic passing through $x_j$, the argument used in Step 2 of the proof of Proposition \ref{prop:bi_Lip_k_strainer}, together with our choices of $t_j$ and $y$, also gives
        \begin{equation}\label{eq:hausdorff_measure_strained_set_1}
            \left|\tilde{\angle}p_ix_j y - \frac{\pi}{2} \right|<4\delta.
        \end{equation}
        Thus, the almost orthogonality condition in Definition \ref{def:k_strainer} is satisfied.
        Combining these estimates with the preceding verification of conditions \eqref{cond:k_strainer_1}--\eqref{cond:k_strainer_2}, we conclude that the tuple $(p_1,\ldots,p_k,x)$ is a $(k+1,4\delta)$-strainer at $x_j$.
        This contradicts $x_j\in K\subset E$, since $E$ contains no $(k+1,4\delta)$-strained point.
        Therefore, the claim follows.
        

        \item We are ready to derive a contradiction to $\mathcal{H}^n(E)>0$ with an idea from \cite[Lemma 3.1]{lytchak2006open}.
        Let $Z_m\subset K$ be the set of points $x\in K$ satisfying that if $z\in K$ with $|zx|<1/m$, then $|f(x)f(z)|\geq \delta|xz|/2$.
        One can check that $Z_m$ is closed in $K$. 
        Since $\lim\inf_{K\ni y\to x}|f(y)f(x)|/|xy|\geq \delta$ when $f$ is restricted to $K$ as shown in Step 2, it follows that $K=\cup_{m}Z_m$.
        Note that $f$ is locally bi-Lipschitz when restricted to each $Z_m$.
        This implies that the Hausdorff dimension of $Z_m$ is at most $k$ and so is the Hausdorff dimension of $K$.
        Therefore, $\mathcal{H}^n(K)=0$ since $k\leq n-1$.
        This contradicts the assumption that $\mathcal{H}^n(K)=\mathcal{H}^n(E)>0$.
    \end{enumerate}
\end{proof}

The next theorem concerns the $n$-Hausdorff measure of the set of $(n,\delta)$-strained points.
\begin{theorem}\label{thm:full_measure_n_strained_points}
    Let $X$ be an $n$-dimensional $S$-concave, locally semi-convex Busemann concave space.
    Then for any $\delta>0$, $\mathcal{H}^n(X \setminus \mathcal{A}(n,\delta))=0$.
    In particular, $X$ contains an open dense topological $n$-manifold part which has full $n$-Hausdorff measure.
\end{theorem}
\begin{proof}
    It suffices to show the first claim for $0<\delta<\delta_n$.
    For $k=1,\ldots,n$, denote $\delta'_k:=\delta/4^{n-k}<\delta_k$.
    We claim that $\mathcal{H}^n(\mathcal{A}(k-1,\delta'_{k-1})\setminus \mathcal{A}(k,\delta'_k))=0$ for $k=2,\ldots,n$.
    Indeed, for each $x\in \mathcal{A}(k-1,\delta'_{k-1})$, by Lemma \ref{lemma:strained_point_openness}, we can find $r_x>0$ and a $(k-1)$-tuple $(p_1,\ldots,p_{k-1})$, such that $B(x,r_x)\subset \mathcal{A}(k-1,\delta'_{k-1})$ and $(p_1,\ldots,p_{k-1})$ is a $(k-1,\delta'_{k-1})$-strainer on $B(x,r_x)$.
    Now since $X$ is proper, by approximating $\mathcal{A}(k-1,\delta'_{k-1})$ by a sequence of increasing compact subsets $K_j$ and then covering $K_j$ by finitely many balls $B(x,r_x)$ with $x\in K_j$, we can cover $\mathcal{A}(k-1,\delta'_{k-1})$ by countably many balls $\{B(x_l, r_l)\}_{l\in \mathbb{N}}$.
    By Lemma \ref{lemma:hausdorff_measure_strained_set}, it follows that $\mathcal{H}^n(B(x_l,r_l)\setminus \mathcal{A}(k,\delta'_k))=0$ for all $l\in \mathbb{N}$.
    Thus, $\mathcal{H}^n(\mathcal{A}(k-1,\delta'_{k-1})\setminus \mathcal{A}(k, \delta'_k))=0$ for $k=2,\ldots,n$. 
    On the other hand, by Lemma \ref{lemma:null_measure_cut_points}, we have $\mathcal{H}^n(X \setminus \mathcal{A}(1,\delta'_1))=0$.
    Therefore, it follows that
    \begin{equation}
        \mathcal{H}^n(X \setminus \mathcal{A}(n,\delta))\leq \mathcal{H}^n\left(X\setminus \mathcal{A}(1,\delta'_1)\right)+\sum_{k=2}^{n}\mathcal{H}^n\left(\mathcal{A}(k-1,\delta'_{k-1})\setminus \mathcal{A}(k, \delta'_k)\right)=0, 
    \end{equation} 
    which is the thesis.

    For $\delta<\delta_{n+1}/2$, by Corollary \ref{cor:strained_points_open_dense_top_manifold} and the first assertion, the set $\mathcal{A}(n,\delta)$ is a topological $n$-manifold, which is open and dense in $X$ and has full $n$-Hausdorff measure.
\end{proof}

\subsection{Rectifiability and Banach tangent cones}\label{sect:Rectifiability_Banach_tangent_cones}

In this subsection, we derive several structural theorems for finite-dimensional $S$-concave, locally semi-convex Busemann concave spaces.

We first prove $n$-rectifiability.
For more on rectifiability in general metric spaces, we refer to \cite{bate2022characterising,guy_c_david2015GAFA,bate2017characterizations}.

\begin{theorem}[n-rectifiable]\label{thm:n_rect}
    Let $(X,d)$ be an $n$-dimensional $S$-concave, locally semi-convex Busemann concave space.
    Then $X$ is $n$-rectifiable.
\end{theorem}
\begin{proof}
    Let $\delta<\delta_{n+1}/2$.  
    For any $x\in \mathcal{A}(n,\delta)$, let $r_x>0$ be such that $B(x,r_x)\subset \mathcal{A}(n,\delta)$ and that there exists a $(n,\delta)$-strainer map $f_x$ on $B(x,r_x)$.
    By Lemma \ref{lemma:strainer_num_equal_H_dim}, the strainer number of $X$ is $n$.
    Thus, by Proposition \ref{prop:bi_Lip_k_strainer}, we can find a neighborhood $U_x\subset B(x,r_x)$ such that $f_x$ is a bi-Lipschitz homeomorphism from $U_x$ onto $f_x(U_x)$.
    Cover $\mathcal{A}(n,\delta)$ by at most countably many open subsets $U_i:=U_{x_i}$ as in the proof of Theorem \ref{thm:full_measure_n_strained_points}.
    Let $A_i:=f_{x_i}(U_i)$ and $g_i:=f^{-1}_{x_i}:A_i\to U_i$ be the inverse map of $f_{x_i}$.
    Then by Theorem \ref{thm:full_measure_n_strained_points}, we have $\mathcal{H}^n(X \setminus \cup_i g_i(A_i))=0$, which shows that $X$ is $n$-rectifiable.
\end{proof}

The next theorem, concerning Banach tangent cones, is a direct consequence of $n$-rectifiability and Kirchheim's local structure theorem of rectifiable sets in metric spaces \citep[Theorem 9]{kirchheim1994rectifiable}, which is well-known to experts in the field.
For the sake of completeness, we provide a detailed proof here.
\begin{theorem}[Unique Banach tangent cones]\label{thm:Banach_tangent_cone}
    Let $(X,d)$ be an $n$-dimensional $S$-concave, locally semi-convex Busemann concave space.
    Then $\mathcal{H}^n$-almost every point admits a unique tangent cone $(T_xX, d_x, o)$, which is isometric to a finite-dimensional Banach space.
\end{theorem}
\begin{proof}
    Our proof closely follows \citep[Theorem 6.6]{bate2022characterising}.
    Since $X$ is $n$-rectifiable, it follows from Kirchheim's local structure theorem \cite[Theorem 9]{kirchheim1994rectifiable} that for $\mathcal{H}^n$-a.e. point $x$, we can find a norm $\|\cdot\|_x$ on $\mathbb{R}^n$, a map $f_x:X\to \mathbb{R}^n$ and a closed subset $C_x\subset X$ such that $f_x(x)=0$, $x$ is a spherical density point\footnote{Given a metric space $(X,d)$, we say $x\in X$ is a {spherical} density point of $E\subset X$ with respect to $\mathcal{H}^n$ if $\lim_{r\to 0}\mathcal{H}^n(B(x,r)\cap E)/(\omega_nr^n)=1$, see \cite[Definition 2.4.1]{ambrosio2004topics}.} of $C_x$ with respect to $\mathcal{H}^n$, and
    \begin{equation}
        \lim_{r\searrow 0} \sup\left\{\left|1-\frac{\| f_x(y)-f_x(z)\|_x}{|yz|}\right|: y\neq z, y,z\in B(x,r)\cap C_x\right\}=0.
    \end{equation}
    Fix such a point $x\in X$.
    Define $\varepsilon_r$ as
    \begin{equation}
        \varepsilon_r:=\sup\left\{\left|1-\frac{\| f_x(y)-f_x(z)\|_x}{|yz|}\right|: y\neq z, y,z\in B(x,\sqrt{r})\cap C_x\right\}.
    \end{equation}
    Let $C_{x,r}:=B(x,\sqrt{r})\cap C_x$ and $f_{x,r}:=f_x/r$.
    Then for any $y,z\in C_{x,r}$ with $|yz|/r\leq \min\{r^{-1/2},\varepsilon_{r}^{-1/2}\}$, it follows that
    \begin{equation}\label{eq:Banach_tangent_cone_1}
        \left|\frac{|yz|}{r}- \|f_{x,r}(y)-f_{x,r}(z)\|_x\right|
        \leq
        \varepsilon_r \frac{|yz|}{r}
        \leq
        \varepsilon_r^{1/2}.
    \end{equation}
    This implies that the map $f_{x,r}:(C_{x,r},d/r,x)\to (\mathbb{R}^n,\|\cdot\|_x, 0^n)$ is a bi-Lipschitz map with the bi-Lipschitz constant $L_r$ satisfying
    \begin{equation}
        L_r:=\frac{1+\varepsilon_r}{1-\varepsilon_r^2}\to 1,\quad \text{as $r\to 0$}.
    \end{equation}
    Thus, for any $R>0$, it follows that
    \begin{multline}
        \mathcal{H}^n_{\|\cdot\|_x}\left(B(0^n,R)\setminus f_{x,r}(C_{x,r})\right)
        \leq
        \mathcal{H}^n_{\|\cdot\|_x}\left(B(0^n,R)\setminus f_{x,r}\left(C_{x,r}\cap B_X(x,Rr/L_r)\right)\right)\\
        \leq
        \mathcal{H}^n_{\|\cdot\|_x}(B(0^n,R)) - \mathcal{H}^n_{\|\cdot\|_x}\left(f_{x,r}(C_{x,r}\cap B_X(x,Rr/L_r))\right)\\
        \leq
        \mathcal{H}^n_{\|\cdot\|_x}(B(0^n,R)) - L_r^{-n} \mathcal{H}^n_{d/r}\left(C_{x,r}\cap B_X(x, rR/L_r)\right)\\
        \leq
        \mathcal{H}^n_{\|\cdot\|_x}(B(0^n,R)) - L^{-n}_r r^{-n}\mathcal{H}^n_d\left(C_{x,r}\cap B_X(x,rR/L_r)\right)\\
        =
        \mathcal{H}^n_{\|\cdot\|_x}\left(B(0^n,R)\right) -(\omega_n R^n)\frac{\mathcal{H}^n_d(C_{x,r}\cap B_X(x,rR/L_r))}{\omega_n(rR/L_r)^n}\frac{1}{L_r^{2n}}\\
        \rightarrow
        \omega_n R^n - \omega_n R^n =0,\quad \text{as $r\to 0$},
    \end{multline}
    where we use the identity $\mathcal{H}^n_{\|\cdot\|_x}(B(0^n,R))=\omega_nR^n$ (see \cite[Lemma 6]{kirchheim1994rectifiable}) and the facts that $x$ is a spherical density point of $C_x$ and $L_r\to 1$ in the last step. 
    This implies that for any $R, \delta>0$, the ball $B(0^n,R)$ is contained in the $\delta$-neighborhood of $f_{x,r}(C_{x,r})$ for $r>0$ small enough.
    Together with the inequality \eqref{eq:Banach_tangent_cone_1}, it follows that $f_{x,r}:(C_x,d/r,x)\to (\mathbb{R}^n,\|\cdot\|_x, 0^n)$ is an $\varepsilon'_r$-isometry for some $\varepsilon'_r\to 0$ as $r\to 0$.
    Therefore, $(\mathbb{R}^n,\|\cdot\|_x,0^n)$ is the unique pointed Gromov--Hausdorff limit of the blow-ups $\{C_x, d/r,x\}_r$.

    On the other hand, the upper density estimate \cite[2.10.19 (5)]{federer2014geometric} gives
    \begin{equation}
        \limsup_{r\searrow 0}\frac{\mathcal{H}^n\left(B(x,r)\right)}{\omega_n r^n}
        \leq
        \limsup_{r\searrow 0}\frac{\mathcal{S}^n\left(B(x,r)\right)}{\omega_n r^n}
        \leq
        1,\quad \text{for $\mathcal{H}^n$-a.e. $x\in X$}. 
    \end{equation}
    Here, $\mathcal{S}^n$ denotes the $n$-dimensional spherical measure, defined analogously to the Hausdorff measure $\mathcal{H}^n$ but using coverings by closed balls rather than by arbitrary sets, and satisfies $\mathcal{H}^n\leq \mathcal{S}^n \leq 2^n \mathcal{H}^n$; see \cite[2.10.2 (2)]{federer2014geometric} for details.
    Now let $x\in X$ be a point at which the preceding upper density estimate holds, and assume that $x$ is also a spherical density point of $C_x$. 
    These conditions hold for $\mathcal{H}^n$-a.e. point $x\in X$.
    The Bishop--Gromov volume inequality $\mathrm{BG}(0,n)$ (Proposition \ref{prop:Buseman_concave_MCP}) then yields
    \begin{equation}
        1\geq 
        \lim_{r\searrow 0}\frac{\mathcal{H}^n\left(B(x,r)\cap C_x\right)}{\mathcal{H}^n\left(B(x,r)\right)}
        =
        \lim_{r\searrow 0}\frac{\mathcal{H}^n\left(B(x,r)\cap C_x\right)}{\omega_nr^n}\frac{\omega_nr^n}{\mathcal{H}^n\left(B(x,r)\right)}
        \geq 1,
    \end{equation}
    which implies that $x$ is a density point of $C_x$ with respect to $\mathcal{H}^n$ in the usual sense.
    Together with the fact that the Hausdorff measure $\mathcal{H}^n$ is an outer measure, it follows from \citep[Proposition 3.1]{donne2010unique_tangent_cone} that
    \begin{equation}
        \mathrm{Tan}(C_x,d,x)=\mathrm{Tan}(X,d,x)=\{(T_xX, d_x, o)\},\quad \text{for $\mathcal{H}^n$-a.e. $x\in X$}.
    \end{equation}
    Consequently, for $\mathcal{H}^n$-a.e. $x\in X$, we obtain that $(T_xX,d_x,o)=(\mathbb{R}^n,\|\cdot\|_x, 0^n)$, up to an isometry.
\end{proof}
\begin{remark}
    It remains an open question for us whether, under the assumptions of Theorem \ref{thm:Banach_tangent_cone}, all Banach tangent cones are isometric to the same finite-dimensional Banach space. 
    In contrast, it has been shown in \cite[Theorem 1.2]{ivanov2019rigidity} that, within the class of Finsler manifolds, Busemann spaces of non-positive curvature are precisely those with Berwald metrics of non-positive flag curvature. 
    In particular, all tangent cones of connected Finsler manifolds of Busemann non-positive curvature are isometric to the same normed vector space.
\end{remark}
\begin{remark}
    We remark that rectifiability as well as the a.e. unique Banach tangent cones can actually hold for more general $\mathrm{MCP}(K,N)$ spaces satisfying a.e. unique tangent cones and a non-collapsing condition.
    We refer to \cite{mondino2025rectifiability} for a detailed discussion on this topic.
\end{remark}

We conclude this subsection with a geometric characterization of tangent cones, under an additional assumption on the distance functions from a point.
\begin{corollary}\label{cor:characterize_Banach_tangent_cone}
    Let $X$ be an $n$-dimensional $S$-concave, locally semi-convex Busemann concave space.
    Suppose that for every point $x \in X$, there exists $r_x > 0$ such that for every constant-speed geodesic $\xi: [0,1] \to B(x, r_x)$ and every point $y \in B(x, r_x)$, the distance function from $y$ satisfies the $p$-uniform convexity inequality
    \begin{equation}\label{eq:characterize_Banach_tangent_cone}
        |y\xi(1/2)|^p
        \leq
        \frac{1}{2}|y\xi(0)|^p + \frac{1}{2}|y\xi(1)|^p - \frac{C'}{4}\left|\xi(0)\xi(1)\right|^p,
    \end{equation}
    for some $p \geq S+1$ and $C' = 4/2^p$.
    Then for $\mathcal{H}^n$-almost every point $x$, the unique tangent cone $(T_xX, d_x, o)$ is isometric to a Banach space $(E_x,\|\cdot\|_x, o_x)$ with a strictly convex norm $\|\cdot\|_x$, which is both $2$-uniformly smooth and $p$-uniformly convex.
\end{corollary}
\begin{proof}
    By Theorem \ref{thm:Banach_tangent_cone}, we may assume that the tangent cone $(T_xX,d_x,o)$ at $x\in X$ is isometric to a finite-dimensional Banach space $(E_x, \|\cdot\|_x, o_x)$.
    We first show that $E_x$ is uniquely geodesic.
    Let $u,v\in E_x$ be two points different from $o_x$, and $\xi:[0,1]\to E_x$ be a constant-speed geodesic from $u$ to $v$ and let $w:=\xi(1/2)$ be the midpoint of $\xi$.
    Let $a:=d_x(u,w)=\|u-v\|_x/2$.
    For simplicity, we naturally identify the points of $T_xX$ with $E_x$ by the same notation.
    From the construction of tangent cones, we can find sequences of points $u_n:=(\gamma_n,t_n), v_n:=(\eta_n,s_n), w_n:=(\xi_n,l_n)\in \hat{T}_xX$ such that $u_n\to u, v_n\to v, w_n\to w$ in $E_x$.
    For $\theta\in [0,1]$, let $x_{n,\theta}:=\xi_n(\theta l_n)$ and $c_{n,\theta}:[0,1]\to X$ be a constant-speed geodesic from $\gamma_n(\theta t_n)$ to $\eta_n(\theta s_n)$.
    Note that for $\theta\in (0,1)$ sufficiently small and $n$ sufficiently large, both $x_{n,\theta}, c_{n,\theta}$ are contained in $B(x,r_x)$.
    Applying the $p$-uniformly convex inequality \eqref{eq:characterize_Banach_tangent_cone} to the point $x_{n,\theta}$ and the geodesic $c_{n,\theta}$, and then dividing both sides by $\theta^p$, it follows that
    \begin{multline}\label{eq:characterize_Banach_tangent_cone_2}
        \left|x_{n,\theta}c_{n,\theta}(1/2)\right|^p_{\theta}
        \leq
        \frac{1}{2}\left|x_{n,\theta}\gamma_n(\theta t_n)\right|^p_{\theta} + \frac{1}{2}\left|x_{n,\theta}\eta_n(\theta s_n)\right|_{\theta}^p
        - \frac{C'}{4}\left|\gamma_n(\theta t_n)\eta_n(\theta s_n)\right|^p_{\theta},
    \end{multline}
    where $|\cdot|_{\theta}:=d/\theta$ denotes the scaled distance function.
    Since $X_{\theta}:=(X,d/\theta,x)$ converges to $(T_xX,d_x,o)$ in the pointed Gromov--Hausdorff sense, one can readily check that $\gamma_n(\theta t_n)\to u_n, \eta_n(\theta s_n)\to v_n, x_{n,\theta}\to w_n$ as $\theta\to 0$.
    Furthermore, it follows from the Arzel\`a--Ascoli theorem that the geodesic $c_{n,\theta}$ in $X_{\theta}$ uniformly converges to some constant-speed geodesic $c_n$ in $E_x$ from $u_n$ to $v_n$, up to a subsequence of $\theta$.
    By taking $\theta\to 0$ along the subsequence on both sides of the inequality above, it follows that
    \begin{equation}\label{eq:characterize_Banach_tangent_cone_3}
        \left\|w_n-c_n(1/2) \right\|^p_x
        \leq
        \frac{1}{2}\left\|w_n-u_n \right\|^p_x + \frac{1}{2}\left\|w_n-v_n \right\|^p_x- \frac{C'}{4}\left\|u_n-v_n \right\|^p_x.
    \end{equation}
    By letting $n\to \infty$ and applying the Arzel\`a--Ascoli theorem again to the geodesic $c_n$, it follows that
    \begin{equation}\label{eq:characterize_Banach_tangent_cone_4}
        \|w- c(1/2)\|^p_x
        \leq
        \frac{1}{2}\|w-u\|^p_x + \frac{1}{2}\|w-v\|^p_x -\frac{C'}{4}\|u-v\|^p_x
        \leq
        a^p- \frac{C'}{4}\left(2a\right)^p
        \leq 0,
    \end{equation}
    where $c_n$ converges to some constant-speed geodesic $c$ from $u$ to $v$, up to some subsequence of $n$.
    This implies that $w=c(1/2)$.
    Note that the proof above implies that the equality $w=c(1/2)$ holds for any accumulation point $c$ of $c_n$ and $c_{n,\theta}$, and for any choice of geodesic $c_{n,\theta}$ connecting $\gamma_n(\theta t_n)$ and $\eta_n(\theta s_n)$.
    Since $w$ is also an arbitrary midpoint of $u$ and $v$, and is independent of the choice of subsequences of $n$ and $\theta$, we conclude that $u$ and $v$ admit a unique midpoint.
    Thus, $(E_x,\|\cdot\|_x,o_x)$ is uniquely geodesic.
    By \citep[Proposition 7.2.1]{papadopoulos2014metric}, it follows that $(E_x,\|\cdot\|_x, o_x)$ is a strictly convex Banach space.

    For the $2$-uniform smoothness and $p$-uniform convexity of $E_x$, by applying the $S$-concave inequality to the point $x$ and the geodesic $c_{n,\theta}$, it follows that
    \begin{equation}
        \left|x c_{n,\theta}(1/2)\right|^2
        \geq
        \frac{1}{2}\left|xc_{n,\theta}(0)\right|^2 + \frac{1}{2}\left|xc_{n,\theta}(1)\right|^2 - \frac{S}{4}\left|c_{n,\theta}(0)c_{n,\theta}(1)\right|^2.
    \end{equation}
    By taking $\theta \to 0$ first and then $n\to \infty$, we obtain that
    \begin{equation}
        \left\|\frac{u+v}{2}\right\|_x^2=\|w\|^2_x
        \geq
        \frac{1}{2}\| u\|^2_x + \frac{1}{2}\|v\|^2_x- \frac{S}{4}\| u-v\|^2_x.
    \end{equation}
    By the same argument, it follows from the inequality \eqref{eq:characterize_Banach_tangent_cone} that
    \begin{equation}
        \left\|\frac{u+v}{2}\right\|^p_x
        \leq
        \frac{1}{2}\|u\|^p_x + \frac{1}{2}\|v\|^p_x - \frac{C'}{4}\|u-v\|^p_x.
    \end{equation}
    Therefore, our claim follows.
\end{proof}
\begin{remark}\label{rmk:charcter_Banach_tangent_cone_Kell}
    The same argument in fact shows that all tangent cones of $X$ are uniquely geodesic, and are both $2$-uniformly smooth and $p$-uniformly convex.
    Furthermore, as pointed out by Kell (see \cite[Remark after Definition 2.4]{kell2019sectional}), outside a thin subset of $X$ (i.e., a set of null measure for every doubling measure on $X$), all tangent cones are finite-dimensional Carnot groups.
    In particular, away from this thin subset, tangent cones are isometric to finite-dimensional Banach spaces with strictly convex norms; see \cite[Proposition 2.5]{kell2019sectional}.
\end{remark}

\subsection{Hausdorff dimension of singular strata}\label{sect:Hausdorff_dimension_strata}

In the previous subsections, we have shown that an $n$-dimensional $S$-concave, locally semi-convex Busemann concave space $X$ admits a natural stratification $\{X_k\}_{k=0}^{n}$, where $X_n:=\mathcal{A}(n,\delta)$ and $X_k:=\mathcal{A}(k,\delta)\setminus \mathcal{A}(k+1,\delta)$ for $k=0,\ldots,n-1$ and $\delta<\delta_n$.
In this subsection, we will study the Hausdorff dimension of the singular stratum $X_k$.
For more details on the stratification of Alexandrov spaces, see \cite[Section 10.10]{burago2001course} and \cite[Section 2]{perel1994extremal}.

As noted before, the technique developed in \cite{burago1992ad} for Alexandrov spaces cannot be directly applied here due to the absence of metric cone structure for tangent cones. 
Furthermore, the technique employed in Lemma \ref{lemma:hausdorff_measure_strained_set} is not sufficient to determine the Hausdorff dimension of singular sets, since the `almost extendable' property of geodesics holds only up to an $\mathcal{H}^n$-null set and does not provide more refined measure-theoretic information.

To deal with the difficulty, we make use of the relationship between two notions of angle introduced in Section \ref{sect:angles}, namely angles of fixed scale and angles viewed from a fixed point.
These two notions of angle are generally not equal, as illustrated in Example \ref{example:angle_asymmetry_2}.
However, when angles of common scale are small, they are almost comparable to angles viewed from a fixed point, up to some error (see Lemma \ref{lemma:variant_S_concave}).
This observation enables us to adapt the strategy from \cite[Lemma 10.5]{burago1992ad} to derive an upper bound for the maximal cardinality of $r$-separated subsets of singular sets within small cylindrical regions, by utilizing the uniform compactness of the space of directions with common length (Lemma \ref{lemma:space_direction_uniform_compact}) instead of the space of directions in Alexandrov spaces, thereby obtaining the desired dimension estimates for the singular strata.

We first introduce the following lemma, which is nearly the same as \citep[Lemma 10.2]{lytchak2019geod} and \citep[Lemma 10.3]{burago1992ad}.
\begin{lemma}\label{lemma:number_points}
    For any $N, L \geq 1$ and any natural number $M \geq 1$, there exists a constant $\bar{K} = \bar{K}(N, L, M)$ with the following property: if $X$ is an $N$-doubling metric space and $E \subset X$ is any subset containing at least $\bar{K}$ elements, then there exist $M$ points $\{x_i\}_{i=0}^{M-1} \subset E$ such that $|x_{i+1} x_0| \geq L |x_i x_0|$ for each $i = 1, \ldots, M-2$.
\end{lemma}
\begin{proof}
    Our proof is adapted from \cite[Lemma 10.2]{lytchak2019geod}.
    Let $N,L\geq 1$ be fixed.
    Let $C:=C(N,L)$ be a constant such that any set $E\subset X$ of diameter $D$ can be covered by at most $C$ subsets of diameter at most $D/(2L)$.
    
    We show the claim by induction.
    We show that the claim holds for the number $\bar{K}=C^{M-1}$.
    The case $M=1$ is trivial.
    Suppose that the claim holds for the case $M-1$.
    Let $E\subset X$ be a subset containing at least $C^{M-1}$ points.
    By removing the redundant elements, we may assume that $E$ is bounded with diameter $D>0$.
    We cover $E$ by at most $C$ subsets of diameter at most $D/(2L)$.
    Then it follows that there exists at least one subset $\tilde{E}$ in this covering that contains at least $C^{M-2}$ elements.
    By the assumption of induction, we can find $M-1$ points $\{x_i\}_{i=0}^{M-2}\subset \tilde{E}$ such that $|x_0x_{i+1}|\geq L|x_0x_{i}|$ for $i=1,\cdots, M-3$.
    Now, pick a point $x \in E$ such that $|x x_0| \geq D/2$.
    This is always possible; otherwise, the diameter of $E$ would be less than $D$.
    Let $x_{M-1}:=x$.
    The collection $\{x_i\}_{i=0}^{M-1}$ satisfies the required property because
    \begin{equation}
        |x_0x_{M-1}|\geq \frac{D}{2}\geq L\frac{D}{2L}\geq L|x_0x_i|,
    \end{equation}
    for all $i=1,\cdots,M-2$.
\end{proof}

We next present the second key lemma, which quantifies how the ratio of side-lengths controls the discrepancy between angles of common scale and angles viewed from a fixed point, when the angle of common scale is small.
In the following, for a given $\delta\in (0, 1)$, we define $\beta(\delta)$ as
\begin{equation}\label{def:beta_delta_num}
    \beta(\delta):=\frac{(1-\cos\delta)\sin\delta}{2(1+\sin\delta)}.
\end{equation}
Notice that $\beta(\delta)< 1/4$ and $\beta(\delta)=O(\delta^3)$ as $\delta\to 0$.
\begin{lemma}\label{lemma:variant_S_concave}
    Let $X$ be an $S$-concave Busemann concave space, and $0<\delta <1/2$.
    There exists a constant $L_0 = L_0(\delta, S) > 1$ such that the following holds: let $x,y,z\in X$ be any points such that $|xz|/|xy| \geq L_0$, and let $\gamma, \eta$ be unit-speed geodesics from $x$ to $y$ and $z$ respectively. 
    If the angle $\angle_x(\gamma(r), \eta(r)) < \beta(\delta)$ for some $r > 0$, then we have $\tilde{\angle_x}(y,z)< \delta$.
    Moreover, $z$ is a $(1, 2\delta)$-strainer at $y$ with $x$ as an opposite strainer.
\end{lemma}
To prove this lemma, we need the following technical lemma from \cite[Lemma 6]{lebedeva2021self}.
\begin{lemma}[Lebedeva--Ohta--Zolotov, \cite{lebedeva2021self}]\label{lemma:Lebedeva_Ohta_Zolotov}
    Let $0<\delta<\pi/2$ and $(X,d)$ be a metric space, and take $x\in X$ and $y,z\in X\setminus\{x\}$.
    If there exists $z'\in X$ such that $|xz|=|xz'|+|z'z|$ and $|yz'|\leq \beta(\delta)|xy|$, then we have $\tilde{\angle}_x(y,z)<\delta$.
\end{lemma}
We are now in a position to prove Lemma \ref{lemma:variant_S_concave}.
\begin{proof}[Proof of Lemma \ref{lemma:variant_S_concave}]
    Let $0<\delta<1/2$ be fixed.
    We choose $L_0:=L_{0}(\delta, S)>1$ large enough such that $\arccos(1-S/(2(L_0-1)))< \delta$ and $\arcsin(1/L_0)< \delta$.
    We claim that $L_0$ satisfies the requirement of lemma.
    Indeed, let $x\in X$ and $\gamma,\eta$ be any geodesics from $x$ to $y, z$ respectively such that $|xz|/|xy|\geq L_0$ and the angle of common scale $\angle_x(\gamma(r),\eta(r))<\beta(\delta)$ for some $r>0$.
    By the positive scaling-invariance of angles of fixed scale, Lemma \ref{lemma:relation_metric_and_angle}, we may assume that $r=|xy|$.
    Let $z':=\eta(r)$.
    From the definition of angles of fixed scale, it follows that $\tilde{\angle}_x(\gamma(r),\eta(r))=\tilde{\angle}_x(y,z')<\beta(\delta)$.
    From elementary plane geometry theory, we can see that
    \begin{equation}
        |yz'|\leq 2\sin \left(\frac{\beta(\delta)}{2}\right)|xy|\leq \beta(\delta)|xy|.
    \end{equation}
    Since $|xz|=|xz'|+|z'z|$ by our choice of $z'$, it follows from Lemma \ref{lemma:Lebedeva_Ohta_Zolotov} that $\tilde{\angle}_x(y,z)<\delta$.
    This proves the first part of lemma.
    As for the second part of lemma, by our choice of $L_0$, it follows that
    \begin{equation}
        \tilde{\angle}_z(x,y)\leq \arcsin\left(\frac{|xy|}{|xz|}\right)\leq \arcsin\left(\frac{1}{L_0}\right)<\delta.
    \end{equation}
    This, together with $\tilde{\angle}_x(y,z)<\delta$, implies that $\tilde{\angle}_y(x,z)> \pi- 2\delta$.
    Furthermore, our choice of $L_0$ also implies that
    \begin{equation}
        \bar{\delta}_S(|xy|;|yz|)
        =
        \arccos\left(1-S\frac{|xy|}{2|yz|}\right)
        \leq
        \arccos\left(1-\frac{S}{2(L_0-1)}\right)
        <\delta.
    \end{equation}
    According to Definition \ref{def:1-strainer}, $z$ is a $(1,2\delta)$-strainer at $y$ with $x$ as an opposite strainer.
\end{proof}

The next lemma provides an estimate for the maximal possible cardinality of a $\delta^{-1}r$-separated subset within a cylindrical region of a small ball.
Before stating the lemma, we recall a notion called \emph{R-long} strainers introduced in \cite{burago1992ad}.
\begin{definition}[$R$-long strainers]
    For $R>0$, a $(k,\delta)$-strainer $(p_1,\ldots,p_k)$ at a point $x$ with the opposite strainer $(q_1,\ldots,q_k)$ is said to be \emph{$R$-long} if $\min_{i=1,\ldots,k}\{|p_ix|,|q_ix|\}> \delta^{-1}R$.
    We denote by $\mathcal{A}(k,\delta,R)\subset \mathcal{A}(k,\delta)$ the set of $(k,\delta)$-strained points that admit an $R$-long $(k,\delta)$-strainer.
\end{definition}

We are now in a position to state the lemma.
Our proof follows a strategy similar to \cite[Lemma 10.5]{burago1992ad}, based on the uniform compactness of the space of directions with common length (Lemma \ref{lemma:space_direction_uniform_compact}) and Lemma \ref{lemma:variant_S_concave}. 
Recall that $\beta_E(r)$ denotes the largest possible cardinality of any maximal $r$-separated subset of $E$.
\begin{lemma}\label{lemma:hausdorff_dim_est}
    Let $X$ be an $n$-dimensional $S$-concave, $(C,D)$-semi-convex Busemann concave space, and let $\delta<\delta_k$ for $k\in \{1,\ldots,n\}$.
    Suppose that $(p_1,\ldots,p_k)$ is an $R$-long strainer on a neighborhood $U$ of $x\in X$.
    Then there exists $r_x>0$ such that $B(x,r_x)\subset U$, and for any $0<r\leq \delta r_x$ and any $m=(m_1,\ldots,m_k)\in \mathbb{Z}^k$, it holds
    \begin{equation}
        \beta_{D_x(r,m)\setminus \mathcal{A}(k+1,2\delta,r)}(\delta^{-1}r)
        \leq
        K(N,\delta,S),
    \end{equation}
    where $D_x(r,m):=\{z\in B(x,r_x): 0.1(m_i-1)r\leq |p_ix|-|p_iz|\leq 0.1m_ir, i=1,\cdots,k\}$ is a cylindrical region in $B(x,r_x)$, and $K(N,\delta,S)$ is a constant depending only on the doubling constant $N$ of $X$, $\delta$, and the uniform smoothness constant $S$.
\end{lemma}
\begin{proof}
    Let $0<\delta<\delta_k$ and $R>0$ be fixed. 
    We choose $r_x$ as follows.
    Choose $r_0>0$ small enough such that $B(x, r_0)\subset U$, $r_0\leq R/2$, and 
    \begin{equation}\label{eq:prop_H_dim_2}
        \arccos\left(1-\frac{\left(S+C\right)r_0}{2(\delta^{-1}R-r_0)}\right)<\delta.
    \end{equation}
    Let $r_x:=r_0/2\leq R/4$.
    Then $\bar{\delta}_{S,C}(|yz|;|p_iy|)<\delta$ for any $y,z\in B(x,r_x)$ and $i=1,\ldots,k$. 
    
    Fix $r\in (0,\delta r_x]$ and $m\in \mathbb{Z}^k$.
    We prove by contradiction that the largest possible cardinality of any $\delta^{-1}r$-separated subset of $D_x(r,m)\setminus \mathcal{A}(k+1, 2\delta, r)$ is bounded above by $K(N,\delta,S):=\bar{K}(N,L_0, M)-1$, where $\bar{K}$ is the constant in Lemma \ref{lemma:number_points}, $N$ is the doubling constant of $X$, $L_0:=L_0(\delta,S)$ is the constant in Lemma \ref{lemma:variant_S_concave}, and $M:=M(\delta)=N_0(\beta(\delta))+2$, where $N_0$ is the constant in Lemma \ref{lemma:space_direction_uniform_compact} and $\beta(\delta)$ is defined in \eqref{def:beta_delta_num}.

    Suppose that our assertion does not hold.
    Then we can find a $\delta^{-1}r$-separated subset of $D_x(r,m)\setminus \mathcal{A}(k+1,2\delta,r)$ containing at least $\bar{K}(N,L_0,M)$ points.
    By Lemma \ref{lemma:number_points}, there are $M$ points $\{x_i\}_{i=0}^{M-1}\subset D_x(r, m)\setminus \mathcal{A}(k+1,2\delta,r)$ such that $|x_0x_{i+1}|\geq L_0|x_0x_i|$ for all $i=1,\ldots, M-2$.
    Let $\bar{l}:=\min_{i=1,\ldots,M-1}\{|x_0x_i|\}$ and let each $\xi_i$ be a unit-speed geodesics from $x_0$ to $x_i$, and let $\bar{\xi}_i$ be the maximal extension of the geodesic $\xi_i$ for all $i=1,\ldots,M-1$.
    By Lemma \ref{lemma:space_direction_uniform_compact} and our choice of $M=N_0(\beta(\delta))+2$, we can find two indices $1\leq i<j\leq M-1$ such that $|x_0x_j|\geq L_0|x_0x_i|$ and
    \begin{equation}\label{eq:prop_H_dim_6}
        \angle_{x_0}\left((\bar{\xi}_i, \bar{l}), (\bar{\xi}_j, \bar{l})\right)
        =
        \angle_{x_0}\left(\xi_i(\bar{l}),\xi_j(\bar{l})\right)
        <
        \beta(\delta).
    \end{equation}

    In the following, we show that the $(k+1)$-tuple $(p_1,\ldots,p_k,x_j)$ is an $r$-long $(k+1, 2\delta)$-strainer at $x_i$.
    This contradicts our choice of $x_i$ that $x_i\notin \mathcal{A}(k+1, 2\delta, r)$.

    \begin{enumerate}[label=\textbf{\textsc{Step \arabic*:}}, fullwidth]
        \item In this step, we check the first two conditions in Definition \ref{def:k_strainer} for the $(k+1)$-tuple $(p_1,\ldots,p_k,x_j)$ at $x_i$.
        Note that $(p_1,\ldots,p_k)$ is an $r$-long $(k,\delta)$-strainer at $x_i$, since $(p_1,\ldots,p_k)$ is an $R$-long strainer on $U$ by assumption and $r<R$.
        Furthermore, since $\{x_i\}_{i=1}^{M-1}$ is $\delta^{-1}r$-separated, it follows that $|x_0x_i|$ and $|x_0x_j|$ are larger than $\delta^{-1}r$.
        Since $|x_0x_j|\geq L_0|x_0x_i|$, by Lemma \ref{lemma:variant_S_concave} together with the inequality \eqref{eq:prop_H_dim_6}, it follows that $x_j$ is an $r$-long $(1,2\delta)$-strainer at $x_i$ with the opposite strainer $x_0$.
        Furthermore, by our choice of $r_0$, it follows that $\bar{\delta}_{S,C}(|x_jx_i|;|p_lx_i|) < \delta$ for all $l = 1,\ldots, k$.
        Thus, the first two conditions in Definition \ref{def:k_strainer} are satisfied for the $(k+1)$-tuple $(p_1,\ldots,p_k,x_j)$ at $x_i$.
        
        \item We are left to verify the almost orthogonality condition in Definition \ref{def:k_strainer} for $\tilde{\angle}p_l x_i x_j$ and $\tilde{\angle}p_lx_i x_0$ for any $l\in \{1,\ldots,k\}$.
        For simplicity, we denote $p := p_l$.

        We first establish the almost orthogonality for the comparison angles $\tilde{\angle} p x_i x_j$.
        Indeed, by applying the Euclidean law of cosines to the comparison triangle $\tilde{\Delta}px_ix_j$, it follows that
        \begin{align}\label{eq:prop_H_dim_1}
            \cos\tilde{\angle}px_ix_j
            &=
            \frac{|px_i|^2+|x_ix_j|^2-|px_j|^2}{2|px_i||x_ix_j|}\nonumber\\
            &=
            \frac{\left(|px_i|+|px_j|\right)\left(|px_i|-|px_j|\right)}{2|px_i||x_ix_j|}
            +
            \frac{|x_ix_j|}{2|px_i|}.
        \end{align}
        For the first term on the right-hand side of \eqref{eq:prop_H_dim_1}, since $x_i,x_j\in D_x(r,m)$, it follows that $\big||px_i|-|px_j|\big|<0.1r\leq 0.1\delta R\leq 0.1\delta^2|px_i|$, where we use the fact that $(p_1,\ldots,p_k)$ is $R$-long on $U$ in the last inequality.
        Since $|x_ix_j|\geq \delta^{-1}r$, it follows that
        \begin{equation}\label{eq:eq:prop_H_dim_5}
            \left|\frac{\left(|px_i|+|px_j|\right)\left(|px_i|-|px_j|\right)}{2|px_i||x_ix_j|}\right|
            \leq
            \frac{(2+0.1\delta^2)}{2}\frac{0.1r}{\delta^{-1}r}
            <
            \frac{1}{4}\delta.
        \end{equation}
       For the second term on the right-hand side of \eqref{eq:prop_H_dim_1}, by our choice of $r_0$ that $|x_ix_j|<2r_x=r_0<R/2$, it follows that
       \begin{equation}
        \frac{|x_ix_j|}{2|px_i|}
        \leq
        \frac{r_0}{2\delta^{-1}R}
        <
        \frac{1}{4}\delta.
       \end{equation}
       In conclusion, we have $|\cos\tilde{\angle}px_ix_j|\leq \delta/2$.
       This implies that $|\tilde{\angle}px_ix_j-\pi/2|<\delta$.
       Similarly, we can prove $|\tilde{\angle}px_i x_0 -\pi/2|<\delta$.
    \end{enumerate}

    In conclusion, $(p_1,\cdots,p_k,x_j)$ is an $r$-long $(k+1, 2\delta)$-strainer at $x_i$.
    This contradicts the fact that $x_i\in D_x(r,m)\setminus \mathcal{A}(k+1,2\delta,r)$ and we complete the proof.
\end{proof}
\begin{remark}
    From our proof of Lemma \ref{lemma:hausdorff_dim_est}, one may see that locally doubling $S$-concave Busemann concave spaces satisfy a strengthened doubling condition, called the \emph{$\mathrm{ATB}(\delta)$}-condition.
    For more details on this condition, we refer to \cite{lebedeva2021self}.
\end{remark}

We are now ready to prove the last theorem. 
This theorem provides a Hausdorff dimension estimate for the singular strata, which is sharper than Theorem \ref{thm:full_measure_n_strained_points}.
A similar result for Alexandrov spaces was established in \citep[Theorem 10.6]{burago1992ad}.

\begin{theorem}\label{thm:main_thm_Hausdorff_dim_est}
    Let $X$ be an $n$-dimensional $S$-concave, locally semi-convex Busemann concave space.
    Then $\mathrm{dim}_H(X \setminus \mathcal{A}(k, \delta)) \leq k - 1$ for any $\delta>0$ and $k=1,\ldots, n$.
\end{theorem}
\begin{proof}
    It suffices to show the case $\delta<\delta_n$, and we prove the assertion by induction.
    Let $\delta'_k:=\delta/2^{n-k}<\delta_k$ for $k=1,\ldots,n-1$.
    We first show that $X\setminus \mathcal{A}(1,\delta'_1)$ is of Hausdorff dimension zero.
    In fact, we show that $X\setminus \mathcal{A}(1,\delta'_1)$ is discrete, and its cardinality is bounded above by $K(N, \delta'_1/2,S)$, where $K$ is the constant in Lemma \ref{lemma:hausdorff_dim_est}.
    Indeed, suppose that $X\setminus \mathcal{A}(1,\delta'_1)$ contains at least $K(N, \delta'_1/2,S)+1=\bar{K}(N,L_0,M)$ points, where $\bar{K}$ is the same as the one in Lemma \ref{lemma:hausdorff_dim_est} with $\delta$ replaced by $\delta'_1/2$.
    Then following the same argument as the ones preceding Step 1 of the proof of Lemma \ref{lemma:hausdorff_dim_est} together with Lemma \ref{lemma:variant_S_concave}, we can find three points $x_0,x_i,x_j\in X\setminus \mathcal{A}(1,\delta'_1)$ such that $x_j$ is a $(1, \delta'_1)$-strainer at $x_i$ with the opposite strainer $x_0$.
    This contradicts the fact that $x_i \notin \mathcal{A}(1,\delta'_1)$.
    Hence, we conclude that the cardinality of $X\setminus \mathcal{A}(1,\delta'_1)$ is bounded above by $K(N,\delta'_1/2,S)$, and therefore $\mathrm{dim}_H(X\setminus \mathcal{A}(1,\delta'_1))=0$.
    Now suppose that $\mathrm{dim}_H(X\setminus \mathcal{A}(k,\delta'_{k}))\leq k-1$ for some $k\in \{1,\ldots,n-1\}$.

    For any $x\in \mathcal{A}(k,\delta'_{k})$, we can find $R_x>0$ such that $x$ admits a $4R_x$-long $(k,\delta'_{k})$-strainer $(p_1,\ldots,p_k)$.
    By Lemma \ref{lemma:strained_point_openness}, we can find a small neighborhood $U$ of $x$, such that $(p_1,\ldots,p_k)$ is an $R_x$-long $(k,\delta'_{k})$-strainer on $U$. 
    Let $r_x>0$ be the number given in Lemma \ref{lemma:hausdorff_dim_est}.
    
    We claim that the Hausdorff dimension of $B(x, r_x) \setminus \mathcal{A}(k+1, \delta'_{k+1})$ is at most $k$.
    Indeed, it is straightforward to verify that the ball $B(x, r_x)$ can be covered by the cylindrical regions $\{D_x(r, m)\}_{m \in \mathbb{Z}^k}$, with the number of non-empty elements in this covering bounded above by $(2r_x/(0.1 r))^k$.
    Combining with Lemma \ref{lemma:hausdorff_dim_est}, we have
    \begin{multline}\label{eq:main_thm_Hausdorff_dim_est}
        \beta_{B(x,r_x)\setminus \mathcal{A}(k+1,\delta'_{k+1})}\left((\delta'_k)^{-1}r \right)
        \leq
        \sum_{\substack{m\in \mathbb{Z}^k\\ D_x(r,m)\neq \emptyset}}\beta_{D_x(r,m)\setminus \mathcal{A}(k+1,\delta'_{k+1})}\left((\delta'_k)^{-1}r\right)\\
        \leq
        \sum_{\substack{m\in \mathbb{Z}^k\\ D_x(r,m)\neq \emptyset}}\beta_{D_x(r,m)\setminus \mathcal{A}(k+1,\delta'_{k+1},r)}\left((\delta'_k)^{-1}r\right)
        \leq
        \left(\frac{2r_x}{0.1r}\right)^k K(N,\delta'_k,S).
    \end{multline}
    Thus, for any $\varepsilon>0$, it follows from the inequality \eqref{eq:main_thm_Hausdorff_dim_est} that
    \begin{equation}
        \limsup_{r\to 0}\left((\delta'_k)^{-1}r\right)^{k+\varepsilon}\beta_{B(x,r_x)\setminus \mathcal{A}(k+1,\delta'_{k+1})}\left((\delta'_k)^{-1}r\right)
        =
        0.
    \end{equation}
    This implies that the rough dimension (see Section \ref{subsect:H_measure_dimensions_rect}) of $B(x,r_x)\setminus \mathcal{A}(k+1,\delta'_{k+1})$ is at most $k$.
    Since the Hausdorff dimension is not greater than the rough dimension, we obtain that $\mathrm{dim}_H(B(x,r_x)\setminus \mathcal{A}(k+1,\delta'_{k+1}))\leq k$.

    Finally, by covering $\mathcal{A}(k,\delta'_k)$ by at most countably many balls $B(x,r_x)$, it follows that
    \begin{equation}
        \mathrm{dim}_H(\mathcal{A}(k,\delta'_k)\setminus \mathcal{A}(k+1,\delta'_{k+1}))\leq k.
    \end{equation}
    By the induction hypothesis, it follows that
    \begin{multline}
        \mathrm{dim}_H\left(X\setminus \mathcal{A}(k+1,\delta'_{k+1})\right)\\
        =
        \max\left\{\mathrm{dim}_H\left(\mathcal{A}(k,\delta'_k)\setminus \mathcal{A}(k+1,\delta'_{k+1})\right), \mathrm{dim}_H\left(X\setminus \mathcal{A}(k,\delta'_{k})\right)\right\}
        \leq
        k.
    \end{multline}
    By induction, we complete the proof.
\end{proof}

As a corollary, we get:
\begin{corollary}\label{cor:stratification}
    Let $X$ be an $n$-dimensional $S$-concave, locally semi-convex Busemann concave space.
    Then $X$ admits a stratification $\{X_k\}_{k=0}^n$ such that $X$ is the disjoint union of $\{X_k\}_{k=0}^n$, with $\mathrm{dim}_H(X_k)\leq k$ for $k=0,\ldots,n$, and the top-dimensional stratum $X_n$ is a topological $n$-manifold.
\end{corollary}
\begin{proof}
    Let $\delta<\delta_{n+1}$.
    Observe that $\mathcal{A}(n+1,\delta)$ is empty.
    Otherwise, by the self-improvement property of strainers (Lemma \ref{lemma:self_improvement_strainer}), $X$ would contain $(n+1,\delta')$-strained points for any $\delta'>0$. 
    This would imply that the strainer number of $X$ is greater than $n$, contradicting the assumption that $X$ is $n$-dimensional.
    Therefore, it is clear that $X$ can be decomposed into the disjoint union of sets $X_0:=X\setminus \mathcal{A}(1,\delta)$, $X_k:=\mathcal{A}(k,\delta)\setminus \mathcal{A}(k+1,\delta)$ for $k=1,\ldots,n-1$, and $X_n:=\mathcal{A}(n,\delta)$. 
    Now the claim follows from Theorem \ref{thm:main_thm_Hausdorff_dim_est} that $\mathrm{dim}_H(X_k)\leq \mathrm{dim}_H(X\setminus \mathcal{A}(k+1,\delta))\leq k$ for $k=0,\ldots,n-1$, and from Theorem \ref{thm:full_measure_n_strained_points} that $X_n$ is a topological $n$-manifold.
\end{proof}


\begin{appendix}
\section{Criterion for open maps}\label{appendix}
\noindent
In the appendix, we provide a proof of the criterion for $\varepsilon$-open maps, Lemma \ref{lemma:criterion_open_map}.
Our proof follows an argument similar to that in \cite[Lemma 8.1]{lytchak2019geod}, based on a lemma of Lytchak \cite[Lemma 4.1]{lytchak2006open}.
\begin{lemma}
    Let $f:X\rightarrow Y$ be a locally Lipschitz map from a locally complete metric space $X$ to a geodesic space $Y$.
    Suppose there exists an $\varepsilon>0$ such that the following holds: for every $x\in X$ and $v\in Y\setminus \{f(x)\}$ sufficiently close to $f(x)$, there exists $y\in X$ such that
    \begin{equation}\label{lemma:Append_1}
        |f(y)v|-|f(x)v| \leq -\varepsilon|xy|.
    \end{equation}
    Then $f$ is an $\varepsilon'$-open map for any $0<\varepsilon'<\varepsilon$.
\end{lemma}
\begin{proof}
    Let $x\in X$ and $r>0$ be such that $\bar{B}(x,\varepsilon^{-1}r)$ is complete.

    We first show that $B(f(x),s)\subset f(\bar{B}(x,\varepsilon^{-1}s))$ for any $0<s\leq r$.
    Let $v\in B(f(x),s)\setminus\{f(x)\}$ and $l_v:=|f(x)v|<s$.
    Let $h:X\to \mathbb{R}$ be the function defined as $h(z):=l_v-|f(z)v|$.
    Note that $h(x)=0$. 
    In order to find $y\in \bar{B}(x,\varepsilon^{-1}s)$ such that $f(y)=v$, it suffices to find $y$ such that $h(y)=l_v$.
    To apply \cite[Lemma 4.1]{lytchak2006open}, we need to show that $\limsup_{z'\to z}\frac{h(z')-h(z)}{|z'z|}\geq \varepsilon$ for any $z\in \bar{B}(x,\varepsilon^{-1}s)$.
    Let $z\in \bar{B}(x,\varepsilon^{-1}s)$. 
    Since $Y$ is geodesic, there exists a unit-speed geodesic $\gamma$ connecting $f(z)$ and $v$.
    For a sequence of points $w_n\in \gamma$ converging to $f(z)$, it follows from the assumption \eqref{lemma:Append_1} that there exists a sequence $z'_n\in X$ such that
    \begin{equation}
        |f(z'_n)w_n|-|f(z)w_n|\leq -\varepsilon|z'_nz|.
    \end{equation}
    It follows that $|z'_nz|\leq \varepsilon^{-1}|f(z)w_n|\to 0$.
    Furthermore, since $w_n$ lies in a geodesic connecting $f(z)$ and $v$, it follows that
    \begin{multline}
        |f(z'_n)v|-|f(z)v|
        =
        |f(z'_n)v| - |f(z)w_n| - |w_n v|\\
        \leq
        |f(z'_n)w_n| - |f(z)w_n|
        \leq
        -\varepsilon|z'_nz|.
    \end{multline}
    This implies that we can find a sequence of points $z'_n$ converging to $z$ such that $h(z'_n)-h(z)\geq \varepsilon |z'_nz|$.
    Thus, it follows that
    \begin{equation}
        \left|\nabla_z h\right|
        :=
        \limsup_{z'\to z}\frac{h(z')-h(z)}{|z'z|}
        \geq
        \limsup_{n\to \infty}\frac{h(z'_n)-h(z)}{|z'_nz|}
        \geq
        \varepsilon.
    \end{equation}
    Therefore, we can apply \cite[Lemma 4.1]{lytchak2006open} to the function $h$.
    Let $\varepsilon':=\varepsilon l_v/s<\varepsilon$.
    Since $\bar{B}(x, \varepsilon^{-1}s)\subset \bar{B}(x, \varepsilon^{-1}r)$ is complete, it follows from \cite[Lemma 4.1]{lytchak2006open} that there exists $z\in \bar{B}(x,\varepsilon^{-1}s)$ such that $h(z)=\varepsilon'\varepsilon^{-1}s=l_v$.
    This implies that $B(f(x),s)\subset f(\bar{B}(x,\varepsilon^{-1}s))$.

    We now show that $f$ is $\varepsilon'$-open for any $0<\varepsilon'<\varepsilon$ in the sense of Definition \ref{def:epsilon_open_map}.
    Let $0<\varepsilon'<\varepsilon$ be fixed.
    Let $0<r'<r$ be such that $\bar{B}(x, (\varepsilon')^{-1}r')\subset \bar{B}(x,\varepsilon^{-1}r)$ is complete.
    Let $v\in B(f(x),r')$ and $s:=|f(x)v|<r'$.
    For any $n\in \mathbb{N}$ satisfying $s+1/n<r'$, it holds that $v\in B(f(x), s+1/n)\subset f(\bar{B}(x, \varepsilon^{-1}(s+1/n)))$.
    Therefore, we can find a sequence of points $z_n$ such that $v=f(z_n)$ and
    \begin{equation}
        |xz_n|\leq \varepsilon^{-1}\left(|f(x)v|+\frac{1}{n}\right).
    \end{equation}
    By choosing $n\in \mathbb{N}$ sufficiently large, it holds that
    \begin{equation}
        |xz_n|\leq \varepsilon^{-1}\left(|f(x)v|+\frac{1}{n}\right)
        \leq
        \left(\varepsilon'\right)^{-1}|f(x)v|.
    \end{equation}
    Thus, we have shown that for any $x\in X$, there exists $r'>0$ such that $\bar{B}(x,(\varepsilon')^{-1}r')$ is complete, and for any $v\in B(f(x),r')$, we can find a $z\in X$ such that $f(z)=v$ and 
    \begin{equation}
        \varepsilon' |xz|\leq |f(x)v|.
    \end{equation}
    This implies that $f$ is an $\varepsilon'$-open map.
    
\end{proof}
\end{appendix}

\section*{Declaration}
\noindent
{The authors declare that there is no conflict of interest, and this paper has no associated data.}

\section*{Acknowledgements}
\noindent
The authors would like to express their gratitude to Tadashi Fujioka for bringing to their attention the reference \cite{lebedeva2021self} and fruitful discussions. 
His suggestions greatly improve the results of Lemma \ref{lemma:variant_S_concave} and Theorem \ref{thm:main_thm_Hausdorff_dim_est}.
The authors would also like to thank Professors Alexander Lytchak, Shijie Gu and Shin-ichi Ohta for their constructive comments on this manuscript, and productive discussions on several potential open problems of Busemann concave spaces.
The authors are also grateful to the anonymous referees for their careful reading and constructive comments, which significantly improved the quality of this paper.

\section*{Grant/Funding}
\noindent
This work is supported in part by Young Scientist Programs of the Ministry of Science and Technology of China (2021YFA1000900, 2021YFA1002200),  National Natural Science Foundation of China (12201596), Shandong Provincial Natural Science Foundation (ZR2025QB05),  the Taishan Scholars Program of Shandong Province (tsqn202408059).

\bibliographystyle{plainnat}
\bibliography{biblio_revision}

\bigskip
\begin{flushleft}
\small \normalfont
\textsc{Bang-Xian Han\\
School of Mathematics, Shandong University, Jinan, 250100, China}\\
\texttt{\textbf{hanbx@sdu.edu.cn}}
\end{flushleft}

\medskip
\begin{flushleft}
\small \normalfont
\textsc{Liming Yin\\
School of Mathematical Sciences, University of Science and Technology of China, Hefei, 230026, China}\\
\texttt{\textbf{yinliming@ustc.edu.cn}}
\end{flushleft}

\end{document}